\documentclass{article}

\usepackage{graphicx}
\usepackage{tikz}
\usepackage{amsmath}
\usepackage{amsthm}
\usepackage{amssymb}
\usepackage{amsfonts}
\usepackage{float} 
\usepackage{tabularx,booktabs}
\usepackage{array}
\usepackage{diagbox}
\usepackage{multirow}
\usepackage[T1]{fontenc}
\usepackage{fix-cm}
\usepackage{mathrsfs}

\usepackage{booktabs}
\usepackage{tabularx}
\usepackage{array}
\usepackage[section]{placeins}
\usepackage[abbrev,nobysame]{amsrefs}

\usepackage[a4paper, left=2cm, right=2cm, top=3cm, bottom=3cm]{geometry}
\usepackage{hyperref}

\newenvironment{restated}[1]{%
  \par\noindent\textbf{Theorem #1.}\itshape\ }%
  {\par}  

\theoremstyle{plain}
\theoremstyle{definition}
\newtheorem{theorem}{Theorem}[section]
\newtheorem{lemma}[theorem]{Lemma}
\newtheorem{definition}[theorem]{Definition}
\newtheorem{example}[theorem]{Example}
\newtheorem{remark}[theorem]{Remark}
\newtheorem{corollary}[theorem]{Corollary}

\newtheorem{proposition}[theorem]{Proposition}
\newtheorem{conjecture}[theorem]{Conjecture}
\newtheorem{assumption}[theorem]{Assumption}

\newcommand{\dist}{\mathrm{d}}
\newcommand{\diam}{\mathrm{diam}}
\newcommand{\card}{\mathrm{card}}

\newcommand{\Reff}{\mathscr{R}}

\newcommand{\dimbox}{\dim_\mathrm{B}}
\newcommand{\dimdeg}{\dim_\mathrm{D}}
\newcommand{\dimhaus}{\dim_\mathrm{H}}

\newcommand{\dimwalk}{\dim_\mathrm{W}}
\newcommand{\dimresis}{\dim_\mathrm{R}}
\newcommand{\dimspec}{\dim_\mathrm{S}}

\begin{document}

\title{Iterated Graph Systems (I): random walks and diffusion limits}

\author{Ziyu Neroli \\
\small{Department of Mathematics, Imperial College London, United Kingdom}
}
\date{}

\maketitle

\begin{abstract}
This paper investigates random walks and diffusion limits on a broad class of fractals generated by Edge Iterated Graph Systems (EIGS), a generalisation of hierarchical lattices.
We prove that the rescaled simple random walks converge in the Gromov--Hausdorff--Prokhorov--Skorokhod topology to the limiting diffusion, which coincides with Brownian motion when the resistance dimension is positive.
The graph analysis underlying this convergence identifies the degree dimension as the natural correction term for on-diagonal heat-kernel estimates, yielding a unified formulation in the locally finite and locally infinite (scale-free) regimes.
Using this framework, we solve the open problem on the diamond hierarchical lattice (DHL) percolation cluster posed by Hambly and Kumagai [Commun. Math. Phys. 295 (2010), 29--69]; this suggests that the Alexander--Orbach-type conjecture fails for the natural Brownian motion on the cluster.
Sections~2 to~4 have been formalised in Lean.

\end{abstract}



MSC2020 subject classifications. 60J65, 05C81, 28A80.

\tableofcontents

\section{Introduction}\label{sec:intro}

\begin{figure}
    \centering
    \includegraphics[width=\linewidth]{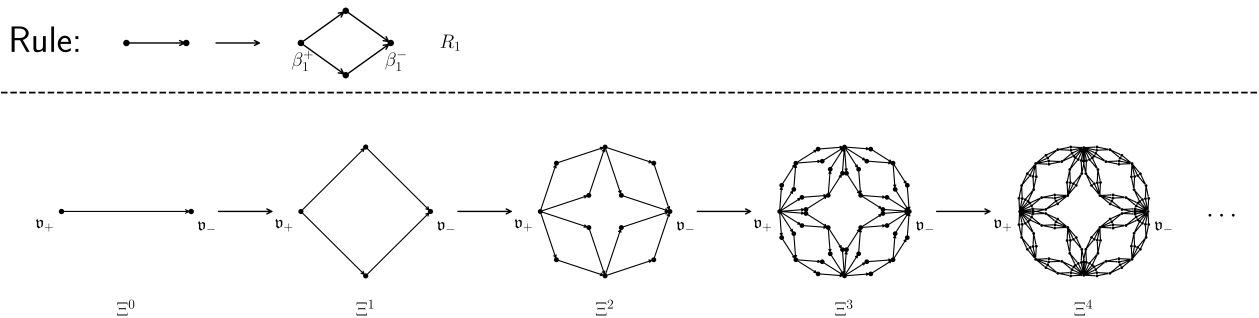}
    \caption{Example of EIGS: diamond hierarchical lattice (DHL)}
    \label{fig:DHL}
\end{figure}

In 1982, Alexander and Orbach conjectured that the simple random walk on the incipient infinite cluster of critical Bernoulli percolation on $\mathbb Z^d$ has spectral dimension $\frac43$ in every dimension $d\ge 2$ \cite{alexander1982density}.
Although Kozma--Nachmias and Fitzner--van der Hofstad proved the conjecture for nearest-neighbour percolation in dimension $d\ge 11$ \cites{kozma2009alexander,fitzner2017meanfield}, it remains open to this day.

Hambly and Kumagai \cite{hambly2010diffusion} looked at the question from the opposite direction: on what kind of spaces does the conjecture fail, and by how much?
A natural testing ground is a model on which critical percolation renormalises exactly, and they therefore turned to the diamond hierarchical lattice (DHL), an exactly renormalisable model rooted in the statistical physics of the 1980s, see Figure~\ref{fig:DHL}.
At that time, however, heat-kernel estimates for the percolation cluster on the DHL were far out of reach, since not even the diffusion on the deterministic DHL had been constructed.
They therefore began with the DHL itself, and obtained on-diagonal heat-kernel estimates for its canonical diffusion and for the scaling limit of the corresponding critical bond-percolation cluster.
More precisely, there exist $\varepsilon>0$, constants $c_1,\dots,c_8>0$, and a positive function $T$ such that
\begin{align*}
c_1 t^{-1} |\log t|^{-2-\varepsilon}\le &\, p_t(x,x)\le c_2 t^{-1},
\qquad \text{for a.e. }x\in \mathrm{DHL}, \quad \forall t<T(x), \\
c_3 t^{-1/2}\le &\, p_t(0,0)\le c_4 t^{-1/2},
\qquad \forall t<1.
\end{align*}
For the critical percolation cluster on the DHL, they proved that its scaling limit is a graph-directed random recursive fractal and that, in the effective resistance metric with $\theta\approx 1.8993$, the on-diagonal heat-kernel bounds are
\begin{align*}
c_5 t^{-\theta/(\theta+1)} |\log t|^{-2(2\theta+3)(\theta+2)-\varepsilon}
\le &\, q_t^\omega(x,x)\le
c_6 t^{-\theta/(\theta+1)} \bigl|\log |\log t| \bigr|^{(\theta-1)/(\theta+1)}, \\
&\mu^\omega\text{-a.e.\ }x\in C(\omega),\ \mathbb P\text{-a.s.},\ \forall t<1, \\
c_7 t^{-\theta^*/(\theta^*+1)}\le &\, q_t^\omega(0,0)\le c_8 t^{-\theta^*/(\theta^*+1)},
\qquad
\mathbb P\text{-a.e.\ }\omega,\ \forall t<1.
\end{align*}
Here $p_t$ and $q_t^\omega$ denote the on-diagonal heat kernels of the diffusions on the DHL and on the cluster, $C(\omega)$ is the cluster with its natural measure $\mu^\omega$, $\theta\approx 1.8993$ is the almost sure growth exponent of the cluster mass, and $\theta^*$ is the corresponding exponent at the distinguished boundary point $0$.

Owing to technical limitations, the exponent $\theta$ was obtained under a special resistance normalisation, but this did not prevent them from showing that the Alexander--Orbach-type conjecture fails for the critical percolation cluster on the DHL, where $\frac{2\theta}{\theta+1}\approx 1.3102\neq\frac43$.
They further conjectured that this special normalisation should not be necessary: in the unit-resistance setting, the growth exponent of the two-terminal effective resistance should exist almost surely.
This open problem is the gateway to the natural setting: once the exponent exists, the natural Brownian motion, namely the one built in the unit-resistance setting, should also exist and should likewise fail the Alexander--Orbach-type conjecture.

\bigskip

Hambly and Kumagai's analysis of the DHL is a landmark result for diffusions on hierarchical lattices.
It constructs a diffusion on the scaling limit and derives on-diagonal heat-kernel bounds in a setting where local volume behaviour is highly non-uniform and infinite degree appears.
At the level of scaling exponents, however, the DHL sits at a marginal point:
the Hausdorff and Minkowski dimensions of the limit space both equal $2$, the two-terminal effective resistance stays of order one across levels, and hence $\dimresis(\Xi)=0$ (the resistance dimension, Definition~\ref{def:resistance_dimension}).
In this regime the scaling-limit diffusion is defined directly from the renormalised Dirichlet energy rather than from a resistance form.
Consequently, the typical spectral dimension equals $2$, while the distinguished finite-born point carries spectral dimension $1$.

The purpose of the present paper is to develop an analytic framework for a broad class of fractal graphs including the DHL and hierarchical lattices.
Edge Iterated Graph Systems (EIGS) provide such a class, encompassing many classical examples, including the DHL, Vicsek-type graphs, Cayley graph of $F_2$ and $(u,v)$-flowers.
However, this breadth comes at genuine analytic costs and technical difficulties.

We shall next review the background, in order to better explain and position the contributions of this paper.

\subsection{Background}\label{subsec:background}

\begin{table}[ht]
\centering
\caption{Three classes of fractal spaces for diffusion processes and Dirichlet forms.}
\begingroup
\footnotesize
\setlength{\tabcolsep}{3.2pt}
\renewcommand{\arraystretch}{1.30}
\begin{tabularx}{\textwidth}{
>{\raggedright\arraybackslash}p{0.155\textwidth}
>{\raggedright\arraybackslash}X
>{\raggedright\arraybackslash}X
>{\raggedright\arraybackslash}X}
\toprule
\textbf{Fractal type}
&
\textbf{Finitely ramified / p.c.f.\ fractals}
&
\textbf{Infinitely ramified / non-p.c.f.\ fractals}
&
\textbf{Iterated graph systems including hierarchical lattices}
\\
\midrule
\textbf{Ambient space}
&
{Euclidean space}
&
{Euclidean space}
&
{graph / metric space}
\\
\textbf{Classical example(s)}
&
Sierpi\'nski gasket
&
Sierpi\'nski carpet
&
hierarchical lattice; EIGS
\\
\textbf{Bounded degree}
&
\checkmark
&
\checkmark
&
\checkmark\ \& \texttimes
\\
\textbf{Usually recurrent}
&
\checkmark
&
\checkmark\ \& \texttimes
&
\checkmark\ \& \texttimes
\\
\textbf{Exact renormalisation}
&
\checkmark
&
\texttimes
&
\checkmark
\\
\textbf{Non-trivial percolation}
&
\texttimes
&
\checkmark
&
\checkmark\ \& \texttimes
\\
\textbf{Representative literature}
&
Barlow--Perkins \cite{barlow1988brownian};
Kigami \cite{kigami1993harmonic,kigami2001analysis};
Kusuoka \cite{kusuoka1989dirichlet,kusuoka1993lecture};
Kumagai \cite{kumagai1993estimates,kumagai2004heat};
Hambly--Kumagai \cite{hambly1999transition}
&
Barlow--Bass \cite{barlow1989the,barlow1992transition,barlow1999brownian};
Barlow--Hambly \cite{barlow1997diffusion};
Croydon \cite{croydon2007heat,croydon2008volume,croydon2018scaling};
Grigor'yan--Telcs \cite{grigoryan2002harnack};
Cao--Qiu \cite{cao2021sierpinski};
Halberstam--Hutchcroft \cite{halberstam2022what};
&
Hambly--Kumagai \cite{hambly2010diffusion};
Ruiz \cite{ruiz2018explicit,ruiz2021heat}
\\
\textbf{Role in this paper}
&
parallel results for Brownian motion and heat kernel
&
parallel results for diffusion limit and heat kernel
&
main class of this paper
\\
\bottomrule
\end{tabularx}
\endgroup
\label{tab:three_fractal_diffusion_classes}
\end{table}

Table~\ref{tab:three_fractal_diffusion_classes} summarises the comparison that will guide the discussion below.
The rigorous construction of diffusion processes on fractal spaces goes back to the late 1980s, starting with Brownian motion on the Sierpi\'nski gasket and the development of resistance forms and Dirichlet-form methods on self-similar sets
\cite{barlow1988brownian,kusuoka1989dirichlet,kigami1993harmonic,kigami2001analysis}. 
For the purposes of this paper, it is useful to organise the fractals into three classes.

The first class consists of finitely ramified, and in particular p.c.f., self-similar fractals such as the Sierpi\'nski gasket. 
In this setting, the finite boundary structure makes the renormalisation possible. 
Consequently, one can often construct resistance forms, Laplacians, Brownian motions and heat kernels by an exact renormalisation procedure. 
This leads to many sharp and explicit results. 
At the same time, the finite ramification imposes strong geometric restrictions. 
In the classical examples the associated diffusion is recurrent, and often point recurrent, and the usual Bernoulli percolation problem has no non-trivial phase transition: the critical probability is $p_c=1$.

The second class consists of infinitely ramified, non-p.c.f. fractals such as the Sierpi\'nski carpet and related carpet-like spaces. 
These spaces have much richer large-scale geometry and may support either recurrent or transient behaviour, depending on the underlying dimension and scaling. 
Diffusions on such spaces are usually constructed by analytic Dirichlet-form methods, invariance principles, or heat-kernel estimates rather than by a finite-dimensional resistance-renormalisation scheme. 
Compared with the p.c.f. case, exact formulae for quantities such as the spectral dimension or the relevant renormalisation constants are generally unavailable. 
However, non-trivial percolation phase transitions may occur, but the corresponding critical parameter $p_c\in(0,1)$ is typically not available in closed form.

The third class consists of hierarchical and iterated systems, including Edge Iterated Graph Systems (including hierarchical lattices) considered in this paper. 
These models retain an exact recursive structure, while allowing much greater flexibility than classical p.c.f. fractals: they may be finitely or infinitely ramified, recurrent or transient, and their graphs may have either bounded or unbounded degrees. 
This class has been studied much less systematically from the viewpoint of probability. 
The closest benchmark for the present work is the landmark paper of Hambly and Kumagai on the diamond hierarchical lattice and its critical percolation cluster \cite{hambly2010diffusion}. 
A major additional difficulty in the EIGS setting is that the graphs need not have uniformly bounded vertex degree; finite-born vertices may acquire degrees growing with the level. 
Nevertheless, hierarchical and iterated systems also have decisive advantages. 
Their exact replacement structure makes it possible, in most cases, to compute scaling exponents explicitly, to distinguish recurrent and transient regimes, and to obtain exact renormalisation formulae. 
As this paper shows, the tools that are standard in classes 1 and 2 do not transfer easily to class 3, and the corresponding results are accordingly harder to obtain.

\subsection{Edge iterated graph systems}\label{subsec:eigs}

Iterated graph systems are a natural extension of what physicists call hierarchical lattices; they
are fractal graphs generated by recursive iteration.
We do not claim that the basic idea underlying EIGS is new.
Closely related constructions can be traced back at least to hierarchical lattices in statistical physics in the 1980s, where they have long served as important model spaces for renormalisation and related problems. Nevertheless, much of the existing literature is model-specific and heuristic, and it does not provide a sufficiently general mathematical framework for the questions studied in this paper. 
There are several concrete reasons for this.

\begin{itemize}
    \item Hierarchical lattices have not been formulated using directed rule graphs. 
    This is harmless in many classical examples, since the rule graphs are symmetric with respect to the two terminals. However, without directed rule graphs, a general edge-substitution procedure is not well-defined.

    \item As shown in \cite{li2024on}, if the rule graphs are not directed, then the substitution rule does not determine a unique degree growth. Consequently, fundamental quantities such as degree growth and degree dimension cannot be studied unambiguously in the general setting.

    \item The existing literature rarely separates the rule graphs from the iterative procedure itself, in the sense of treating the construction as a discrete dynamical system generated by a substitution operator. 
    This multicolour formalism is not merely cosmetic: it provides exactly the bookkeeping required by the critical percolation cluster.
\end{itemize}

Thus, while the substitution intuition behind EIGS is classical, a precise definition of Edge Iterated Graph Systems is both natural and necessary. 
It turns a family of previously heuristic or model-specific constructions into a well-defined framework in which distances, degrees, scaling limits, and random models such as percolation can be studied consistently.

Specifically, closely related edge-replacement constructions have appeared in the statistical-physics and complex-networks literature.
In particular, a line of work by Xi studied such models first as self-similar networks and later under the name substitution networks, with emphasis on scale-free behaviour, fractality, and average-distance asymptotics \cite{chen2014class,xi2017fractality,ye2019average}.
Useful tools for the rigorous study of these models, most notably the degree dimension, were generalised to a multiple-colour random setting in \cite{li2024on}.
This point of view was then formalised in the language of iterated graph systems in \cite{neroli2024fractal}, where the IGS framework was developed as a mathematical setting for studying fractal dimensions on such graph limits.
The follow-up works \cite{anttila2024constructions,anttila2025combinatorial,anttila2025construction,li2025reducible} already demonstrate the remarkable breadth and strength of the IGS programme, ranging from rich replacement constructions and counterexamples to Kleiner's conjecture, through broad results on the combinatorial Loewner property and super-multiplicativity, to the construction of self-similar energy forms and the discovery of singular Sobolev phenomena on Laakso-type IGS fractals.

For the sake of rigour, we briefly restate the definition of EIGS here.

\begin{figure}[ht]
    \centering
    \includegraphics[width=0.9\linewidth]{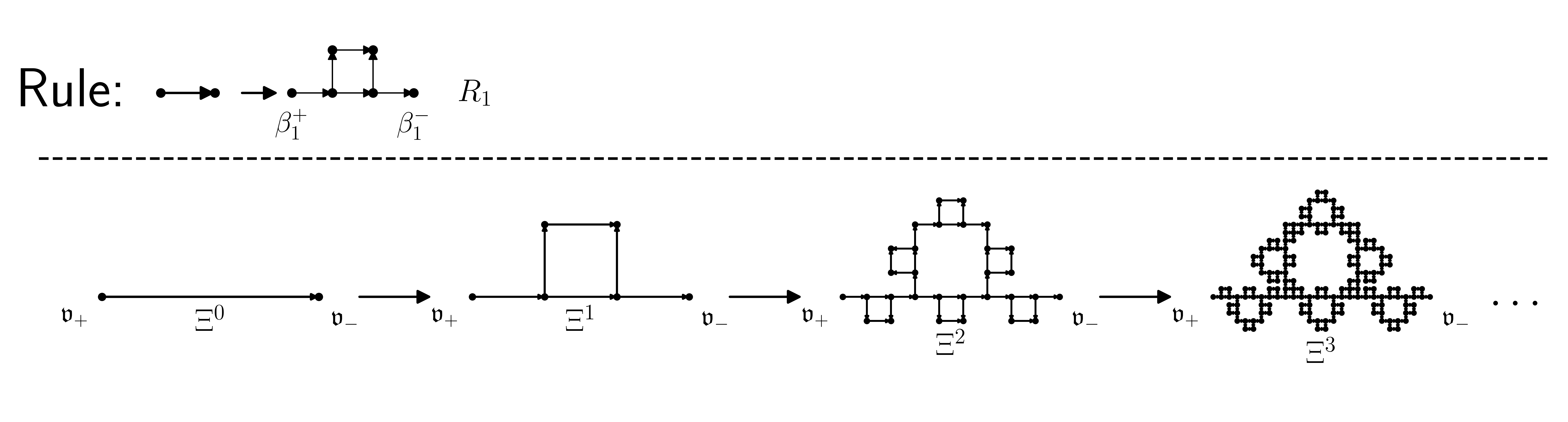}
    \caption{An example: Canonical Xi graph (in $\mathbb Z^2$)}
    \label{fig:canonical_xi}
\end{figure}

\begin{definition}\label{definition:EIGS}
See Figure~\ref{fig:canonical_xi} and Figure~\ref{fig:two_colour} for examples.

An \emph{Edge Iterated Graph System} is a triple
\[
   \mathscr I=(\Xi^{0},\mathcal R,\mathcal S)
\]
specified as follows.
Let $K\in\mathbb N$ be the number of edge colours.
For every graph $G$ that appears in the construction, let
$
   \mathscr{C} : E(G) \longrightarrow[K]
$
be a colour map.

\begin{enumerate}
\item \textbf{Initial graph.}
      $\Xi^{0}$ is an edge-coloured finite directed graph.
      If $\Xi^{0}$ is just an edge, we denote its two vertices by $\mathfrak v_+$ and $\mathfrak v_-$.
      In particular, $\Xi_{\iota}^0$ denotes the special case that the initial graph is just one $\iota$-coloured directed edge.
      In this paper, $\iota$ always represents the colour of the initial graph if it only contains one edge or one vertex.

\item \textbf{Rule graphs.}
      $\mathcal{R} = \{ R_i \}_{i=1}^{K}$ is a family of finite directed graphs.
      Each $R_i$ contains two \emph{planting vertices}
      \[
         \beta_i^{+},\beta_i^{-} \in  V(R_i)
         \quad\&\quad
         \beta_i^{+}\neq\beta_i^{-}.
      \]
      Every edge of $R_i$ is coloured by the same map $\mathscr{C}$.

\item \textbf{Substitution operator.}
      For an edge $e=(a,b)\in E(G)$, the operator
      $
         \mathcal S\colon e\longmapsto R_{\mathscr{C}(e)}  \ \big|\
         a\equiv\beta_i^{+},\ b\equiv\beta_i^{-}
      $
      replaces $e$ by a \emph{fresh copy} of the rule graph $R_i$ where \(i=\mathscr C(e)\), identifying
      $
         a\equiv\beta_i^{+}
      $ and
      $
         b\equiv\beta_i^{-}.
      $
      All arcs are substituted \emph{simultaneously in parallel} at each step.
      Vertices are identified only at the planting vertices of their respective copies, and no further identifications are made.

\item \textbf{Iteration.}
      Set $\Xi^{n+1}:=\mathcal S(\Xi^{n})$ for $n\ge 0$; explicitly
      \[
         \Xi^{n+1}
         =
         \bigcup_{e\in E(\Xi^{n})}\mathcal S(e),
      \]
      where the union is taken with the vertex identifications described above.
      The sequence $(\Xi^{n})_{n\in\mathbb N}$ is called the \emph{EIGS substitution networks}.
\end{enumerate}
\end{definition}

\begin{figure}[ht]
    \centering
    \includegraphics[width=0.8\linewidth]{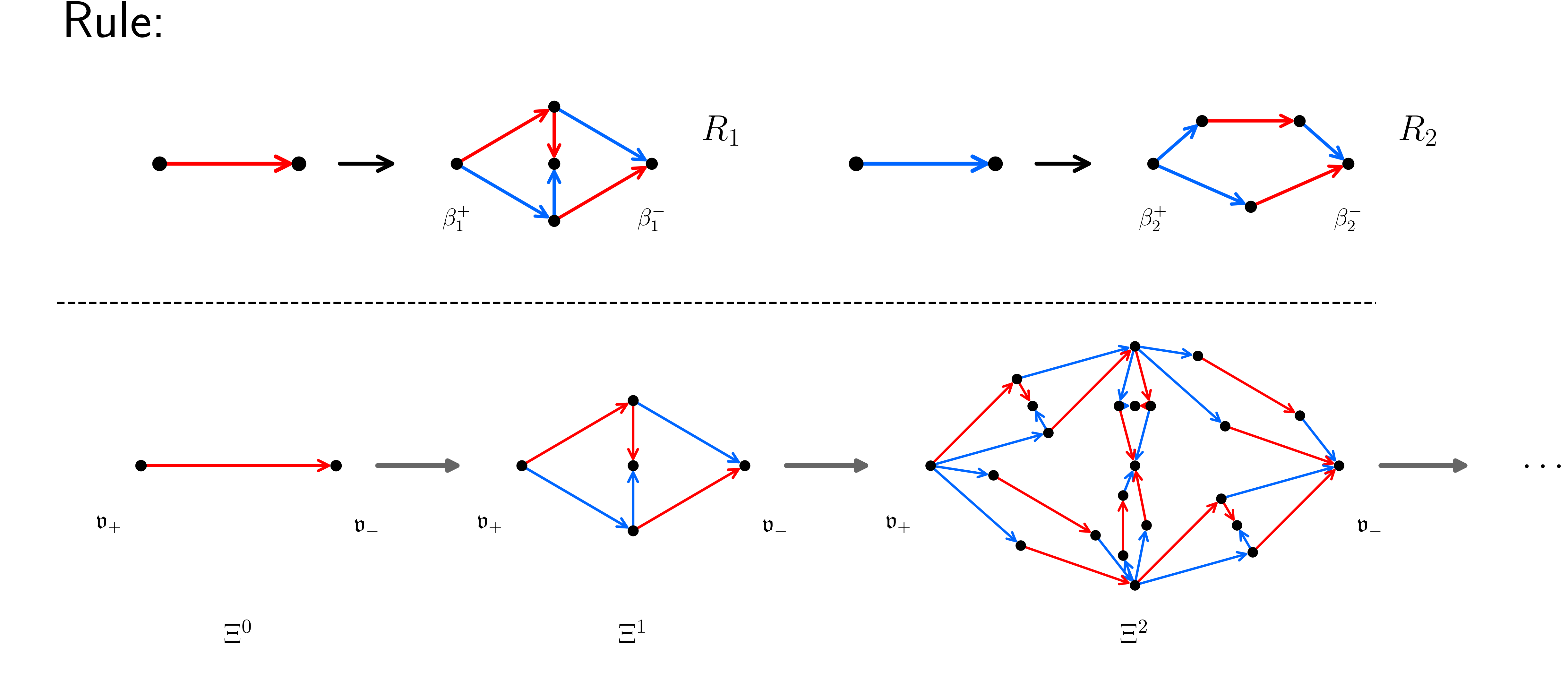}
    \caption{An example of two-coloured EIGS}
    \label{fig:two_colour}
\end{figure}

In this paper, we must distinguish two classes of objects in order to avoid ambiguity.

\textbf{Combinatorial limit.}
Under our construction, the combinatorial limit graph $\Xi$ is a natural object, and the simple random walk on it is usually defined at the intuitive level.
However, when we work combinatorially we are often studying asymptotic properties of the graph sequence, since
\[
\mathscr V:=V(\Xi):=\bigcup_{n\in\mathbb N} V(\Xi^n)
\]
is simply the union of the finite vertex sets.
Nevertheless, understanding the combinatorial limit and its asymptotics is an essential prerequisite for what follows.

\textbf{Metric scaling limit.}
To give the limiting object a fully explicit definition, we embed the rescaled graphs into a compact metric measure space
\[
(V(\Xi^n),\hat d_{\Xi^n},\mu_{\Xi^n})
\longrightarrow
(\Xi_{\mathrm M}=\mathscr V \cup \mathscr V^\infty, d_{\Xi_{\mathrm M}},\mu_{\Xi_{\mathrm M}}),
\]
which we call the Gromov--Hausdorff--Prokhorov scaling limit.
Here $\mu_{\Xi^n}$ is the normalised counting measure on $V(\Xi^n)$, and $\mu_{\Xi_{\mathrm M}}$ is its weak limit.
In this setting, the limit is a well-defined compact space, and we can clearly distinguish the embedded finite-born points
$
\mathscr V\subset \Xi_{\mathrm M}
$
from the points added by the closure,
$
\mathscr V^\infty:=\Xi_{\mathrm M}\setminus \mathscr V.
$
As we will see later, these two classes of points may exhibit different local mass and resistance behaviour, but surprisingly they share the same walk dimension.

\paragraph{Growth matrices and scaling exponents.}
We briefly recall the three growth quantities that govern the large-scale geometry of an EIGS.
Let \(K\) be the number of colours and, for a coloured directed graph \(G\), define
\[
    [\chi(G)]_j
    :=
    \bigl|\{e\in E(G): \mathscr C(e)=j\}\bigr|,
    \qquad j\in[K].
\]
The \emph{mass matrix} \(\mathbf M\in\mathbb N^{K\times K}\) is defined by
\[
    [\mathbf M]_{ij}
    :=
    [\chi(R_i)]_j
    =
    \bigl|\{e\in E(R_i): \mathscr C(e)=j\}\bigr|.
\]
Thus the \(i\)-th row of \(\mathbf M\) records the colour composition of the rule graph \(R_i\). 
Consequently, in the primitive deterministic setting,
\[
    |E(\Xi_\iota^n)| \asymp |V(\Xi_\iota^n)|
    \asymp \rho(\mathbf M)^n ,
\]
so \(\rho(\mathbf M)\), the spectral radius of $\mathbf M$, is the exponential mass-growth factor \cite{neroli2024fractal}.

The \emph{degree matrix} \(\mathbf N\in\mathbb N^{2K\times 2K}\) records how coloured degrees at planting vertices are amplified under substitution. For a vertex \(v\) in a coloured directed graph \(G\), define
\[
    [\boldsymbol \kappa_G(v)]_{2j-1}
    :=
    \bigl|\{e=(v,w)\in E(G): \mathscr C(e)=j\}\bigr|,
    \qquad
    [\boldsymbol \kappa_G(v)]_{2j}
    :=
    \bigl|\{e=(w,v)\in E(G): \mathscr C(e)=j\}\bigr|.
\]
Then
\[
    \mathbf N
    :=
    \begin{pmatrix}
        \boldsymbol \kappa_{R_1}(\beta_1^+)\\
        \boldsymbol \kappa_{R_1}(\beta_1^-)\\
        \vdots\\
        \boldsymbol \kappa_{R_K}(\beta_K^+)\\
        \boldsymbol \kappa_{R_K}(\beta_K^-)
    \end{pmatrix}.
\]
If a vertex \(v\) is born at level \(m\) (that is, \(v\in V(\Xi^m)\setminus V(\Xi^{m-1})\), then its coloured incidence vector at level \(n\ge m\) is governed by
\[
    \boldsymbol \kappa_{\Xi^n}(v)=\boldsymbol \kappa_{\Xi^m}(v)\mathbf N^{\,n-m}.
\]
Thus \(\rho(\mathbf N)\) is the exponential degree-amplification factor. When \(\rho(\mathbf N)>1\), the corresponding degree dimension in the primitive deterministic case is, by \cite{li2024on},
\[
    \dimdeg(\Xi)
    =
    \frac{\log \rho(\mathbf M)}{\log \rho(\mathbf N)}.
\]
Equivalently, \(\mathbf M\) measures how many vertices and edges are produced globally, whereas \(\mathbf N\) measures how fast the degrees of finite-born vertices grow.

Finally, the distance scale is encoded by the path-selection family \(\mathcal D\). The full construction of \(\mathcal D\) is given in \cite{neroli2024fractal}; for the present overview, we only need the resulting quantity
\[
    \rho_{\min}(\mathcal D),
\]
which is the minimal exponential growth rate of graph distances under substitution. More precisely,
\[
    d_{\Xi^n}(\mathfrak v_+,\mathfrak v_-)
    \asymp
    \operatorname{diam}(\Xi^n)
    \asymp
    \rho_{\min}(\mathcal D)^n .
\]
Therefore, in the primitive deterministic setting,
\[
    \dimbox(\Xi)
    =
    \dimhaus(\Xi)
    =
    \frac{\log \rho(\mathbf M)}
         {\log \rho_{\min}(\mathcal D)} .
\]
Here $\dimhaus(\Xi)$ denotes the Hausdorff dimension of the limit metric space, and the equality is Theorem~2.1 of \cite{neroli2024fractal}.
In short, \(\mathbf M\) gives the mass scale, \(\mathcal D\) gives the metric scale, and \(\mathbf N\) gives the degree scale. This is why the fractal dimension is a mass--distance ratio, while the degree dimension is a mass--degree ratio. 
In the random coloured setting, the same interpretation holds with the spectral radii of the expectations of random matrices \cite{hambly2010diffusion}.

\bigskip

\begin{assumption}\label{ass:standing}
Although our main objects arise from EIGSs, several later arguments only use a more general graph-theoretic framework.
We therefore fix the following standing convention.

Throughout this paper, unless explicitly stated otherwise, when we write $G$ for a graph in a general graph-theoretic setting, $G$ is finite, simple, connected, and undirected.
If $(G^n)_{n\in\mathbb N}$ is a vertex-nested sequence of such graphs, we write $G^\infty$ for its combinatorial limit and assume that $\diam(G^n)\to\infty$ as $n\to\infty$.

\begin{enumerate}
\item \textbf{Canonicality.}
We say that $\mathscr I$ is \textbf{canonical} if for every $i\in[K]$ and every edge $e\in E(R_i)$, the edge $e$ belongs to at least one simple path in the underlying undirected graph of $R_i$ connecting $\beta_i^+$ and $\beta_i^-$.
We assume that the underlying Edge Iterated Graph System is canonical.

\item \textbf{Distance-positivity and degree-primitivity.}
The Edge Iterated Graph System is distance-positive.
That is, all matrices in $\mathcal D$ are positive, see Definition~2.8 in \cite{neroli2024fractal}.
If $K=1$, we require $d_{R_1}(\beta_1^+,\beta_1^-) \geq 2$.
Note that, under this assumption, $\diam(\Xi^n)\to\infty$.
We assume $\mathbf N$ is primitive, see Definition~5.3 in \cite{li2024on}.

\item \textbf{Single-edge initial graph.}
The initial graph $\Xi^{0}$ consists of a single directed edge whose endpoints are denoted by $\mathfrak v_{+}$ and $\mathfrak v_{-}$.
\end{enumerate}
\end{assumption}

The canonicality assumption is primarily used to exclude electrically irrelevant decorations, namely edges and hence vertex sets which do not lie on any simple path connecting the planting vertices.
Such decorations do not affect the two-terminal renormalisation map $\Psi$, but they can alter global resistance geometry after renormalisation.
Although the canonicality assumption is strong, it is not fundamentally unmanageable.
We adopt it mainly because it makes the presentation substantially cleaner and more readable.
In fact, in the regime $\dimresis(\Xi)>0$, whether or not the system is canonical makes no difference to most of our conclusions at the level of scaling exponents.
When $\dimresis(\Xi)\le 0$, one needs more delicate resistance estimates, and genuinely technical phenomena may occur, for instance uniform renormalisation may break down.
However, these issues do not affect the overarching structure of the theory or its main ideas.

\subsection{Main results}\label{subsec:main_results}

The foundational pillar of this paper is that every scaling exponent of an EIGS admits a closed form.
The mass, distance and degree growth rates $\rho(\mathbf M)$, $\rho_{\min}(\mathcal D)$ and $\rho(\mathbf N)$ come from the substitution structure itself \cite{neroli2024fractal,li2024on}.
These generate the box and degree dimensions through $\dimbox(\Xi)=\frac{\log\rho(\mathbf M)}{\log\rho_{\min}(\mathcal D)}$ and $\dimdeg(\Xi)=\frac{\log\rho(\mathbf M)}{\log\rho(\mathbf N)}$, whose precise definitions are recalled in Section~\ref{sec:random_walks}.
One feature deserves emphasis before we state the results.
When $\rho(\mathbf N)>1$, a vertex born at level $k$ has degree of order $\rho(\mathbf N)^{n-k}$ in $\Xi^n$, so degrees at a fixed level are highly non-uniform, and in the combinatorial limit every finite-born vertex has infinite degree.
As a consequence even the volume-doubling property fails, a phenomenon already observed for the DHL in \cite{hambly2010diffusion}.

The resistance growth rate $\rho(\Psi)$, in turn, is the nonlinear Perron eigenvalue of the resistance renormalisation map $\Psi$ introduced in Section~\ref{sec:random_walks} (Lemma~\ref{thm:sdprimitive_eigenpair}).
From this we prove that the resistance dimension exists and satisfies $\dimresis(\Xi)=\frac{\log\rho(\Psi)}{\log\rho_{\min}(\mathcal D)}$ (Theorem~\ref{thm:resistance_dimension_formula}), and we derive the local mass dimension, the local resistance dimension and the walk dimension on the combinatorial limit.
The behaviour of an EIGS is organised by the position of $\rho(\Psi)$ relative to $1$, equivalently by the sign of $\dimresis(\Xi)$.

\begin{itemize}
\item \textbf{Subcritical: superconducting} ($\rho(\Psi)<1$, equivalently $\dimresis(\Xi)<0$).
The two-terminal effective resistance tends to $0$ while every single edge keeps resistance $1$, so the graph is asymptotically superconducting although no single wire is; after renormalisation, each microscopic edge becomes an insulator.
The random walk is transient and, counter-intuitively, typically reaches distant hubs before nearby low-degree targets.
The $(u,v)$-flowers with $u>v$ are of this type.
\item \textbf{Critical: Hambly--Kumagai} ($\rho(\Psi)=1$).
The two-terminal effective resistance converges to a positive constant and no renormalisation is needed, so crossing a cell costs the same order-one resistance at every scale.
The random walk is therefore insensitive to distance in the resistance sense.
This is precisely the setting of \cite{hambly2010diffusion}, whose classical representative is the DHL.
\item \textbf{Supercritical: Brownian} ($\rho(\Psi)>1$).
The two-terminal effective resistance diverges while the renormalised resistance of a microscopic edge vanishes, so resistance grows with distance, exactly as physical intuition suggests.
The simple random walk is recurrent, and the diffusion limit upgrades to the Brownian motion of a compatible resistance form.
The Vicsek, Laakso and Xi graphs are of this type.
\end{itemize}

The three types could hardly behave more differently.
Nevertheless, we prove that all vertices, whether finite-born with exploding degrees or late-born with bounded degrees, and all three types, share one and the same walk dimension following the Einstein relation.

\begin{restated}{\ref{thm:Einstein}}
All vertices in $\Xi$ share the same walk dimension at every scale, even with local infinity.
\[
\dimwalk(\Xi)=\dimbox(\Xi)+\dimresis(\Xi).
\]
\end{restated}

Two cancellations lie behind this uniformity.
Locally, the degree amplification and the decay of the local effective resistance cancel each other in the exit time, so the walk dimension does not feel the birth level of the base vertex (Theorem~\ref{thm:exit_time_scale}).
Across the three types, the common value is then dictated by the Einstein relation alone.
In particular, the Einstein relation yields $\dimwalk(\Xi)\ge 2$ (Proposition~\ref{prop:walk_dimension_ge_2}), so the random walk on an EIGS is always diffusive or subdiffusive, and never superdiffusive.

A walk dimension shared by every vertex at every scale is exactly what a diffusion limit requires, because it allows a single time acceleration $a_n$ for the whole graph.
In Section~\ref{sec:diffusion} we therefore realise the combinatorial limit inside a compact metric space $\Xi_{\mathrm M}=\mathscr V\cup\mathscr V^\infty$, equip it with the weak limit $\mu_{\Xi_{\mathrm M}}$ of the normalised counting measures, and construct the limiting Dirichlet form.
Below, $\mathscr V^\infty:=\Xi_{\mathrm M}\setminus\mathscr V$ is the set of infinite-born limit points, $\hat d_{\Xi^n}$ and $\mu_{\Xi^n}$ are the diameter-rescaled graph metric and the normalised counting measure, $v_n$ is a distinguished root, $X^{\Xi^n}$ is the simple random walk, and GHPS abbreviates the Gromov--Hausdorff--Prokhorov--Skorokhod convergence.

\begin{restated}{\ref{thm:sec31_ghps}}
For every $T>0$,
\[
\Bigl(V(\Xi^n),\hat d_{\Xi^n},\mu_{\Xi^n},v_n,\bigl(X^{\Xi^n}_{\lfloor a_nt\rfloor}\bigr)_{0\le t\le T}\Bigr)
\xrightarrow{\ \mathrm{GHPS}\ }
\Bigl(\Xi_{\mathrm M},d_{\Xi_{\mathrm M}},\mu_{\Xi_{\mathrm M}},x,(W_t)_{0\le t\le T}\Bigr).
\]
\end{restated}

In the Brownian type, the limiting process $W$ coincides with the Brownian motion of a compatible resistance form (Theorem~\ref{thm:sec32_brownian_motion}).
With the basic construction in hand, the uniform exit-time scaling becomes a tool for on-diagonal heat kernel estimates in the natural metric.
Below, $p_t^\Xi$ denotes the jointly continuous heat kernel of the limiting diffusion with respect to $\mu_{\Xi_{\mathrm M}}$, and $t_0>0$ is a fixed small time.
The key point is that $\dimdeg(\Xi)$ enters as a correction term at finite-born points, while the correction disappears at $\mu_{\Xi_{\mathrm M}}$-almost every point of $\mathscr V^\infty$.

\begin{restated}{\ref{thm:diffusion-heat}}
For every $x\in\mathscr V$ there exist $c_{h,\mathscr V},C_{h,\mathscr V}>0$ depending on $x$ such that for all $t\in(0,t_0)$,
\[
c_{h,\mathscr V}\, t^{-\frac{\dimbox(\Xi)\bigl(1-\frac{1}{\dimdeg(\Xi)}\bigr)}{\dimwalk(\Xi)}}
\le
p_t^{\Xi}(x,x)
\le
C_{h,\mathscr V}\, t^{-\frac{\dimbox(\Xi)\bigl(1-\frac{1}{\dimdeg(\Xi)}\bigr)}{\dimwalk(\Xi)}}.
\]
For every $\varepsilon>0$ and for $\mu_{\Xi_{\mathrm M}}$-almost every $x\in\mathscr V^\infty$, there exist $c_{h,\mathscr V^\infty},C_{h,\mathscr V^\infty}>0$ depending on $x$ and $\varepsilon$ such that for all $t\in(0,t_0)$,
\[
c_{h,\mathscr V^\infty}\,
|\log t|^{-\frac{1}{\dimdeg(\Xi)-1}-\varepsilon}\,
t^{-\frac{\dimbox(\Xi)}{\dimwalk(\Xi)}}
\le
p_t^{\Xi}(x,x)
\le
C_{h,\mathscr V^\infty}\,
t^{-\frac{\dimbox(\Xi)}{\dimwalk(\Xi)}}.
\]
\end{restated}

For the DHL these bounds recover the polynomial order of \cite{hambly2010diffusion} and improve the logarithmic correction in their typical-point lower bound from $|\log t|^{-2-\varepsilon}$ to $|\log t|^{-1-\varepsilon}$.
Though the statement above is admittedly heavy,
rewritten as a spectral dimension, it becomes transparent.
Here $\dimspec^{(\mathrm N)}(\Xi_{\mathrm M}:x):=-\lim_{t\downarrow 0}\frac{2\log p_t^\Xi(x,x)}{\log t}$ is the local spectral dimension; see Section~\ref{sec:diffusion} for the formal definition.

\begin{restated}{\ref{thm:diffusion-spectral}}
With the convention that $\dimdeg(\Xi)=\infty$ whenever $\Xi$ is not scale-free, we have
\[
\dimspec^{(\mathrm N)}\bigl(\Xi_{\mathrm M}:x\bigr)
=
\begin{cases}
\dfrac{2\dimbox(\Xi)\bigl(1-\frac{1}{\dimdeg(\Xi)}\bigr)}{\dimwalk(\Xi)}, & x\in\mathscr V,\\[2ex]
\dfrac{2\dimbox(\Xi)}{\dimwalk(\Xi)}, & x\in\mathscr V^\infty\ (\mu_{\Xi_{\mathrm M}}\text{-a.e.}).
\end{cases}
\]
\end{restated}

The formula reveals a dichotomy: points of one EIGS fractal may diffuse at two different speeds, and the discrepancy is caused purely by local infinity.
The degree dimension is exactly the correction term that unifies the two, and under the convention $\dimdeg(\Xi)=\infty$ in the locally finite case, a unified formula covers all EIGS.
In the Brownian type the diffusion and Brownian spectral dimensions moreover coincide (Theorem~\ref{thm:brownian-spectral}).
Table~\ref{tab:dimensions} collects deterministic and random examples: once $\dimbox(\Xi)$, $\dimdeg(\Xi)$ and $\dimresis(\Xi)$ are known, the walk and spectral dimensions follow automatically, and the last column records the diffusion type of each example.

\begin{table}[t]
\caption{Dimensions of common fractal graphs (* conjectural values, see Section~\ref{subsec:dhl_cluster}).}
\label{tab:dimensions}
\makebox[\linewidth][c]{%
\begin{tabular}{|>{\centering\arraybackslash}m{2.2cm}|*{7}{>{\centering\arraybackslash}m{1.6cm}|}}
\hline
\multirow{2}{*}{\diagbox[width=2.64cm,height=1.27cm]{\emph{Graphs}}{\emph{Dimensions}}}
& $\dimbox$ & $\dimdeg$ & $\dimresis$ & $\dimwalk$ & $\dimspec(\mathscr V)$ & $\dimspec(\mathscr V^\infty)$
& Diffusion type \\
\cline{2-8}
& Thm 2.1 \cite{neroli2024fractal}
& Thm 5.5 \cite{li2024on}
& Thm \ref{thm:resistance_dimension_formula}
& Thm \ref{thm:Einstein}
& Thm \ref{thm:diffusion-spectral}
& Thm \ref{thm:diffusion-spectral}
& Rem \ref{rem:three_resistance_regimes} \\
\hline
DHL
& $2$ & $2$ & $0$ & $2$ & $1$ & $2$ & HK \\
\hline
$(u,v)$-flower
& $1+\frac{\log u}{\log v}$
& $1+\frac{\log v}{\log u}$
& $1-\frac{\log u}{\log v}$
& $2$
& $1$
& $1+\frac{\log u}{\log v}$
& {\small SC/HK/BM} \\
\hline
Laakso graph
& $\frac{\log 6}{\log 4}$
& $\infty$
& $\frac{\log 3}{\log 4}$
& $\frac{\log 18}{\log 4}$
& $\frac{2\log 6}{\log 18}$
& $\frac{2\log 6}{\log 18}$
& BM \\
\hline
Canonical Xi graph
& $\frac{\log 6}{\log 3}$
& $\infty$
& $\frac{\log (11/4)}{\log 3}$
& $\frac{\log (33/2)}{\log 3}$
& $\frac{2\log 6}{\log (33/2)}$
& $\frac{2\log 6}{\log (33/2)}$
& BM \\
\hline
Figure~\ref{fig:two_colour}
& $2.4461$
& $2.4461$
& $0.1455$
& $2.5916$
& $1.1160$
& $1.8877$
& BM \\
\hline
\hline
Percolation cluster DHL
& $1.8993$
& $1.8993$
& $0.8124$
& *$2.7117$
& *$0.6633$
& *$1.4008$
& *BM \\
\hline
\end{tabular}%
}
\end{table}

After all this, we have not forgotten one of the motivations behind Hambly and Kumagai's study of the DHL.
This is an Alexander--Orbach-type question \cite{alexander1982density}: does the critical percolation cluster carry the spectral dimension $\frac43$?
Hambly and Kumagai constructed the diffusion on the scaling limit of the two-terminal critical cluster in a setting where the edge weights are renormalised so that the total resistance across the cluster stays equal to one.
In the natural unit-resistance setting, everything hinges on the exponential growth rate of the two-terminal effective resistance, and the existence of this rate was left open in \cite{hambly2010diffusion}.
We solve this problem; here $R_n$ denotes the effective resistance between the two boundary vertices of the level-$n$ cluster.

\begin{restated}{\ref{thm:reduced_quenched_exponent}}
For every $p\in(0,1)$, there exists a constant $\alpha(p)\in[0,\log 2]$ such that
\[
\lim_{n\to\infty}\frac1n\log R_n=\lim_{n\to\infty}\frac1n\log\mathbb E[R_n]=\alpha(p)
\qquad\text{almost surely}.
\]
Moreover, at the critical parameter $p_c=\frac{\sqrt5-1}{2}$, $\alpha(p_c)\ge\log\frac{14-4\sqrt5}{3}>0$ (Proposition~\ref{prop:alpha_pc_positive}).
\end{restated}

Numerically $\alpha(p_c)\approx 0.5631$, hence $\dimresis\approx 0.8124>0$.
In other words, critical percolation drives the DHL out of the Hambly--Kumagai regime into the Brownian regime: the lattice itself is critical for the resistance flow, while its critical cluster becomes supercritical.
Accordingly, we conjecture in Section~\ref{sec:percolation} that the cluster carries a Brownian motion in the natural metric, with typical spectral dimension approximately $1.4008$.
This value is strictly larger than the Alexander--Orbach value $\frac43\approx 1.3333$, so the Alexander--Orbach-type conjecture is expected to fail for this hierarchical model.
With this, the story returns to where it began.

\subsection{Formalisation by Lean}\label{subsec:formalisation}
Sections~\ref{sec:random_walks} to~\ref{sec:percolation} of this paper have been formalised in Lean~4, with proofs produced by Aristotle (Harmonic) \cite{achim2025aristotle}, and the code is available at \url{https://github.com/Nero-17/IGS-RWDL-LEAN}.
The projects compile against pinned revisions of \texttt{mathlib} and contain no \texttt{sorry}, no \texttt{admit} and no additional \texttt{axiom}, so every formalised statement depends only on the standard axioms \texttt{propext}, \texttt{Classical.choice} and \texttt{Quot.sound}.
The scope is summarised in Table~\ref{tab:lean}, and we stress that the formalisation is partial.

Results quoted from the literature enter the formalisation as explicit hypotheses.
We emphasise for Section~\ref{sec:diffusion} that Lean verifies the implications from these hypotheses to the conclusions, not the hypotheses themselves.
The machine-checked guarantee therefore falls short of the unconditional correctness of the mathematical statements exactly by the correctness of the quoted results.
Aristotle is used to generate Lean code only.

\begin{table}[ht]
\caption{Scope of the Lean formalisation.}
\label{tab:lean}
\makebox[\linewidth][c]{%
\small
\begin{tabular}{|>{\centering\arraybackslash}m{1.8cm}|m{6cm}|m{6.8cm}|}
\hline
& \emph{Results cited as hypotheses} & \emph{Formalised results} \\
\hline
Section~\ref{sec:random_walks}
& The eigenpair-existence input from nonlinear Perron--Frobenius theory in the proof of Lemma~\ref{thm:sdprimitive_eigenpair}; the distance growth in Lemma~\ref{lem:basic_prop}, quoted from \cite{neroli2024fractal}.
& Lemmas \ref{lem:psi-props} and \ref{lem:renorm}; the uniqueness part of Lemma~\ref{thm:sdprimitive_eigenpair}; Corollaries \ref{lem:Psi_range} and \ref{cor:Psi_strong_convergence}; Lemma~\ref{lem:Reff-asymp-rho}; Theorem~\ref{thm:resistance_dimension_formula}; the mass and degree growth in Lemma~\ref{lem:basic_prop}; Lemmas \ref{lem:ball_volume_k_born} and \ref{lem:mesoscopic_Reff_k_born}; Propositions \ref{prop:local_mass_dimension_constant}, \ref{prop:local_resistance_dimension_k_born} and \ref{prop:asymp_res_recurrence_rhoPsi}; Lemma~\ref{lem:star_domain_exit_time_scale}; Theorem~\ref{thm:exit_time_scale}; Propositions \ref{lem:commute-time-eigs-asymp}, \ref{prop:MPsiD} and \ref{prop:walk_dimension_ge_2}; Theorem~\ref{thm:Einstein}. \\
\hline
Section~\ref{sec:diffusion}
& The GHP limit space; the closure into a regular Dirichlet form and the Hunt-process correspondence; the Kuwae--Shioya process convergence and the Skorokhod representation; Azuma--Hoeffding and Aldous' criterion; the resistance-form theory of Kigami; the multiscale Nash iteration; the inputs of Section~\ref{sec:random_walks}.
& Every lemma, proposition, theorem, corollary and worked example of Section~\ref{sec:diffusion}, except Lemma~\ref{lem:GH_limit_metric} and Theorem~\ref{thm:sec31_ghps}, which are quoted from the literature; the oscillation contraction in Lemma~\ref{lem:sec31_continuous_harmonics} is proved from Assumption~\ref{ass:standing}. \\
\hline
Section~\ref{sec:percolation}
& None.
& Proposition~\ref{prop:dhl_reduction}, as equality of the joint laws of the full resistance sequences; the mean-matrix eigenvalues and the exponent $1.8993$ of Subsection~\ref{subsec:dhl_cluster}; Theorem~\ref{thm:reduced_quenched_exponent} and Proposition~\ref{prop:alpha_pc_positive}; Proposition~\ref{prop:no_deterministic_limit}; Lemmas \ref{lem:annealed_exponent}, \ref{lem:logistic_product} and \ref{lem:uniform_product}. \\
\hline
\end{tabular}%
}
\end{table}

\section{Random walks on graphs}\label{sec:random_walks}

In this section we study various dimensions associated with random walks on the combinatorial direct-limit graph $\Xi$.
As discussed above, we first need asymptotic geometric and electrical information about $\Xi$, such as global effective resistance, local mass growth, and local resistance growth, in order to describe the random walk.
Ultimately, however, the main goal of this section is to obtain sharp exit-time estimates and to identify the walk dimension.

For $N\in\mathbb N$, write $[N]:=\{1,\dots,N\}$.
For a vector $\mathbf x$, write $[\mathbf x]_i$ for its $i$-th coordinate.
For a matrix $\mathbf X$, write $[\mathbf X]_{i,j}$ for its $(i,j)$-entry.
If $\mathbf X$ is square, write $\rho(\mathbf X)$ for its spectral radius.
We write $\mathbf x\le \mathbf y$ when the inequality holds coordinatewise.
We write $f\asymp g$ when there exist constants $c,C>0$ such that $cg\le f\le Cg$.

For $v\in V(G)$ and $r>0$, write
\[
B_G(v,r):=\{u\in V(G): d_G(v,u)<r\}
\]
for the open ball of radius $r$ centred at $v$.
For a finite graph $G$ and a function $f:V(G)\to\mathbb R$, define the Dirichlet energy by
\[
\mathcal E_G(f,f):=\sum_{\{u,w\}\in E(G)} (f(u)-f(w))^2.
\]
For $u,w\in V(G)$, write $\Reff_G(u,w)$ for the effective resistance between $u$ and $w$ in the unit-resistance network on $G$.
For $A\subseteq V(G)$, write
\[
\Reff_G(v,A^{\mathrm c})
:=
\Bigl(
\inf\bigl\{
\mathcal E_G(f,f): f:V(G)\to\mathbb R,\ f(v)=1,\ f\equiv 0 \text{ on } A^{\mathrm c}
\bigr\}
\Bigr)^{-1}.
\]

\subsection{Resistance renormalisation}\label{subsec:resistance}

\begin{definition}\label{def:psi_Psi}
For any finite undirected $K$-coloured graph $G$ and any two vertices $a,b\in V(G)$,
let $\psi_{G:(a,b)}(\mathbf x)\in[0,\infty]$ denote the effective resistance between $a$ and $b$
in the electrical network obtained by assigning resistance $[\mathbf x]_j$ to every $j$-coloured edge.
If $a$ and $b$ are disconnected in $G$, we set $\psi_{G:(a,b)}(\mathbf x)=\infty$.

Define the resistance renormalisation map of any given EIGS $\mathscr I$ by
\[
\Psi(\mathbf x):=\big(\psi_{R_1:(\beta_1^+,\beta_1^-)}(\mathbf x),\dots,
\psi_{R_K:(\beta_K^+,\beta_K^-)}(\mathbf x)\big)\in\mathbb R_{\ge 0}^K.
\]
\end{definition}

\begin{example}\label{ex:Psi_two_colours}
Consider the two-colour EIGS shown in Figure~\ref{fig:two_colour}.
We label the red colour by $1$ and the blue colour by $2$.
Let $\mathbf x=(x_1,x_2)\in(0,\infty)^2$.
A direct calculation gives
\[
\Psi(x_1,x_2)
=
\Big(
\frac{x_1^2+4x_1x_2+x_2^2}{3(x_1+x_2)},
\frac{(x_1+2x_2)(x_1+x_2)}{2x_1+3x_2}
\Big).
\]
\end{example}

\begin{lemma}\label{lem:psi-props}
$\Psi:\mathbb R_{\ge 0}^K\to \mathbb R_{\ge 0}^K$ satisfies:
\begin{enumerate}
\item \textbf{(Monotonicity).}
If $\mathbf x\le \mathbf y$ coordinatewise, then
$
\Psi(\mathbf x)\le \Psi(\mathbf y).
$

\item \textbf{(Homogeneity).}
For all $t\ge 0$,
$
\Psi(t\mathbf x)=t\Psi(\mathbf x).
$

\item \textbf{(Continuity).}
The map $\mathbf x\mapsto \Psi(\mathbf x)$ is continuous on $(0,\infty)^K$.
Moreover, for $\mathbf x\in\mathbb R_{\ge 0}^K$ the limit
$
\lim_{\varepsilon\downarrow 0}\Psi(\mathbf x+\varepsilon\mathbf 1)
$
exists and defines an extension of $\Psi$ to $\mathbb R_{\ge 0}^K$.

\item \textbf{(Concavity).}
For all $\lambda\in[0,1]$ and all $\mathbf x,\mathbf y\in\mathbb R_{\ge 0}^K$,
\[
\Psi(\lambda\mathbf x+(1-\lambda)\mathbf y)
\ge
\lambda \Psi(\mathbf x)+(1-\lambda)\Psi(\mathbf y).
\]

\item \textbf{(Superadditivity).}
For all $\mathbf x,\mathbf y\in\mathbb R_{\ge 0}^K$,
$
\Psi(\mathbf x+\mathbf y)\ge \Psi(\mathbf x)+\Psi(\mathbf y).
$
\end{enumerate}
\end{lemma}

\begin{proof}
Fix $i\in[K]$.
Write $\mathcal F_i$ for the set of unit flows on $R_i$ from $\beta_i^+$ to $\beta_i^-$.
By Thomson's principle,
\[
[\Psi(\mathbf x)]_i
=
\inf_{\theta\in \mathcal F_i}
\sum_{e\in E(R_i)} [\mathbf x]_{\mathscr C(e)}\theta(e)^2.
\]
For fixed $\theta$, the displayed expression is linear in $\mathbf x$.
Therefore monotonicity and homogeneity are immediate, and concavity follows because the infimum of linear functions is concave.

Continuity on $(0,\infty)^K$ is standard for finite electrical networks, since the effective resistance is obtained by solving a finite linear system whose coefficients depend continuously on the edge conductances.
For $\mathbf x\in\mathbb R_{\ge 0}^K$, monotonicity implies that $\Psi(\mathbf x+\varepsilon\mathbf 1)$ decreases coordinatewise as $\varepsilon\downarrow 0$, so the limit exists in $\mathbb R_{\ge 0}^K$.

Finally, homogeneity and concavity give
\[
\Psi(\mathbf x+\mathbf y)
=
2\Psi\Bigl(\frac{\mathbf x+\mathbf y}{2}\Bigr)
\ge
\Psi(\mathbf x)+\Psi(\mathbf y),
\]
which is the required superadditivity.
\end{proof}

\begin{lemma}\label{lem:renorm}
For every $n\ge 0$ and every $\mathbf x\in\mathbb R_{\ge 0}^K$,
\[
\psi_{\Xi^{n}:(\mathfrak v_+,\mathfrak v_-)}(\mathbf x)
=
\psi_{\Xi^{0}:(\mathfrak v_+,\mathfrak v_-)}\big(\Psi^{n}(\mathbf x)\big).
\]
In particular,
\[
\Reff_{\Xi^n}(\mathfrak v_+,\mathfrak v_-)
=
\psi_{\Xi^{n}:(\mathfrak v_+,\mathfrak v_-)}(\mathbf 1)
=
\psi_{\Xi^{0}:(\mathfrak v_+,\mathfrak v_-)}\big(\Psi^{n}(\mathbf 1)\big).
\]
\end{lemma}

\begin{proof}
The case $n=0$ is immediate.
Fix $n\ge 1$ and $\mathbf x\in\mathbb R_{>0}^K$.
The extension to $\mathbf x\in\mathbb R_{\ge 0}^K$ follows from the monotone limit
$
\mathbf x=\lim_{\varepsilon\downarrow 0}(\mathbf x+\varepsilon\mathbf 1)
$
together with Lemma~\ref{lem:psi-props}.

For every edge $e\in E(\Xi^0)$ of colour $i=\mathscr C(e)$, the substitution process produces inside $\Xi^n$ a subgraph $H_e^{(n)}$ which is a fresh copy of the $n$-step substitution network starting from a single $i$-coloured edge.
The subgraph $H_e^{(n)}$ is attached to the rest of $\Xi^n$ only through the two endpoints of $e$.
Moreover,
\[
\Xi^n=\bigcup_{e\in E(\Xi^0)} H_e^{(n)},
\]
and different $H_e^{(n)}$ are edge-disjoint and intersect only at vertices in $V(\Xi^0)$.

Consider any unit flow $\theta$ on $\Xi^n$ from $\mathfrak v_+$ to $\mathfrak v_-$.
For each $e\in E(\Xi^0)$, let $I_e$ be the net current that $\theta$ sends through the two terminals of $H_e^{(n)}$.
Then $I=(I_e)_{e\in E(\Xi^0)}$ is a unit flow on $\Xi^0$ from $\mathfrak v_+$ to $\mathfrak v_-$.
Decompose the energy as
\[
\mathcal E_{\Xi^n,\mathbf x}(\theta)
=
\sum_{e\in E(\Xi^0)} \mathcal E_{H_e^{(n)},\mathbf x}(\theta|_{H_e^{(n)}}).
\]
For fixed $e\in E(\Xi^0)$, the minimal energy among flows on $H_e^{(n)}$ that send net current $I_e$ equals
$
I_e^2 \psi_{H_e^{(n)}}(\mathbf x).
$
Hence
\[
\mathcal E_{\Xi^n,\mathbf x}(\theta)
\ge
\sum_{e\in E(\Xi^0)} I_e^2 \psi_{H_e^{(n)}}(\mathbf x).
\]
Minimising over $\theta$ shows that $\psi_{\Xi^n:(\mathfrak v_+,\mathfrak v_-)}(\mathbf x)$ equals the effective resistance in $\Xi^0$ where each edge $e$ is assigned resistance $\psi_{H_e^{(n)}}(\mathbf x)$.

For each colour $i\in[K]$, pick one $i$-coloured edge $e_i\in E(\Xi^0)$ and define $\mathbf y_n(\mathbf x)\in\mathbb R_{\ge 0}^K$ by
\[
[\mathbf y_n(\mathbf x)]_i:=\psi_{H_{e_i}^{(n)}}(\mathbf x).
\]
Since $\psi_{H_e^{(n)}}(\mathbf x)$ depends only on the colour of $e$, we obtain
\[
\psi_{\Xi^n:(\mathfrak v_+,\mathfrak v_-)}(\mathbf x)
=
\psi_{\Xi^0:(\mathfrak v_+,\mathfrak v_-)}\bigl(\mathbf y_n(\mathbf x)\bigr).
\]

Now $\mathbf y_1(\mathbf x)=\Psi(\mathbf x)$ by definition.
Applying the same reduction inside each substituted copy shows that
$
\mathbf y_{n+1}(\mathbf x)=\Psi(\mathbf y_n(\mathbf x)).
$
Thus $\mathbf y_n(\mathbf x)=\Psi^n(\mathbf x)$ for every $n$, which proves the claim.
\end{proof}

\begin{remark}
Lemma~\ref{lem:renorm} can also be proved directly from Kirchhoff's laws.
Each substituted copy of a rule graph is a two-terminal subnetwork attached only at planting vertices.
Since we make no further identifications beyond planting vertices, each copy may be replaced by a single super-edge whose resistance equals the effective resistance between the terminals.
\end{remark}

Naturally, we next investigate the dynamical properties of the iterations of $\Psi$.
In particular, the structural properties listed in Lemma~\ref{lem:psi-props} place the problem in a particularly tractable setting.
The work \cite{gaubert2004perronfrobenius} gives uniqueness of the eigenvalue under monotonicity, homogeneity, and Gaubert--Gunawardena irreducibility, although eigenvectors need not be unique.
Under the stronger hypotheses of \cite{nussbaum1986convexity,lemmens2012nonlinear}, both the eigenvalue and the eigenvector are unique.

\begin{lemma}\label{thm:sdprimitive_eigenpair}
Recall that the underlying EIGS in Assumption~\ref{ass:standing} is distance-positive.
Then there exist unique $\rho(\Psi)>0$ and unique $\mathbf v\in(0,\infty)^K$ up to scalar multiples such that
\[
\Psi(\mathbf v)=\rho(\Psi)\mathbf v.
\]
\end{lemma}

\begin{proof}
By Definition~\ref{def:psi_Psi}, for each $i\in[K]$ we have
$
[\Psi(\mathbf x)]_i=\psi_{R_i:(\beta_i^+,\beta_i^-)}(\mathbf x).
$
Therefore Lemma~\ref{lem:psi-props} implies that $\Psi$ is order-preserving, homogeneous of degree $1$, and superadditive on $(0,\infty)^K$.
Moreover, each coordinate admits the extension from $(0,\infty)^K$ to $\mathbb R_{\ge 0}^K$ given by the limit $\varepsilon\downarrow 0$ as in Lemma~\ref{lem:psi-props}, hence $\Psi$ admits the same extension.

In \cite{nussbaum1986convexity}, a map $F:\mathbb R_{\ge 0}^K\to\mathbb R_{\ge 0}^K$ is called power-bounded below if for each $i\in[K]$ there exist a constant $c_i>0$ and a stochastic vector $\mathbf u(i)\in\mathbb R_{\ge 0}^K$ with $\sum_{j=1}^K [\mathbf u(i)]_j=1$ such that
\[
[F(\mathbf x)]_i
\ge
c_i \prod_{k=1}^K [\mathbf x]_k^{[\mathbf u(i)]_k}
\]
for all $\mathbf x\in(0,\infty)^K$.
A nonnegative matrix $\mathbf A=(a_{ij})_{i,j\in[K]}$ is called an incidence matrix for $F$ if, whenever $a_{ij}>0$, there exist a constant $c_{ij}>0$ and a stochastic vector $\mathbf u(i,j)\in\mathbb R_{\ge 0}^K$ with $\sum_{k=1}^K [\mathbf u(i,j)]_k=1$ and $[\mathbf u(i,j)]_j>0$ such that
\[
[F(\mathbf x)]_i
\ge
c_{ij} \prod_{k=1}^K [\mathbf x]_k^{[\mathbf u(i,j)]_k}
\]
for all $\mathbf x\in(0,\infty)^K$.
We say that $F$ is Nussbaum-irreducible if it admits an incidence matrix $\mathbf A$ which is irreducible in the Perron--Frobenius sense, that is, for every $p,q\in[K]$ there exists $m\ge 1$ such that $[\mathbf A^m]_{p,q}>0$.

Fix $i,j\in[K]$.
Distance-positivity means that every simple path in the underlying undirected graph of $R_i$ connecting $\beta_i^+$ and $\beta_i^-$ contains at least one $j$-coloured edge.
Equivalently, if we delete all $j$-coloured edges from $R_i$, then $\beta_i^+$ and $\beta_i^-$ become disconnected.

Let $R_i^{(j)}$ be the graph obtained from the underlying undirected graph of $R_i$ by deleting all $j$-coloured edges.
Let $S_{ij}\subseteq V(R_i)$ be the set of vertices reachable from $\beta_i^+$ in $R_i^{(j)}$.
Then $\beta_i^-\notin S_{ij}$.
Define the edge boundary
\[
C_{ij}:=\big\{\{a,b\}\in E(R_i): a\in S_{ij},\ b\notin S_{ij}\big\},
\]
which is non-empty.
Furthermore, every edge in $C_{ij}$ must be $j$-coloured.
Indeed, if $\{a,b\}\in C_{ij}$ had colour different from $j$, then it would still be present in $R_i^{(j)}$, and since $a\in S_{ij}$ this would imply $b\in S_{ij}$, a contradiction.

Fix $\mathbf x\in(0,\infty)^K$.
Consider the resistance network on the underlying undirected graph of $R_i$ where each $k$-coloured edge has resistance $[\mathbf x]_k$.
Short all vertices in $S_{ij}$ to a single node, and short all vertices in $V(R_i)\setminus S_{ij}$ to a single node.
By Rayleigh monotonicity, shorting cannot increase the effective resistance between $\beta_i^+$ and $\beta_i^-$.
After shorting, the resulting two-terminal network consists of $|C_{ij}|$ parallel edges, each of resistance $[\mathbf x]_j$.
Hence
\[
[\Psi(\mathbf x)]_i
=
\Reff_{R_i}(\beta_i^+,\beta_i^-)
\ge
\frac{[\mathbf x]_j}{|C_{ij}|}.
\]

Let $\mathbf e_j$ denote the $j$-th standard basis vector in $\mathbb R^K$.
Then $\mathbf e_j$ is a stochastic vector and
$
[\mathbf x]_j=\prod_{k=1}^K [\mathbf x]_k^{[\mathbf e_j]_k}.
$
The previous estimate yields
\[
[\Psi(\mathbf x)]_i
\ge
\frac{1}{|C_{ij}|}\prod_{k=1}^K [\mathbf x]_k^{[\mathbf e_j]_k}
\]
for all $\mathbf x\in(0,\infty)^K$.
In particular, for each $i$ we may take $c_i=1/|C_{ii}|$ and $\mathbf u(i)=\mathbf e_i$, so $\Psi$ is power-bounded below.

Now define the matrix $\mathbf A=(a_{ij})$ by $a_{ij}=1$ for all $i,j\in[K]$.
For each $i,j$ the above inequality provides an admissible choice of $c_{ij}$ and $\mathbf u(i,j)=\mathbf e_j$ with $[\mathbf u(i,j)]_j=1>0$.
Therefore $\mathbf A$ is an incidence matrix for $\Psi$, and it is irreducible since $\mathbf A$ is positive.
Hence $\Psi$ is Nussbaum-irreducible.

We have shown that $\Psi$ is homogeneous of degree $1$, superadditive, power-bounded below, and Nussbaum-irreducible.
Therefore Theorem~4.1 in \cite{nussbaum1986convexity} yields a unique positive eigenvector $\mathbf v\in(0,\infty)^K$ up to scalar multiples and a unique eigenvalue $\rho(\Psi)>0$ such that
\[
\Psi(\mathbf v)=\rho(\Psi)\mathbf v.
\]
\end{proof}

\begin{corollary}\label{lem:Psi_range}
There exist constants $c_{\Psi},C_{\Psi}>0$ such that for every $\ell\in\mathbb N$ and every colour $i\in[K]$,
\[
c_{\Psi} \rho(\Psi)^\ell
\le
\big[\Psi^\ell(\mathbf 1)\big]_i
\le
C_{\Psi} \rho(\Psi)^\ell.
\]
\end{corollary}

\begin{proof}
Fix a positive eigenpair $\Psi(\mathbf v)=\rho(\Psi)\mathbf v$ with $\mathbf v\in(0,\infty)^K$.
Choose $a,b>0$ such that $a\mathbf v\le \mathbf 1\le b\mathbf v$.
By monotonicity and homogeneity of $\Psi$, for every $\ell\in\mathbb N$,
$
a \rho(\Psi)^\ell \mathbf v
\le
\Psi^\ell(\mathbf 1)
\le
b \rho(\Psi)^\ell \mathbf v.
$
Taking coordinates and setting
$
c_{\Psi}:=a\min_{j\in[K]}[\mathbf v]_j,
$
and
$
C_{\Psi}:=b\max_{j\in[K]}[\mathbf v]_j
$
gives the claim.
\end{proof}

\begin{corollary}\label{cor:Psi_strong_convergence}
Let $(\rho(\Psi),\mathbf v)$ be the unique positive eigenpair with $\Psi(\mathbf v)=\rho(\Psi)\mathbf v$.
Fix the normalisation $\|\mathbf v\|_1=1$.
Then there exists a constant $c_{\Psi}^*>0$ such that
\[
\lim_{n\to\infty}\frac{\Psi^n(\mathbf 1)}{\rho(\Psi)^n}=c_{\Psi}^*\mathbf v.
\]
\end{corollary}

\begin{proof}
Set
$
\mathbf x_n:=\frac{\Psi^n(\mathbf 1)}{\|\Psi^n(\mathbf 1)\|_1}.
$
By Theorem~4.2 in \cite{nussbaum1986convexity}, the sequence $(\mathbf x_n)$ converges geometrically to $\mathbf v$ in Hilbert's projective metric.
By Corollary~\ref{lem:Psi_range}, the sequence $(\mathbf x_n)$ stays in a compact subset of the positive simplex, so on that set Hilbert's projective metric is equivalent to the Euclidean norm.
Hence there exist constants $C>0$ and $\eta\in(0,1)$ such that
$
\|\mathbf x_n-\mathbf v\|_1\le C\eta^n.
$

Now set
$
b_n:=\frac{\|\Psi^n(\mathbf 1)\|_1}{\rho(\Psi)^n}.
$
Then
$
\frac{b_{n+1}}{b_n}
=
\frac{\|\Psi(\mathbf x_n)\|_1}{\rho(\Psi)}.
$
Since $\Psi$ is continuous and concave on the positive simplex, the map $\mathbf x\mapsto \|\Psi(\mathbf x)\|_1$ is Lipschitz on the above compact set.
Therefore
\[
\left|\frac{b_{n+1}}{b_n}-1\right|
=
\frac{\bigl|\|\Psi(\mathbf x_n)\|_1-\|\Psi(\mathbf v)\|_1\bigr|}{\rho(\Psi)}
\le
C\eta^n.
\]
Thus
$
\sum_{n=0}^\infty \bigl|\log(b_{n+1}/b_n)\bigr|<\infty,
$
so $(\log b_n)$ converges and hence $b_n\to c_{\Psi}^*>0$ for some constant $c_{\Psi}^*$.

Finally,
\[
\frac{\Psi^n(\mathbf 1)}{\rho(\Psi)^n}
=
b_n\mathbf x_n
\longrightarrow
c_{\Psi}^*\mathbf v.
\]
\end{proof}

\begin{lemma}\label{lem:Reff-asymp-rho}
Let $\Xi^0$ be a single directed edge of colour $\iota\in[K]$ with endpoints $\mathfrak v_+$ and $\mathfrak v_-$ as in Assumption~\ref{ass:standing}.
Then
\[
c_{\Psi} \rho(\Psi)^n
\le
\Reff_{\Xi^n}(\mathfrak v_+,\mathfrak v_-)
\le
C_{\Psi} \rho(\Psi)^n.
\]
\end{lemma}

\begin{proof}
Since $\Xi^0$ is a single $\iota$-coloured edge, we have
$
\psi_{\Xi^0:(\mathfrak v_+,\mathfrak v_-)}(\mathbf x)=[\mathbf x]_\iota
$
for all $\mathbf x\in\mathbb R_{\ge 0}^K$.
By Lemma~\ref{lem:renorm},
\[
\Reff_{\Xi^n}(\mathfrak v_+,\mathfrak v_-)
=
\psi_{\Xi^n:(\mathfrak v_+,\mathfrak v_-)}(\mathbf 1)
=
\psi_{\Xi^0:(\mathfrak v_+,\mathfrak v_-)}\big(\Psi^n(\mathbf 1)\big)
=
[\Psi^n(\mathbf 1)]_\iota.
\]
The claim now follows directly from Corollary~\ref{lem:Psi_range}.
\end{proof}

Lemma~\ref{lem:Reff-asymp-rho} is the main technical input of our estimates.
We will show that the growth rate $\rho(\Psi)^n$ is not specific to the two distinguished endpoints, but in fact describes a uniform resistance growth speed shared across the whole graph.

\begin{definition}\label{def:resistance_dimension}
Assume that $(G^n)_{n\in\mathbb N}$ is vertex-nested and has a combinatorial limit $G^\infty$.
If for every pair of distinct vertices $u,v\in V(G^0)$ the limit
\[
\dimresis(G^\infty):=\lim_{n\to\infty}\frac{\log \Reff_{G^n}(u,v)}{\log d_{G^n}(u,v)}
\]
exists and is independent of the choice of $u,v$,
then $\dimresis(G^\infty)$ is called the \emph{resistance dimension}.
\end{definition}

\begin{theorem}\label{thm:resistance_dimension_formula}
The resistance dimension of $\Xi$ exists and satisfies
\[
\dimresis(\Xi)=\frac{\log \rho(\Psi)}{\log \rho_{\min}(\mathcal D)}.
\]
\end{theorem}

\begin{proof}
Fix distinct vertices $u,v\in V(\Xi^0)$.
Exactly as in Lemma~\ref{lem:renorm}, but with terminals $u$ and $v$, collapsing each substituted copy in $\Xi^n$ to a super-edge yields
$
\Reff_{\Xi^n}(u,v)
=
\psi_{\Xi^0:(u,v)}\bigl(\Psi^n(\mathbf 1)\bigr).
$
By Corollary~\ref{lem:Psi_range},
$
c_{\Psi}\rho(\Psi)^n\mathbf 1
\le
\Psi^n(\mathbf 1)
\le
C_{\Psi}\rho(\Psi)^n\mathbf 1.
$
Using monotonicity and homogeneity of $\psi_{\Xi^0:(u,v)}$, we obtain
\[
c_{\Psi}\psi_{\Xi^0:(u,v)}(\mathbf 1)\rho(\Psi)^n
\le
\Reff_{\Xi^n}(u,v)
\le
C_{\Psi}\psi_{\Xi^0:(u,v)}(\mathbf 1)\rho(\Psi)^n.
\]
Hence there exist constants $c_1,c_2>0$ such that
$
c_1\rho(\Psi)^n
\le
\Reff_{\Xi^n}(u,v)
\le
c_2\rho(\Psi)^n.
$

By Theorem~3.3 in \cite{neroli2024fractal}, there exist constants $c_3,c_4>0$ such that
$
c_3\rho_{\min}(\mathcal D)^n
\le
d_{\Xi^n}(u,v)
\le
c_4\rho_{\min}(\mathcal D)^n.
$

Taking logarithms and dividing gives
\[
\frac{\log c_1+n\log\rho(\Psi)}{\log c_4+n\log\rho_{\min}(\mathcal D)}
\le
\frac{\log \Reff_{\Xi^n}(u,v)}{\log d_{\Xi^n}(u,v)}
\le
\frac{\log c_2+n\log\rho(\Psi)}{\log c_3+n\log\rho_{\min}(\mathcal D)}.
\]
Since distance-positivity implies $\rho_{\min}(\mathcal D)>1$, both denominators diverge as $n\to\infty$.
Letting $n\to\infty$ yields
\[
\lim_{n\to\infty}\frac{\log \Reff_{\Xi^n}(u,v)}{\log d_{\Xi^n}(u,v)}
=
\frac{\log \rho(\Psi)}{\log \rho_{\min}(\mathcal D)}.
\]
The same argument works for any fixed distinct vertices $u,v\in V(\Xi^k)$.
\end{proof}

Theorem~\ref{thm:resistance_dimension_formula} shows that any two vertices in an EIGS graph share the same resistance dimension with respect to the natural metric.
This result will play a crucial role later, as it provides the foundation for the Einstein relation, namely that the walk dimension equals the Minkowski dimension plus the resistance dimension.
Moreover, this theorem gives a closed-form expression for the resistance dimension of any EIGS.

Before going further, we present a concrete example.

\begin{example}\label{example:two_colour_1}
We now compute the Perron eigenpair of $\Psi$ for Figure~\ref{fig:two_colour}$.$
Since $\Psi$ is homogeneous of degree $1$, we may normalise an eigenvector by setting its second coordinate to be $1$.
Let $\mathbf v=(t,1)$ with $t>0$.
The eigen-relation $\Psi(\mathbf v)=\rho(\Psi)\mathbf v$ is equivalent to
\[
\Psi_2(t,1)=\rho(\Psi),
\qquad
\Psi_1(t,1)=t\Psi_2(t,1).
\]
Eliminating $\rho(\Psi)$ gives the polynomial equation
$
3t^4+10t^3+4t^2-8t-3=0.
$
This polynomial has a unique positive root
$
t\approx 0.8200.
$
Substituting this value into $\rho(\Psi)=\Psi_2(t,1)$ yields
\[
\rho(\Psi)
=
\frac{(t+2)(t+1)}{2t+3}
\approx
1.1061.
\]

Equivalently, eliminating $t$ yields a polynomial equation for $\rho(\Psi)$.
One finds that $\rho(\Psi)$ is the unique root in $(1,\infty)$ of
\[
27\lambda^4+6\lambda^3-50\lambda^2+6\lambda+6=0.
\]
For a numerical check, start from $\mathbf x^{(0)}=(1,1)$ and iterate $\mathbf x^{(n+1)}=\Psi(\mathbf x^{(n)})$.
Then
\[
\mathbf x^{(1)}=(1,6/5),
\qquad
\mathbf x^{(2)}\approx (1.09697,1.33571),
\qquad
\mathbf x^{(3)}\approx (1.21244,1.47834),
\]
and the ratios $[\mathbf x^{(n+1)}]_2/[\mathbf x^{(n)}]_2$ converge rapidly to $\rho(\Psi)\approx 1.10613$.

By Definition~2.7 in \cite{neroli2024fractal} and Definition~5.3 in \cite{li2024on},
\[
\mathbf M=
\begin{pmatrix}
3 & 3\\
2 & 3
\end{pmatrix},
\qquad
\mathbf N=
\begin{pmatrix}
1&0&1&0\\
0&1&0&1\\
0&0&2&0\\
0&1&0&1
\end{pmatrix},
\]
and
\[
\mathcal D=
\left\{
\begin{pmatrix}1&1\\ 1&1\end{pmatrix},
\begin{pmatrix}1&1\\ 1&2\end{pmatrix},
\begin{pmatrix}3&1\\ 1&1\end{pmatrix},
\begin{pmatrix}3&1\\ 1&2\end{pmatrix},
\begin{pmatrix}1&3\\ 1&1\end{pmatrix},
\begin{pmatrix}1&3\\ 1&2\end{pmatrix}
\right\}.
\]
We have
\[
\rho(\mathbf M)=3+\sqrt{6},
\qquad
\rho(\mathbf N)=2,
\qquad
\rho_{\min}(\mathcal D)=2.
\]

By Theorem~2.1 in \cite{neroli2024fractal} and Theorem~5.5 in \cite{li2024on},
\[
\dimbox(\Xi)
=
\frac{\log\rho(\mathbf M)}{\log\rho_{\min}(\mathcal D)}
=
\frac{\log(3+\sqrt{6})}{\log 2}
\approx 2.4461,
\]
and
\[
\dimdeg(\Xi)
=
\frac{\log\rho(\mathbf M)}{\log\rho(\mathbf N)}
=
\frac{\log(3+\sqrt{6})}{\log 2}
\approx 2.4461.
\]
By Theorem~\ref{thm:resistance_dimension_formula},
\[
\dimresis(\Xi)
=
\frac{\log \rho(\Psi)}{\log \rho_{\min}(\mathcal D)}
=
\frac{\log 1.1061}{\log 2}
\approx 0.14552.
\]
\end{example}

\subsection{Walk dimension}\label{subsec:walk_dimension}

In this subsection, our goal is to determine the walk dimension of the simple random walk on the combinatorial limit.
This will be an essential input for our subsequent discussion of diffusion limits.
To achieve this, we first establish local mass estimates and local effective resistance estimates, and then deduce the walk dimension.
In addition, we discuss a notion of recurrence for the combinatorial limit in the locally infinite setting.

\begin{lemma}[\cite{li2024on,neroli2024fractal}]\label{lem:basic_prop}
There exist constants $c_{\mathrm V},C_{\mathrm V},c_{\mathrm E},C_{\mathrm E}>0$ such that for every $\iota\in[K]$ and every $m\in\mathbb N$,
\[
c_{\mathrm V}\rho(\mathbf M)^m \le |V(\Xi_\iota^m)| \le C_{\mathrm V}\rho(\mathbf M)^m,
\qquad
c_{\mathrm E}\rho(\mathbf M)^m \le |E(\Xi_\iota^m)| \le C_{\mathrm E}\rho(\mathbf M)^m.
\]
Moreover, there exist constants $c_{\dist},C_{\dist}>0$ such that for every $\iota\in[K]$ and every $m\in\mathbb N$,
\[
c_{\dist}\rho_{\min}(\mathcal D)^m \le d_{\Xi_\iota^m}(\mathfrak v_+,\mathfrak v_-)
\le
\diam(\Xi_\iota^m)\le C_{\dist}\rho_{\min}(\mathcal D)^m.
\]
Write $V_k(\Xi^n):=V(\Xi^k)\setminus V(\Xi^{k-1})\subset V(\Xi^n)$.
Let $v_k^n$ be a $k$-born vertex in $\Xi^n$, that is, $v_k^n\in V_k(\Xi^n)$.
Finally, there exist constants $c_{\deg},C_{\deg}>0$, depending only on the birth data of $v_k^n$ but not on $n$ or $k$, such that for every $n\ge k\in\mathbb N$,
\[
c_{\deg}\rho(\mathbf N)^{n-k}
\le
\deg_{\Xi^{n}}(v_k^n)
\le
C_{\deg}\rho(\mathbf N)^{n-k}.
\]
\end{lemma}

\subsubsection*{Local mass dimension}

We first establish local mass estimates at each vertex.
These are among the most basic geometric features of the graphs and form the foundation for our later arguments.

Fix once and for all an integer $a_0\in\mathbb N$ so large that $C_{\dist}\rho_{\min}(\mathcal D)^{-a_0}\le c_{\dist}$.

\begin{lemma}\label{lem:ball_volume_k_born}
There exist constants $c_{\mathrm{lm}},C_{\mathrm{lm}}>0$, independent of $n,m,r$, such that
\[
c_{\mathrm{lm}} \rho(\mathbf N)^{n-k-m} \rho(\mathbf M)^m
\le
\big|V\big(B_{\Xi^n}(v_k^n,r)\big)\big|
\le
C_{\mathrm{lm}} \rho(\mathbf N)^{n-k-m} \rho(\mathbf M)^m,
\]
for $r>1$ satisfying
$
c_{\dist} \rho_{\min}(\mathcal D)^m
\le
r
\le
C_{\dist} \rho_{\min}(\mathcal D)^m
$
for some integer $m\le n-k{}-a_0$.
\end{lemma}

\begin{proof}
Because $m+a_0\le n-k$, we have $v_k^n\in V(\Xi^{n-m-a_0})\subseteq V(\Xi^{n-m})$.

\textbf{Upper bound.}
Consider the coarse graph $\Xi^{n-m-a_0}$ and view $v_k^n$ as the same vertex in $V(\Xi^{n-m-a_0})$.
For each edge $e$ of $\Xi^{n-m-a_0}$ incident to $v_k^n$, let $H_e^{(m+a_0)}$ be the $(m+a_0)$-step substituted copy that replaces $e$ inside $\Xi^n$.
Different copies $H_e^{(m+a_0)}$ intersect only at vertices of $\Xi^{n-m-a_0}$, and in particular two distinct copies incident to $v_k^n$ intersect only at $v_k^n$.

Let $b_e$ be the endpoint of $e$ different from $v_k^n$.
Any path in $\Xi^n$ from $v_k^n$ to a vertex outside $\bigcup_{e}V(H_e^{(m+a_0)})$ must hit some $b_e$.
Since
$
d_{\Xi^n}(v_k^n,b_e)\ge c_{\dist}\rho_{\min}(\mathcal D)^{m+a_0}\ge C_{\dist}\rho_{\min}(\mathcal D)^{m}\ge r
$
by the choice of $a_0$, we have $b_e\notin B_{\Xi^n}(v_k^n,r)$.
Hence
\[
V\big(B_{\Xi^n}(v_k^n,r)\big)
\subseteq
\bigcup_{e\in E_{\Xi^{n-m-a_0}}(v_k^n)} V(H_e^{(m+a_0)}),
\]
where $E_{\Xi^{n-m-a_0}}(v_k^n)$ denotes the set of edges of $\Xi^{n-m-a_0}$ incident to $v_k^n$.

Each $H_e^{(m+a_0)}$ is isomorphic to $\Xi_{\mathscr C(e)}^{m+a_0}$, so $|V(H_e^{(m+a_0)})|\le C_{\mathrm V}\rho(\mathbf M)^{m+a_0}$ by Lemma~\ref{lem:basic_prop}.
Using disjointness outside $v_k^n$ and the degree estimate at level $n-m-a_0$, we obtain
\[
\big|V\big(B_{\Xi^n}(v_k^n,r)\big)\big|
\le
C_{\mathrm V}\deg_{\Xi^{n-m-a_0}}(v_k^n)\rho(\mathbf M)^{m+a_0}
\le
C_{\mathrm V}C_{\deg}\rho(\mathbf M)^{a_0}\,\rho(\mathbf N)^{n-k-m}\rho(\mathbf M)^m,
\]
where we used $\rho(\mathbf N)^{n-k-m-a_0}\le\rho(\mathbf N)^{n-k-m}$.

\textbf{Lower bound.}
Recall the integer $a_0$ fixed before the statement of the lemma.

If $m\ge a_0$, then for each $e\in E_{\Xi^{n-m}}(v_k^n)$ we construct a subgraph $\Lambda_e\subset H_e^{(m)}$ as follows.
Inside the first substitution step of $H_e^{(m)}$, pick an edge incident to the endpoint corresponding to $v_k^n$, and iterate this choice for $a_0$ substitution steps.
This produces an $(m-a_0)$-step substituted copy of a single coloured edge attached at $v_k^n$.
By the diameter bound,
$
\diam(\Lambda_e)\le C_{\dist}\rho_{\min}(\mathcal D)^{m-a_0}\le r,
$
so $\Lambda_e\subseteq B_{\Xi^n}(v_k^n,r)$.
Moreover, by the vertex-count bound,
\[
|V(\Lambda_e)|\ge c_{\mathrm V}\rho(\mathbf M)^{m-a_0}=c_{\mathrm V}\rho(\mathbf M)^{-a_0}\rho(\mathbf M)^m.
\]
Since $|V(\Lambda_e)|\ge 2$, we have $|V(\Lambda_e)|-1\ge |V(\Lambda_e)|/2$.
As the sets $V(\Lambda_e)\setminus\{v_k^n\}$ are disjoint as $e$ varies, we obtain
\[
\big|V\big(B_{\Xi^n}(v_k^n,r)\big)\big|
\ge
1+\sum_{e\in E_{\Xi^{n-m}}(v_k^n)} \bigl(|V(\Lambda_e)|-1\bigr)
\ge
\frac{c_{\mathrm V}}{2}\rho(\mathbf M)^{-a_0}\deg_{\Xi^{n-m}}(v_k^n)\rho(\mathbf M)^m.
\]
Applying the degree estimate again yields
\[
\big|V\big(B_{\Xi^n}(v_k^n,r)\big)\big|
\ge
\frac{c_{\mathrm V}}{2}\rho(\mathbf M)^{-a_0}c_{\deg}\rho(\mathbf N)^{n-k-m}\rho(\mathbf M)^m.
\]

If $m<a_0$, then $m$ ranges over only finitely many values and $r>1$.
Hence each edge of $\Xi^{n-m}$ incident to $v_k^n$ contributes at least one distinct neighbour of $v_k^n$ inside $B_{\Xi^n}(v_k^n,r)$.
Therefore
\[
\big|V\big(B_{\Xi^n}(v_k^n,r)\big)\big|
\ge
1+\deg_{\Xi^{n-m}}(v_k^n)
\ge
c_{\deg}\rho(\mathbf N)^{n-k-m}.
\]
Since $m<a_0$, the factor $\rho(\mathbf M)^m$ is uniformly bounded above, so after decreasing the lower constant if necessary we again obtain
$
\big|V\big(B_{\Xi^n}(v_k^n,r)\big)\big|
\ge
c_{\mathrm{lm}} \rho(\mathbf N)^{n-k-m} \rho(\mathbf M)^m.
$

Combining the two cases proves the lemma.
\end{proof}

The local mass dimension conveys an important message: even for highly regular EIGS, scale-freeness can force the local and global dimensions to differ.
This phenomenon is already discussed in \cite{hambly2010diffusion}, where it is interpreted as a failure of the volume-doubling property.

\begin{definition}\label{def:local_mass_dimension_constant}
Fix a constant $r_*>1$.
We define the local mass dimension at $v\in V(G^\infty)$ by
\[
\dim_{\mathrm{M,loc}}(G^\infty:v)
:=
\lim_{n\to\infty}
\frac{\log\big|V\big(B_{G^n}(v,r_*)\big)\big|}{-\log \frac{r_*}{\diam(G^n)}},
\]
provided the limit exists and is independent of the choice of $r_*$.
\end{definition}

\begin{proposition}\label{prop:local_mass_dimension_constant}
For any $v\in V(\Xi)$, if $\rho(\mathbf N)>1$ then
\[
\dim_{\mathrm{M,loc}}(\Xi:v)=\frac{\log \rho(\mathbf N)}{\log \rho_{\min}(\mathcal D)}.
\]
\end{proposition}

\begin{proof}
Fix $v\in V(\Xi^k)$.
Choose an integer $m_*\in\mathbb N$ such that
$
c_{\dist}\rho_{\min}(\mathcal D)^{m_*}
\le
r_*
\le
C_{\dist}\rho_{\min}(\mathcal D)^{m_*}.
$
By Lemma~\ref{lem:ball_volume_k_born}, there exist constants $c_{\mathrm{lm}},C_{\mathrm{lm}}>0$, independent of $n$, such that for all $n\ge k+m_*{}+a_0$,
\[
c_{\mathrm{lm}}\rho(\mathbf N)^{n-k-m_*}\rho(\mathbf M)^{m_*}
\le
\big|V\big(B_{\Xi^n}(v,r_*)\big)\big|
\le
C_{\mathrm{lm}}\rho(\mathbf N)^{n-k-m_*}\rho(\mathbf M)^{m_*}.
\]
Since $m_*$ is fixed, this is equivalent to
$
c_{\mathrm V}\rho(\mathbf N)^{n-k}
\le
\big|V\big(B_{\Xi^n}(v,r_*)\big)\big|
\le
C_{\mathrm V}\rho(\mathbf N)^{n-k}
$
for suitable constants $c_{\mathrm V},C_{\mathrm V}>0$.

By the diameter bounds,
\[
\log\frac{\diam(\Xi^n)}{r_*}
=
n\log\rho_{\min}(\mathcal D)+O(1),
\]
where the $O(1)$ term is independent of $n$.
Moreover,
$
\log\big|V\big(B_{\Xi^n}(v,r_*)\big)\big|
=
(n-k)\log\rho(\mathbf N)+O(1),
$
where the $O(1)$ term is independent of $n$.
Dividing and letting $n\to\infty$ gives
\[
\dim_{\mathrm{M,loc}}(\Xi:v)
=
\frac{\log \rho(\mathbf N)}{\log \rho_{\min}(\mathcal D)}.
\]
\end{proof}

Here we observe an apparently paradoxical phenomenon.
We know that the Minkowski dimension and the Hausdorff dimension of $\Xi$ are both
\[
\frac{\log \rho(\mathbf M)}{\log \rho_{\min}(\mathcal D)}.
\]
However, since $\rho(\mathbf N)<\rho(\mathbf M)$, the local mass dimension at every finite-born vertex is strictly smaller than the Minkowski dimension, while the Minkowski dimension and the Hausdorff dimension of the whole space are larger.
This is again interpreted in parts of the literature as a failure of the volume-doubling property \cite{hambly2010diffusion}.
This is a classical feature of scale-free combinatorial limit graphs.

In a later section we will show that, in the Gromov--Hausdorff--Prokhorov scaling limit, $\mu$-almost every point has local mass dimension equal to the Minkowski dimension.

\subsubsection*{Local resistance dimension}

In what follows, we derive local resistance estimates around arbitrary vertices.
As with the local mass dimension, the local resistance dimension exhibits a similar non-uniformity on scale-free EIGS.

Fix once and for all an integer $c_0\in\mathbb N$ so large that
\[
C_{\dist}\rho_{\min}(\mathcal D)^{-c_0}\le \frac{c_{\dist}}{4\rho_{\min}(\mathcal D)}
\qquad\text{and}\qquad
c_{\dist}\rho_{\min}(\mathcal D)^{c_0}\ge 4C_{\dist}.
\]

\begin{lemma}\label{lem:mesoscopic_Reff_k_born}
There exist constants $c_{\mathscr R},C_{\mathscr R}>0$, depending on the birth data of $v_k^n$ but independent of $n,m,r$, such that
\[
c_{\mathscr R} \frac{\rho(\Psi)^m}{\rho(\mathbf N)^{n-k-m}}
\le
\Reff_{\Xi^n}\bigl(v_k^n,B_{\Xi^n}(v_k^n,r)^{\mathrm c}\bigr)
\le
C_{\mathscr R} \frac{\rho(\Psi)^m}{\rho(\mathbf N)^{n-k-m}},
\]
for $r>1$ satisfying
$
c_{\dist} \rho_{\min}(\mathcal D)^m
\le
r
\le
C_{\dist} \rho_{\min}(\mathcal D)^m
$
for some integer $m$ with $c_0\le m\le n-k-c_0$.
\end{lemma}

\begin{proof}
We work with the underlying undirected simple graph of $\Xi^n$ and assign unit resistances to all edges.

Set
$
t_-:=n-m+c_0
$
and
$
t_+:=n-m-c_0.
$
Since $c_0\le m\le n-k-c_0$, we have $k\le t_+\le t_-\le n$.
Let $E_\pm(v_k^n)$ be the set of edges of $\Xi^{t_\pm}$ incident to $v_k^n$.

For each $e\in E_-(v_k^n)$, the substitution from level $t_-$ to level $n$ replaces $e$ by a fresh copy of an $(m-c_0)$-step edge substitution network, denoted by $H_e^-$.
Write $w_e^-$ for the other endpoint of $e$ in $\Xi^{t_-}$, viewed as a vertex of $\Xi^n$.
For each $e\in E_+(v_k^n)$, the substitution from level $t_+$ to level $n$ replaces $e$ by a fresh copy of an $(m+c_0)$-step edge substitution network, denoted by $H_e^+$.
Write $w_e^+$ for the other endpoint of $e$ in $\Xi^{t_+}$, viewed as a vertex of $\Xi^n$.

By construction, the families $(H_e^-)_{e\in E_-(v_k^n)}$ and $(H_e^+)_{e\in E_+(v_k^n)}$ are pairwise edge-disjoint, and different members intersect only at the vertex $v_k^n$.
Moreover, each $H_e^\pm$ is attached to the rest of $\Xi^n$ only at its terminals $v_k^n$ and $w_e^\pm$.

\textbf{Lower bound.}
By the diameter bound and the choice of $c_0$, for every $e\in E_-(v_k^n)$,
\[
\sup_{x\in V(H_e^-)} d_{\Xi^n}(v_k^n,x)
\le
\diam(H_e^-)
\le
C_{\dist}\rho_{\min}(\mathcal D)^{m-c_0}
\le
r.
\]
Hence $V(H_e^-)\subseteq B_{\Xi^n}(v_k^n,r)$.
Define
\[
U_-:=
\Big(\bigcup_{e\in E_-(v_k^n)} V(H_e^-)\Big)
\setminus
\{w_e^-:e\in E_-(v_k^n)\}.
\]
Then $U_-\subseteq B_{\Xi^n}(v_k^n,r)$, so $U_-^{\mathrm c}\supseteq B_{\Xi^n}(v_k^n,r)^{\mathrm c}$.
By monotonicity in the grounded set,
$
\Reff_{\Xi^n}\bigl(v_k^n,B_{\Xi^n}(v_k^n,r)^{\mathrm c}\bigr)
\ge
\Reff_{\Xi^n}(v_k^n,U_-^{\mathrm c}).
$

Let $\theta$ be any unit flow from $v_k^n$ to $U_-^{\mathrm c}$.
For each $e\in E_-(v_k^n)$, let $I_e$ be the net current of $\theta$ entering $H_e^-$ from $v_k^n$.
Then $\sum_{e\in E_-(v_k^n)} I_e=1$.
Since $U_-^{\mathrm c}$ contains $w_e^-$ and $H_e^-$ is attached to $\Xi^n$ only at $v_k^n$ and $w_e^-$, the restriction of $\theta$ to $H_e^-$ has energy at least
$
I_e^2 \Reff_{H_e^-}(v_k^n,w_e^-).
$
Hence
$
\mathcal E_{\Xi^n}(\theta,\theta)
\ge
\sum_{e\in E_-(v_k^n)} I_e^2 \Reff_{H_e^-}(v_k^n,w_e^-).
$
By the renormalisation identity,
$
\Reff_{H_e^-}(v_k^n,w_e^-)=\big[\Psi^{m-c_0}(\mathbf 1)\big]_{\mathscr C(e)},
$
and by Corollary~\ref{lem:Psi_range},
$
\Reff_{H_e^-}(v_k^n,w_e^-)\ge c_{\Psi}\rho(\Psi)^{m-c_0}.
$
Thus
\[
\mathcal E_{\Xi^n}(\theta,\theta)
\ge
c_{\Psi}\rho(\Psi)^{m-c_0}\sum_{e\in E_-(v_k^n)} I_e^2.
\]
By Cauchy--Schwarz,
\[
\sum_{e\in E_-(v_k^n)} I_e^2
\ge
\frac{1}{|E_-(v_k^n)|}
=
\frac{1}{\deg_{\Xi^{t_-}}(v_k^n)}.
\]
Taking the infimum over all unit flows and applying Thomson's principle yields
\[
\Reff_{\Xi^n}\bigl(v_k^n,B_{\Xi^n}(v_k^n,r)^{\mathrm c}\bigr)
\ge
\Reff_{\Xi^n}(v_k^n,U_-^{\mathrm c})
\ge
c_{\Psi}\frac{\rho(\Psi)^{m-c_0}}{\deg_{\Xi^{t_-}}(v_k^n)}.
\]
Using the degree estimate gives
\[
\Reff_{\Xi^n}\bigl(v_k^n,B_{\Xi^n}(v_k^n,r)^{\mathrm c}\bigr)
\ge
c_{\Psi} C_{\deg}^{-1} \rho(\Psi)^{-c_0} \rho(\mathbf N)^{c_0}
\frac{\rho(\Psi)^m}{\rho(\mathbf N)^{n-k-m}}.
\]

\textbf{Upper bound.}
Fix $e\in E_+(v_k^n)$.
Consider any path in $\Xi^n$ from $v_k^n$ to $w_e^+$.
Contracting each $(m+c_0)$-step substitution copy from level $t_+$ to level $n$ to a single edge yields the graph $\Xi^{t_+}$.
Hence any such path projects to a walk in $\Xi^{t_+}$ from $v_k^n$ to $w_e^+$, which must have length at least $1$.
By the endpoint-distance lower bound, traversing one $(m+c_0)$-step copy costs at least $c_{\dist}\rho_{\min}(\mathcal D)^{m+c_0}$ edges, and therefore
\[
d_{\Xi^n}(v_k^n,w_e^+)\ge c_{\dist}\rho_{\min}(\mathcal D)^{m+c_0}\ge 2r.
\]
In particular, $w_e^+\notin B_{\Xi^n}(v_k^n,r)$.
Set
$
S_+:=\{w_e^+:e\in E_+(v_k^n)\}
$
and
$
U_+:=V(\Xi^n)\setminus S_+.
$
Then $U_+^{\mathrm c}=S_+\subseteq B_{\Xi^n}(v_k^n,r)^{\mathrm c}$.
Hence
$
\Reff_{\Xi^n}\bigl(v_k^n,B_{\Xi^n}(v_k^n,r)^{\mathrm c}\bigr)
\le
\Reff_{\Xi^n}(v_k^n,U_+^{\mathrm c})
$
by monotonicity in the grounded set.

For each $e\in E_+(v_k^n)$, let $\theta_e$ be an energy-minimising unit flow in $H_e^+$ from $v_k^n$ to $w_e^+$.
Define
\[
\theta:=\frac{1}{\deg_{\Xi^{t_+}}(v_k^n)}\sum_{e\in E_+(v_k^n)} \theta_e,
\]
extended by $0$ outside $\bigcup_{e\in E_+(v_k^n)} H_e^+$.
Then $\theta$ is a unit flow from $v_k^n$ to $U_+^{\mathrm c}$.
By edge-disjointness,
\[
\mathcal E_{\Xi^n}(\theta,\theta)
=
\sum_{e\in E_+(v_k^n)}
\frac{1}{\deg_{\Xi^{t_+}}(v_k^n)^2}
\mathcal E_{H_e^+}(\theta_e,\theta_e).
\]
Since $\theta_e$ is energy-minimising,
\[
\mathcal E_{H_e^+}(\theta_e,\theta_e)
=
\Reff_{H_e^+}(v_k^n,w_e^+)
=
\big[\Psi^{m+c_0}(\mathbf 1)\big]_{\mathscr C(e)}.
\]
By Corollary~\ref{lem:Psi_range},
$
\mathcal E_{\Xi^n}(\theta,\theta)
\le
\frac{C_{\Psi}\rho(\Psi)^{m+c_0}}{\deg_{\Xi^{t_+}}(v_k^n)}.
$
Thomson's principle implies
\[
\Reff_{\Xi^n}\bigl(v_k^n,B_{\Xi^n}(v_k^n,r)^{\mathrm c}\bigr)
\le
\Reff_{\Xi^n}(v_k^n,U_+^{\mathrm c})
\le
\frac{C_{\Psi}\rho(\Psi)^{m+c_0}}{\deg_{\Xi^{t_+}}(v_k^n)}.
\]
Using the degree estimate once more gives
\[
\Reff_{\Xi^n}\bigl(v_k^n,B_{\Xi^n}(v_k^n,r)^{\mathrm c}\bigr)
\le
C_{\Psi} c_{\deg}^{-1}\rho(\Psi)^{c_0}\rho(\mathbf N)^{-c_0}
\frac{\rho(\Psi)^m}{\rho(\mathbf N)^{n-k-m}}.
\]
This completes the proof.
\end{proof}

\begin{definition}\label{def:local_resistance_dimension}
Fix $r_*>1$.
We define the local resistance dimension at $v\in V(G^\infty)$ by
\[
\dim_{\mathrm{R,loc}}(G^\infty:v)
:=
\lim_{n\to\infty}
\frac{\log \Reff_{G^n}\bigl(v,B_{G^n}(v,r_*)^{\mathrm c}\bigr)}
{-\log\frac{r_*}{\diam(G^n)}},
\]
provided the limit exists and is independent of the choice of $r_*$.
\end{definition}

\begin{proposition}\label{prop:local_resistance_dimension_k_born}
Let $v\in V(\Xi)$.
If $\rho(\mathbf N)>1$ then the local resistance dimension exists and satisfies
\[
\dim_{\mathrm{R,loc}}(\Xi:v)
=
-\frac{\log \rho(\mathbf N)}{\log \rho_{\min}(\mathcal D)}.
\]
\end{proposition}

\begin{proof}
Fix $v\in V(\Xi^k)$ and $r_*>1$.

\textbf{Lower bound.}
For each $n\ge k$, let $E_n(v)$ be the set of edges of $\Xi^n$ incident to $v$.
Since $r_*>1$, every neighbour of $v$ in $\Xi^n$ lies inside $B_{\Xi^n}(v,r_*)$.
Let $\theta$ be any unit flow from $v$ to $B_{\Xi^n}(v,r_*)^{\mathrm c}$.
If $I_e$ denotes the current through $e\in E_n(v)$, then
$
\sum_{e\in E_n(v)} I_e=1.
$
Hence
\[
\mathcal E_{\Xi^n}(\theta,\theta)
\ge
\sum_{e\in E_n(v)} I_e^2
\ge
\frac{1}{|E_n(v)|}
=
\frac{1}{\deg_{\Xi^n}(v)}.
\]
By Thomson's principle and Lemma~\ref{lem:basic_prop},
$
\Reff_{\Xi^n}\bigl(v,B_{\Xi^n}(v,r_*)^{\mathrm c}\bigr)
\ge
\deg_{\Xi^n}(v)^{-1}
\ge
C_{\deg}^{-1}\rho(\mathbf N)^{-(n-k)}.
$

\textbf{Upper bound.}
Choose an integer $a_*\in\mathbb N$ so large that
$
c_{\dist}\rho_{\min}(\mathcal D)^{a_*}\ge 2r_*.
$
For $n\ge k+a_*$, set $t:=n-a_*$.
Let $E_t(v)$ be the set of edges of $\Xi^t$ incident to $v$.
For each $e\in E_t(v)$, let $H_e$ be the $a_*$-step substituted copy that replaces $e$ from level $t$ to level $n$, and let $w_e$ be the other terminal of $H_e$.
Then
\[
d_{\Xi^n}(v,w_e)\ge c_{\dist}\rho_{\min}(\mathcal D)^{a_*}\ge 2r_*,
\]
so $w_e\notin B_{\Xi^n}(v,r_*)$.

For each $e\in E_t(v)$, let $\theta_e$ be an energy-minimising unit flow in $H_e$ from $v$ to $w_e$.
Define
\[
\theta:=\frac{1}{\deg_{\Xi^t}(v)}\sum_{e\in E_t(v)} \theta_e,
\]
extended by $0$ outside $\bigcup_{e\in E_t(v)} H_e$.
Then $\theta$ is a unit flow from $v$ to $B_{\Xi^n}(v,r_*)^{\mathrm c}$.
By edge-disjointness,
\[
\mathcal E_{\Xi^n}(\theta,\theta)
=
\sum_{e\in E_t(v)}
\frac{1}{\deg_{\Xi^t}(v)^2}\mathcal E_{H_e}(\theta_e,\theta_e).
\]
Now
\[
\mathcal E_{H_e}(\theta_e,\theta_e)
=
\Reff_{H_e}(v,w_e)
=
\big[\Psi^{a_*}(\mathbf 1)\big]_{\mathscr C(e)}
\le
C_{\Psi}\rho(\Psi)^{a_*}.
\]
Therefore
\[
\mathcal E_{\Xi^n}(\theta,\theta)
\le
\frac{C_{\Psi}\rho(\Psi)^{a_*}}{\deg_{\Xi^t}(v)}
\le
C_{\Psi}\rho(\Psi)^{a_*} c_{\deg}^{-1}\rho(\mathbf N)^{-(t-k)}.
\]
Since $t=n-a_*$, this gives
$
\Reff_{\Xi^n}\bigl(v,B_{\Xi^n}(v,r_*)^{\mathrm c}\bigr)
\le
C_1\rho(\mathbf N)^{-(n-k)}
$
for some constant $C_1>0$ depending on $v$ and $r_*$.

Combining the two bounds yields
$
c_1\rho(\mathbf N)^{-(n-k)}
\le
\Reff_{\Xi^n}\bigl(v,B_{\Xi^n}(v,r_*)^{\mathrm c}\bigr)
\le
C_1\rho(\mathbf N)^{-(n-k)}.
$
Equivalently,
$
\log \Reff_{\Xi^n}\bigl(v,B_{\Xi^n}(v,r_*)^{\mathrm c}\bigr)
=
-(n-k)\log\rho(\mathbf N)+O(1).
$
By the diameter bounds,
\[
-\log\frac{r_*}{\diam(\Xi^n)}
=
n\log\rho_{\min}(\mathcal D)+O(1).
\]
Dividing and letting $n\to\infty$ yields
$
\dim_{\mathrm{R,loc}}(\Xi:v)
=
-\frac{\log \rho(\mathbf N)}{\log \rho_{\min}(\mathcal D)}.
$
\end{proof}

In contrast with Theorem~\ref{thm:resistance_dimension_formula}, which gives
\[
\dimresis(\Xi)=\frac{\log \rho(\Psi)}{\log \rho_{\min}(\mathcal D)},
\]
we encounter a similar paradoxical phenomenon to that of the local mass dimension.
On the combinatorial limit graph $\Xi$, the local resistance dimension at every finite-born vertex is negative, yet this has no direct relation to the global resistance dimension.
This inconsistency is again a consequence of scale-freeness: degree explosion creates more and more parallel channels at larger scales, so the mass and resistance observed locally can vary in the opposite direction to the global scaling.
As in the mass case, after introducing the Gromov--Hausdorff--Prokhorov scaling limit we will show that $\mu$-almost every point has local resistance dimension equal to the global resistance dimension.

Before we formally study the walk dimension, we briefly discuss the classical question of recurrence for the simple random walk on $\Xi$.

\begin{definition}\label{def:asymp_res_recurrence}
We say that $G^\infty$ is recurrent if for every $v\in V(G^\infty)$ and every $r_*\in(0,1)$,
\[
\Reff_{G^\infty}(v\leftrightarrow \infty)
:=
\liminf_{n\to\infty}
\Reff_{G^n}\Bigl(
v,
B_{G^n}\Bigl(v,r_*\max_{w\in V(G^n)} d_{G^n}(v,w)\Bigr)^{\mathrm c}
\Bigr)
=
\infty.
\]
\end{definition}

\begin{proposition}\label{prop:asymp_res_recurrence_rhoPsi}
$\Xi$ is recurrent if and only if $\rho(\Psi)>1$.
\end{proposition}

\begin{proof}
Fix a finite-born vertex $v\in V(\Xi^k)$.

\textbf{Case 1: $\rho(\Psi)>1$.}
Fix any $r_*\in(0,1)$ and write
$
r_n:=r_*\max_{u\in V(\Xi^n)} d_{\Xi^n}(v,u).
$
By the distance scaling bounds for EIGS, there exist constants $c_v,C_v>0$, depending only on the birth data of $v$, such that
$
r_*c_v\rho_{\min}(\mathcal D)^{n-k}
\le
r_n
\le
r_*C_v\rho_{\min}(\mathcal D)^{n-k}
$
for all $n\ge k$.

Choose an integer $a\in\mathbb N$ so large that
$
C_{\dist}\rho_{\min}(\mathcal D)^{-a}\le r_*c_v.
$
For $n\ge k+a$, let $E_a(v)$ be the set of edges of $\Xi^{k+a}$ incident to $v$.
For each $e\in E_a(v)$, let $H_e^{(n)}$ be the $(n-k-a)$-step substituted copy replacing $e$ inside $\Xi^n$, and let $w_e$ be its other terminal.
Then
$
\diam(H_e^{(n)})
\le
C_{\dist}\rho_{\min}(\mathcal D)^{n-k-a}
\le
r_n,
$
so $H_e^{(n)}\subseteq B_{\Xi^n}(v,r_n)$.
Set
\[
U_a^{(n)}
:=
\Big(\bigcup_{e\in E_a(v)} V(H_e^{(n)})\Big)\setminus \{w_e:e\in E_a(v)\}.
\]
Then $U_a^{(n)}\subseteq B_{\Xi^n}(v,r_n)$, and therefore
$
\Reff_{\Xi^n}\bigl(v,B_{\Xi^n}(v,r_n)^{\mathrm c}\bigr)
\ge
\Reff_{\Xi^n}\bigl(v,(U_a^{(n)})^{\mathrm c}\bigr).
$

Let $\theta$ be any unit flow from $v$ to $(U_a^{(n)})^{\mathrm c}$.
For each $e\in E_a(v)$, let $I_e$ be the net current sent through $H_e^{(n)}$.
Then $\sum_{e\in E_a(v)} I_e=1$, and
$
\mathcal E_{\Xi^n}(\theta,\theta)
\ge
\sum_{e\in E_a(v)} I_e^2 \Reff_{H_e^{(n)}}(v,w_e).
$
By Corollary~\ref{lem:Psi_range},
$
\Reff_{H_e^{(n)}}(v,w_e)
=
\big[\Psi^{n-k-a}(\mathbf 1)\big]_{\mathscr C(e)}
\ge
c_{\Psi}\rho(\Psi)^{n-k-a}.
$
Hence
\[
\mathcal E_{\Xi^n}(\theta,\theta)
\ge
c_{\Psi}\rho(\Psi)^{n-k-a}\sum_{e\in E_a(v)} I_e^2
\ge
\frac{c_{\Psi}\rho(\Psi)^{n-k-a}}{|E_a(v)|}.
\]
By Thomson's principle,
\[
\Reff_{\Xi^n}\bigl(v,B_{\Xi^n}(v,r_n)^{\mathrm c}\bigr)
\ge
\frac{c_{\Psi}}{|E_a(v)|}\rho(\Psi)^{n-k-a}.
\]
Since $\rho(\Psi)>1$, the right-hand side tends to $\infty$.
As $r_*\in(0,1)$ was arbitrary, $\Xi$ is recurrent.

\textbf{Case 2: $\rho(\Psi)\le 1$.}
Choose
$
r_*:=\min\Bigl\{\frac12,\frac{c_{\dist}}{4C_v}\Bigr\}\in(0,1)
$
and
$
r_n:=r_*\max_{u\in V(\Xi^n)} d_{\Xi^n}(v,u).
$
Then
$
r_n\le \frac{c_{\dist}}{4}\rho_{\min}(\mathcal D)^{n-k}.
$
Let $E_k(v)$ be the set of edges of $\Xi^k$ incident to $v$.
For each $e\in E_k(v)$, let $H_e^{(n)}$ be the $(n-k)$-step substituted copy replacing $e$ inside $\Xi^n$, and let $w_e$ be its other terminal.
Then
\[
d_{\Xi^n}(v,w_e)\ge c_{\dist}\rho_{\min}(\mathcal D)^{n-k}\ge 4r_n,
\]
so every $w_e$ lies outside $B_{\Xi^n}(v,r_n)$.

For each $e\in E_k(v)$, let $\theta_e$ be an energy-minimising unit flow in $H_e^{(n)}$ from $v$ to $w_e$.
Define
$
\theta:=\frac{1}{|E_k(v)|}\sum_{e\in E_k(v)} \theta_e,
$
extended by $0$ outside $\bigcup_{e\in E_k(v)} H_e^{(n)}$.
Then $\theta$ is a unit flow from $v$ to $B_{\Xi^n}(v,r_n)^{\mathrm c}$.
By edge-disjointness,
\[
\mathcal E_{\Xi^n}(\theta,\theta)
=
\sum_{e\in E_k(v)}
\frac{1}{|E_k(v)|^2}\mathcal E_{H_e^{(n)}}(\theta_e,\theta_e).
\]
Again by Corollary~\ref{lem:Psi_range},
$
\mathcal E_{H_e^{(n)}}(\theta_e,\theta_e)
=
\Reff_{H_e^{(n)}}(v,w_e)
=
\big[\Psi^{n-k}(\mathbf 1)\big]_{\mathscr C(e)}
\le
C_{\Psi}\rho(\Psi)^{n-k}.
$
Therefore
$
\Reff_{\Xi^n}\bigl(v,B_{\Xi^n}(v,r_n)^{\mathrm c}\bigr)
\le
\mathcal E_{\Xi^n}(\theta,\theta)
\le
\frac{C_{\Psi}}{|E_k(v)|}\rho(\Psi)^{n-k}.
$
If $\rho(\Psi)<1$, the right-hand side tends to $0$.
If $\rho(\Psi)=1$, it stays bounded.
In either case the limit in Definition~\ref{def:asymp_res_recurrence} is not infinite for this choice of $r_*$.
Hence $\Xi$ is not recurrent.
\end{proof}

Proposition~\ref{prop:asymp_res_recurrence_rhoPsi} gives a direct criterion for deciding recurrence of the random walk on the combinatorial limit graph of an EIGS.
Later we will also show that recurrence is equivalent to a condition on the spectral dimension.
Note, however, that these conclusions rely on the standing assumption that the underlying EIGS is canonical.

\subsubsection*{Walk dimension}

With the local mass and local resistance estimates in hand, we can now estimate the local exit time.
We first show that, in contrast to the previous two quantities, the exit time not only does not depend on the birth level of the base vertex, but also does not exhibit the same kind of local inconsistency.
In other words, for the walk dimension, the local and global notions coincide and share the same value.

Let $(X_j^{\Xi^n})_{j\ge 0}$ be the discrete-time simple random walk on $\Xi^n$.
For $v\in V(\Xi^n)$ and $r>0$, define the exit time from the ball $B_{\Xi^n}(v,r)$ by
\[
\tau_{B_{\Xi^n}(v,r)^{\mathrm c}}^{\Xi^n}
:=
\inf\{j\ge 0:\ X_j^{\Xi^n}\notin B_{\Xi^n}(v,r)\}.
\]

We first need a technical tool to estimate exit times across different scales.
The purpose of the following lemma is to treat the good case, where the boundary aligns exactly with an entire block.
Once we obtain an estimate in this good case, we can use it as an input to handle the less favourable situations.

\begin{lemma}\label{lem:star_domain_exit_time_scale}
Fix $n\in\mathbb N$ and let $v\in V_k(\Xi^n)$.
Fix an integer $m$ with $1\le m\le n-k$ and set $t:=n-m$.
Let $E_{\Xi^{t}}(v)$ be the set of edges of $\Xi^{t}$ incident to $v$.

For each $e=\{v,w_e\}\in E_{\Xi^{t}}(v)$, let $H_e^{(m)}\subseteq \Xi^n$ be the $m$-step substituted copy that replaces $e$ from level $t$ to level $n$.
Set
\[
S_{v}^{(n,m)}:=\{w_e:\ e\in E_{\Xi^{t}}(v)\}\subseteq V(\Xi^n),
\qquad
\tau_{v}^{(n,m)}:=\inf\{j\ge 0:\ X_j^{\Xi^n}\in S_{v}^{(n,m)}\}.
\]
Then there exist constants $c_{\star},C_{\star}>0$, depending only on the birth data of $v$, such that for all admissible $n,m$,
\[
c_{\star} \rho(\mathbf M)^m\rho(\Psi)^m
\le
\mathbb E^v\big[\tau_{v}^{(n,m)}\big]
\le
C_{\star} \rho(\mathbf M)^m\rho(\Psi)^m.
\]
\end{lemma}

\begin{proof}

Up to the hitting time of $S_v^{(n,m)}$, the walk cannot leave $F:=\bigcup_{e\in E_{\Xi^{t}}(v)} H_e^{(m)}\subseteq \Xi^n$, since each branch $H_e^{(m)}$ meets the rest of $\Xi^n$ only at $v$ and $w_e$, and different branches meet only at $v$.
We therefore work on the finite graph $F$ with unit conductances, and write $S:=S_v^{(n,m)}$, $D:=V(F)\setminus S$ and $\tau:=\tau_v^{(n,m)}$.

\textbf{Upper bound.}
For $y\in D$, define $G(y):=\deg_F(y)^{-1}\,\mathbb E^v\big[\#\{j<\tau:\ X_j^{\Xi^n}=y\}\big]$ and $h(y):=\mathbb P^y(\tau_v<\tau)$.
Summing the numbers of visits over $y\in D$ gives $\mathbb E^v[\tau]=\sum_{y\in D}G(y)\deg_F(y)$.
The number of visits to $v$ is geometric with success probability $\mathbb P^v(\tau<\tau_v^+)=\big(\deg_F(v)R_F(v,S)\big)^{-1}$, whence $G(v)=R_F(v,S)$.
By reversibility, the kernel $(x,y)\mapsto\deg_F(y)^{-1}\,\mathbb E^x\big[\#\{j<\tau:\ X_j^{\Xi^n}=y\}\big]$ is symmetric, and the strong Markov property at $\tau_v$ gives $G(y)=h(y)\,G(v)$.
Therefore
\[
\mathbb E^v\big[\tau_v^{(n,m)}\big]
=
R_F(v,S)\sum_{y\in D}h(y)\deg_F(y).
\]

The branches connect $v$ to $S$ in parallel, and $R_{H_e^{(m)}}(v,w_e)=\big[\Psi^{m}(\mathbf 1)\big]_{\mathscr C(e)}$ by the renormalisation identity of Lemma~\ref{lem:renorm}, so Corollary~\ref{lem:Psi_range} yields
\[
\frac{|E_{\Xi^{t}}(v)|}{C_{\Psi}\rho(\Psi)^m}
\le
\frac{1}{R_F(v,S)}
\le
\frac{|E_{\Xi^{t}}(v)|}{c_{\Psi}\rho(\Psi)^m}.
\]
Moreover, $\sum_{y\in D}\deg_F(y)\le 2|E(F)|\le 2C_{\mathrm E}|E_{\Xi^{t}}(v)|\rho(\mathbf M)^m$ by Lemma~\ref{lem:basic_prop}.

Since $0\le h\le 1$, combining the Green representation with the two displays above gives
\[
\mathbb E^v\big[\tau_v^{(n,m)}\big]
\le
R_F(v,S)\cdot 2|E(F)|
\le
2C_{\mathrm E}C_{\Psi}\,\rho(\mathbf M)^m\rho(\Psi)^m.
\]

By Corollary~\ref{cor:Psi_strong_convergence}, the conductance vectors obtained by tracing the unit-resistance network of a branch onto its level-$j$ skeletons lie, after normalisation, in a fixed compact subset of $(0,\infty)^K$.
Since only finitely many rule graphs occur, and since by the strong maximum principle every zero--one harmonic profile takes values in $(0,1)$ away from the planting vertices, there exists $\theta\in(0,1)$, depending only on the EIGS, such that every function harmonic in the interior of a cell has oscillation over each child cell at most $\theta$ times its oscillation over the cell.
Fix the integer $j_0:=\min\{j\in\mathbb N:\ \theta^{j}\le\tfrac12\}$.

\textbf{Lower bound.}
Assume first that $m\ge j_0$, and fix $e\in E_{\Xi^{t}}(v)$.
The function $h$ equals $1$ at $v$, vanishes at $w_e$, and is harmonic at every other vertex of $H_e^{(m)}$, so its restriction to $H_e^{(m)}$ is the harmonic profile of this cell.
Let $\Lambda_e$ be the level-$j_0$ descendant cell of $H_e^{(m)}$ obtained by selecting, at each of the first $j_0$ substitution steps, one child cell incident to $v$.
Iterating this oscillation contraction along the chain gives $\operatorname{osc}_{\Lambda_e}(h)\le\theta^{j_0}\le\tfrac12$, and since $h(v)=1$ we obtain $h\ge\tfrac12$ on $\Lambda_e$; in particular $\Lambda_e$ avoids $S$.
The sets $V(\Lambda_e)\setminus\{v\}$ are pairwise disjoint as $e$ varies, and $|V(\Lambda_e)|-1\ge\tfrac{c_{\mathrm V}}{2}\rho(\mathbf M)^{m-j_0}$ by Lemma~\ref{lem:basic_prop}.
Hence
\[
\sum_{y\in D}h(y)\deg_F(y)
\ge
\frac{1}{2}\sum_{e\in E_{\Xi^{t}}(v)}\big(|V(\Lambda_e)|-1\big)
\ge
\frac{c_{\mathrm V}}{4}\rho(\mathbf M)^{-j_0}\,|E_{\Xi^{t}}(v)|\,\rho(\mathbf M)^m,
\]
and combining with the lower bound for $R_F(v,S)$ displayed above cancels the factor $|E_{\Xi^{t}}(v)|$ and yields
\[
\mathbb E^v\big[\tau_v^{(n,m)}\big]
\ge
\frac{c_{\Psi}c_{\mathrm V}}{4}\rho(\mathbf M)^{-j_0}\,\rho(\mathbf M)^m\rho(\Psi)^m.
\]
For the finitely many values $1\le m<j_0$, we have $\mathbb E^v[\tau_v^{(n,m)}]\ge 1$ while $\rho(\mathbf M)^m\rho(\Psi)^m$ stays bounded, so the lower bound holds after decreasing $c_{\star}$.
In fact, all constants above depend only on the EIGS, and not on the birth data of $v$.%

\end{proof}

From the argument above, we already obtain an exit-time estimate in the good block-aligned case.
Our next step is to bound the exit time for a bad boundary from above and below by the corresponding good-boundary cases.

\begin{theorem}\label{thm:exit_time_scale}
Fix $n\in\mathbb N$ and a vertex $v\in V_k(\Xi^n)$.
Fix an integer $m$ with $c_0+1\le m\le n-k-c_0$.
Assume that the radius $r>0$ satisfies
$
c_{\dist}\rho_{\min}(\mathcal D)^m
\le
r
\le
C_{\dist}\rho_{\min}(\mathcal D)^m.
$
Then there exist constants $c_{\mathrm{exit}},C_{\mathrm{exit}}>0$, depending only on the EIGS, such that
\[
c_{\mathrm{exit}}\rho(\mathbf M)^m\rho(\Psi)^m
\le
\mathbb E^v\big[\tau_{B_{\Xi^n}(v,r)^{\mathrm c}}^{\Xi^n}\big]
\le
C_{\mathrm{exit}}\rho(\mathbf M)^m\rho(\Psi)^m.
\]
\end{theorem}

\begin{proof}
Lemma~\ref{lem:star_domain_exit_time_scale} provides two-sided bounds for star-domain hitting times with constants depending only on the birth data of the centre.
Since there are only finitely many possible birth-data types, we may choose constants $c_{\star},C_{\star}>0$, depending only on the EIGS, such that the bounds in Lemma~\ref{lem:star_domain_exit_time_scale} hold for every $v$ and every admissible $(n,m)$.

\textbf{Lower bound.}
Set $m_-:=m-c_0$ and $t_-:=n-m_-=n-m+c_0$.
Then $t_-\ge k$, hence $v\in V(\Xi^{t_-})$.
Let $E_-(v)$ be the set of edges of $\Xi^{t_-}$ incident to $v$.
For each $e=\{v,w_e^-\}\in E_-(v)$, let $H_e^-\subseteq \Xi^n$ be the $m_-$-step substituted copy that replaces $e$ from level $t_-$ to level $n$.
Define
\[
S_-:=\{w_e^-:e\in E_-(v)\},
\qquad
U_-:=\Big(\bigcup_{e\in E_-(v)}V(H_e^-)\Big)\setminus S_-.
\]

By Lemma~\ref{lem:basic_prop} and the choice of $c_0$,
\[
\diam(H_e^-)
\le
C_{\dist}\rho_{\min}(\mathcal D)^{m_-}
=
C_{\dist}\rho_{\min}(\mathcal D)^{m-c_0}
\le
c_{\dist}\rho_{\min}(\mathcal D)^m
\le
r.
\]
So $V(H_e^-)\subseteq B_{\Xi^n}(v,r)$ for every $e$, and hence $U_-\subseteq B_{\Xi^n}(v,r)$.
Therefore
$
\tau_{B_{\Xi^n}(v,r)^{\mathrm c}}^{\Xi^n}\ge \tau_{U_-^{\mathrm c}}^{\Xi^n}.
$
By construction, every edge from $U_-$ to $U_-^{\mathrm c}$ enters $S_-$, and $S_-\subseteq U_-^{\mathrm c}$.
Hence $\tau_{U_-^{\mathrm c}}^{\Xi^n}=\tau_{S_-}^{\Xi^n}$ and
$
\mathbb E^v\big[\tau_{B_{\Xi^n}(v,r)^{\mathrm c}}^{\Xi^n}\big]
\ge
\mathbb E^v\big[\tau_{S_-}^{\Xi^n}\big].
$

Applying Lemma~\ref{lem:star_domain_exit_time_scale} with parameter $m_-$ yields
\[
\mathbb E^v\big[\tau_{S_-}^{\Xi^n}\big]
\ge
c_{\star}\rho(\mathbf M)^{m_-}\rho(\Psi)^{m_-}
=
c_{\star}\big(\rho(\mathbf M)\rho(\Psi)\big)^{-c_0}\rho(\mathbf M)^m\rho(\Psi)^m.
\]

\textbf{Upper bound.}
Set $m_+:=m+c_0$ and $t_+:=n-m_+=n-m-c_0$.
Then $t_+\ge k$, hence $v\in V(\Xi^{t_+})$.
Let $E_+(v)$ be the set of edges of $\Xi^{t_+}$ incident to $v$.
For each $e=\{v,w_e^+\}\in E_+(v)$, let $H_e^+\subseteq \Xi^n$ be the $m_+$-step substituted copy that replaces $e$ from level $t_+$ to level $n$.
Define
\[
S_+:=\{w_e^+:e\in E_+(v)\},
\qquad
U_+:=\Big(\bigcup_{e\in E_+(v)}V(H_e^+)\Big)\setminus S_+.
\]

Fix $e\in E_+(v)$.
Consider any path in $\Xi^n$ from $v$ to $w_e^+$.
Contract each $m_+$-step substitution copy from level $t_+$ to level $n$ to a single edge.
This produces a walk in $\Xi^{t_+}$ from $v$ to $w_e^+$.
Since $v$ and $w_e^+$ are adjacent in $\Xi^{t_+}$, every such walk has length at least $1$.
By Lemma~\ref{lem:basic_prop}, each contracted edge corresponds to an $m_+$-step copy whose endpoint distance is at least
$
c_{\dist}\rho_{\min}(\mathcal D)^{m_+}.
$
Therefore
\[
d_{\Xi^n}(v,w_e^+)
\ge
c_{\dist}\rho_{\min}(\mathcal D)^{m_+}
=
c_{\dist}\rho_{\min}(\mathcal D)^{m+c_0}
\ge
4C_{\dist}\rho_{\min}(\mathcal D)^m
\ge
4r,
\]
and hence $S_+\subseteq B_{\Xi^n}(v,r)^{\mathrm c}$.

We claim that $B_{\Xi^n}(v,r)\subseteq U_+$.
Indeed, any path in $\Xi^n$ starting at $v$ which reaches a vertex outside $\bigcup_{e\in E_+(v)}V(H_e^+)$ must hit some $w_e^+\in S_+$.
Since every $w_e^+$ has distance at least $4r$ from $v$, no vertex within distance $r$ of $v$ can lie outside the union.
Thus $B_{\Xi^n}(v,r)\subseteq U_+$.

Consequently,
$
\tau_{B_{\Xi^n}(v,r)^{\mathrm c}}^{\Xi^n}\le \tau_{U_+^{\mathrm c}}^{\Xi^n}.
$
By construction, every edge from $U_+$ to $U_+^{\mathrm c}$ enters $S_+$, and $S_+\subseteq U_+^{\mathrm c}$.
Hence $\tau_{U_+^{\mathrm c}}^{\Xi^n}=\tau_{S_+}^{\Xi^n}$ and
$
\mathbb E^v\big[\tau_{B_{\Xi^n}(v,r)^{\mathrm c}}^{\Xi^n}\big]
\le
\mathbb E^v\big[\tau_{S_+}^{\Xi^n}\big].
$

Applying Lemma~\ref{lem:star_domain_exit_time_scale} with parameter $m_+$ yields
\[
\mathbb E^v\big[\tau_{S_+}^{\Xi^n}\big]
\le
C_{\star}\rho(\mathbf M)^{m_+}\rho(\Psi)^{m_+}
=
C_{\star}\big(\rho(\mathbf M)\rho(\Psi)\big)^{c_0}\rho(\mathbf M)^m\rho(\Psi)^m.
\]
Combining the two bounds completes the proof.
\end{proof}

We see that the key point of the argument is that the $\rho(\mathbf N)^m$ terms arising in the local mass dimension and the local resistance dimension cancel out.
As a result, the exit time is no longer affected by the local scale-free irregularities, and it becomes consistent across all scales throughout the graph.
This will be crucial for our diffusion limit later: since every point shares the same exit-time scaling at every scale, we can choose a uniform time acceleration.

\begin{proposition}\label{lem:commute-time-eigs-asymp}
The commute time $T_{\Xi^n}^X(\mathfrak v_+,\mathfrak v_-)$ satisfies
\[
2 c_{\mathrm T} \rho(\mathbf M)^n \rho(\Psi)^n
\le
T_{\Xi^n}^X(\mathfrak v_+,\mathfrak v_-)
\le
2 C_{\mathrm T} \rho(\mathbf M)^n \rho(\Psi)^n.
\]
\end{proposition}

\begin{proof}
Since $\Xi^n$ is finite and connected, the commute-time identity for simple random walk applies:
\[
T_{\Xi^n}^X(\mathfrak v_+,\mathfrak v_-)
=
2 \card\big(E(\Xi^n)\big) \Reff_{\Xi^n}(\mathfrak v_+,\mathfrak v_-).
\]
By Lemma~\ref{lem:basic_prop} and Lemma~\ref{lem:Reff-asymp-rho},
$
2c_{\mathrm E}c_{\Psi}\rho(\mathbf M)^n\rho(\Psi)^n
\le
T_{\Xi^n}^X(\mathfrak v_+,\mathfrak v_-)
\le
2C_{\mathrm E}C_{\Psi}\rho(\mathbf M)^n\rho(\Psi)^n.
$
\end{proof}

\begin{definition}\label{def:walk_dimension}
A sequence $(r_n)_{n\ge k}$ is called admissible for $v$ if $r_n\to\infty$ and there exist integers $m_n$ such that for all sufficiently large $n$,
\[
c_0\le m_n\le n-k-c_0
\qquad\text{and}\qquad
c_{\dist}\rho_{\min}(\mathcal D)^{m_n}\le r_n\le C_{\dist}\rho_{\min}(\mathcal D)^{m_n}.
\]
For every $v\in V(G^\infty)$ and every admissible sequence $(r_n)$ for $v$, define
\[
\dimwalk(G^\infty:v)
:=
\lim_{n\to\infty}
\frac{\log \mathbb E^v\big[\tau_{B_{G^n}(v,r_n)^{\mathrm c}}^{G^n}\big]}{\log r_n}.
\]
If all vertices in $G^\infty$ share the same walk dimension, then we say that $G^\infty$ possesses walk dimension $\dimwalk(G^\infty)$.
\end{definition}

\begin{theorem}[Einstein relation]\label{thm:Einstein}
All vertices in $\Xi$ share the same walk dimension at every scale.
Moreover, we obtain the Einstein relation
\[
\dimwalk(\Xi)=\dimbox(\Xi)+\dimresis(\Xi).
\]
\end{theorem}

\begin{proof}
\textbf{Pointwise existence.}
Fix a finite-born vertex $v\in V(\Xi^k)$.
Let $(r_n)_{n\ge k}$ be any admissible sequence for $v$, and let $(m_n)$ be the associated integers in Definition~\ref{def:walk_dimension}.
By Theorem~\ref{thm:exit_time_scale}, for all sufficiently large $n$,
\[
c_{\mathrm{exit}}\rho(\mathbf M)^{m_n}\rho(\Psi)^{m_n}
\le
\mathbb E^v\big[\tau_{B_{\Xi^n}(v,r_n)^{\mathrm c}}^{\Xi^n}\big]
\le
C_{\mathrm{exit}}\rho(\mathbf M)^{m_n}\rho(\Psi)^{m_n}.
\]
Taking logarithms gives
\[
\log \mathbb E^v\big[\tau_{B_{\Xi^n}(v,r_n)^{\mathrm c}}^{\Xi^n}\big]
=
m_n\log\big(\rho(\mathbf M)\rho(\Psi)\big)+O(1),
\]
where the $O(1)$ term is independent of $n$.
By admissibility,
\[
c_{\dist}\rho_{\min}(\mathcal D)^{m_n}\le r_n\le C_{\dist}\rho_{\min}(\mathcal D)^{m_n},
\]
hence
$
\log r_n
=
m_n\log\rho_{\min}(\mathcal D)+O(1),
$
where the $O(1)$ term is independent of $n$.
Since $r_n\to\infty$ and $\rho_{\min}(\mathcal D)>1$, we have $m_n\to\infty$.
Dividing the last two displays and letting $n\to\infty$ yields
\[
\lim_{n\to\infty}
\frac{\log \mathbb E^v\big[\tau_{B_{\Xi^n}(v,r_n)^{\mathrm c}}^{\Xi^n}\big]}{\log r_n}
=
\frac{\log\big(\rho(\mathbf M)\rho(\Psi)\big)}{\log\rho_{\min}(\mathcal D)}.
\]
The right-hand side does not depend on the choice of the admissible sequence $(r_n)$.
Therefore the limit in Definition~\ref{def:walk_dimension} exists for the vertex $v$ and is independent of $(r_n)$.
We denote it by $\dimwalk(\Xi:v)$.

\textbf{Uniformity.}
The expression
\[
\frac{\log\big(\rho(\mathbf M)\rho(\Psi)\big)}{\log\rho_{\min}(\mathcal D)}
\]
is independent of $v$.
Hence all finite-born vertices share the same pointwise walk dimension.
Therefore $\Xi$ possesses a global walk dimension $\dimwalk(\Xi)$, and
\[
\dimwalk(\Xi)
=
\dimwalk(\Xi:v)
=
\frac{\log\big(\rho(\mathbf M)\rho(\Psi)\big)}{\log\rho_{\min}(\mathcal D)}.
\]

By the definitions of $\dimbox(\Xi)$ and $\dimresis(\Xi)$,
\[
\dimbox(\Xi)=\frac{\log\rho(\mathbf M)}{\log\rho_{\min}(\mathcal D)},
\qquad
\dimresis(\Xi)=\frac{\log\rho(\Psi)}{\log\rho_{\min}(\mathcal D)}.
\]
Therefore
\[
\dimwalk(\Xi)
=
\frac{\log\big(\rho(\mathbf M)\rho(\Psi)\big)}{\log\rho_{\min}(\mathcal D)}
=
\dimbox(\Xi)+\dimresis(\Xi).
\]
\end{proof}

\begin{proposition}\label{prop:MPsiD}
\[
\rho(\mathbf M)\rho(\Psi)\ge \rho_{\min}(\mathcal D)^2.
\]
\end{proposition}

\begin{proof}
For every finite connected graph $G$ and all distinct $u,v\in V(G)$,
\[
|E(G)|\Reff_G(u,v)\ge d_G(u,v)^2.
\]
Apply this to $G=\Xi^n$ and $(u,v)=(\mathfrak v_+,\mathfrak v_-)$.
By Lemma~\ref{lem:basic_prop} and Lemma~\ref{lem:Reff-asymp-rho},
\[
C_{\mathrm E}\rho(\mathbf M)^n C_{\Psi}\rho(\Psi)^n
\ge
|E(\Xi^n)|\Reff_{\Xi^n}(\mathfrak v_+,\mathfrak v_-)
\ge
d_{\Xi^n}(\mathfrak v_+,\mathfrak v_-)^2
\ge
c_{\dist}^2 \rho_{\min}(\mathcal D)^{2n}.
\]
Taking logarithms and letting $n\to\infty$ proves the claim.
\end{proof}

\begin{proposition}\label{prop:walk_dimension_ge_2}
\[
\dimwalk(\Xi)\ge 2.
\]
\end{proposition}

\begin{proof}
By Theorem~\ref{thm:Einstein} and Proposition~\ref{prop:MPsiD},
\[
\dimwalk(\Xi)
=
\dimbox(\Xi)+\dimresis(\Xi)
=
\frac{\log(\rho(\mathbf M)\rho(\Psi))}{\log \rho_{\min}(\mathcal D)}
\ge 2.
\]
\end{proof}

Proposition~\ref{prop:walk_dimension_ge_2} shows that this simple random walk is either diffusive or subdiffusive.
In fractal settings, subdiffusivity is usually the more typical behaviour.
Nevertheless, there is a simple and natural class of EIGS for which the walk is exactly diffusive.

\begin{example}\label{ex:uvDHL}
\begin{figure}[t]
\centering
\includegraphics[width=0.5\linewidth]{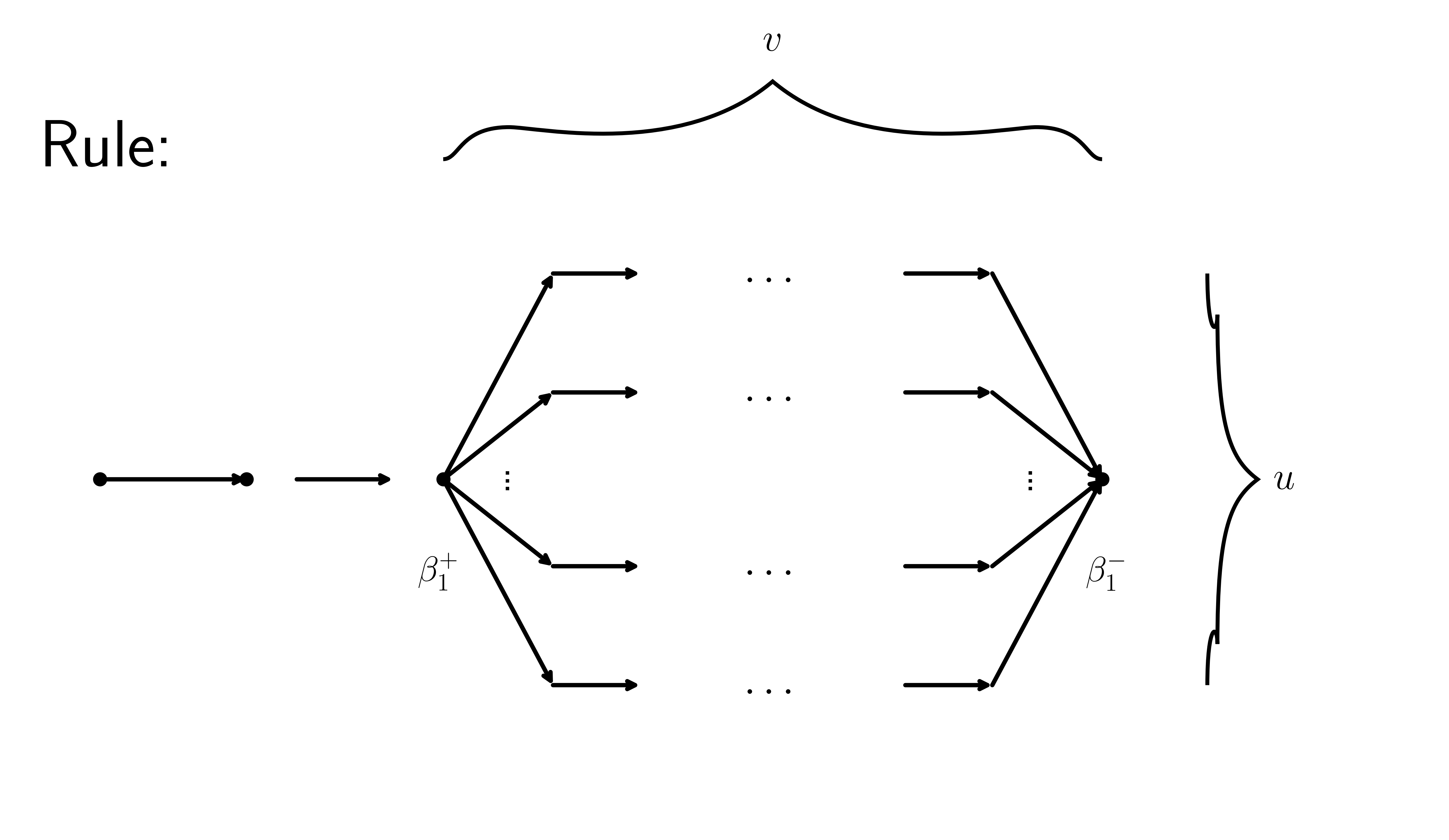}
\caption{$(u,v)$-flower}
\label{fig:uvDHL}
\end{figure}

Fix integers $u\ge 2$ and $v\ge 2$.
The $(u,v)$-flower is the single-coloured EIGS obtained by replacing every edge by $u$ pairwise edge-disjoint parallel paths, each of graph length $v$, with the planting vertices identified with the two endpoints of the original edge.
See Figure~\ref{fig:uvDHL}.
Equivalently, the rule graph consists of $u$ internally disjoint $v$-edge chains between the planting vertices.

For this EIGS,
\[
\Psi(x)=\frac{v}{u}x,
\]
since one replacement block is the parallel combination of $u$ chains, each of total resistance $vx$.
Hence
\[
\rho(\mathbf M)=uv,
\qquad
\rho_{\min}(\mathcal D)=v,
\qquad
\rho(\Psi)=\frac{v}{u}.
\]
Therefore the Minkowski dimension, resistance dimension, and walk dimension are
\[
\dimbox(\Xi)
=
\frac{\log\rho(\mathbf M)}{\log \rho_{\min}(\mathcal D)}
=
1+\frac{\log u}{\log v},
\]
\[
\dimresis(\Xi)
=
\frac{\log \rho(\Psi)}{\log \rho_{\min}(\mathcal D)}
=
1-\frac{\log u}{\log v},
\]
and
\[
\dimwalk(\Xi)
=
\dimbox(\Xi)+\dimresis(\Xi)
=
2.
\]
In particular, Proposition~\ref{prop:asymp_res_recurrence_rhoPsi} shows that $\Xi$ is recurrent exactly when $u<v$.
\end{example}

\section{Diffusion limits on metric spaces}\label{sec:diffusion}

In the previous section we worked with the combinatorial limit graph $\Xi$.
Its vertex set is $V(\Xi):=\bigcup_{n\in\mathbb N_0}V(\Xi^n)$, so every vertex is born at a finite level.
We now pass from this combinatorial object to a compact metric limit.
The metric limit itself is the same as in \cite{neroli2024fractal}, but the reference measure used for the diffusion below is the weak limit of the normalised counting measures on the approximating graphs.
This is necessary because, in the scale-free case, Section~\ref{sec:random_walks} already shows that finite-born points carry a smaller local mass exponent than the global Hausdorff dimension.

\begin{lemma}\label{lem:GH_limit_metric}
After fixing a concrete realisation, $(V(\Xi^n),\hat d_{\Xi^n})$ converges in the Gromov--Hausdorff topology to a compact metric space $(\Xi_{\mathrm M},d_{\Xi_{\mathrm M}})$.
\end{lemma}

\begin{proof}
This is the metric part of \cite{neroli2024fractal}; see also \cite{burago2001course,abraham2013note} for the general formalism.
We fix such a realisation once and for all, and regard every $V(\Xi^n)$ as a finite subset of $\Xi_{\mathrm M}$.
\end{proof}

We write $\mathscr V:=\bigcup_{n\ge 0}V(\Xi^n)\subset \Xi_{\mathrm M}$ for the set of finite-born points and $\mathscr V^\infty:=\Xi_{\mathrm M}\setminus \mathscr V$ for its complement.
We continue to write $\mathscr V^n:=V(\Xi^n)\subset \Xi_{\mathrm M}$.

\begin{lemma}\label{lem:GHP_convergence}
The probability measures $\mu_{\Xi^n}(H):=|V(H)|/|V(\Xi^n)|$ converge weakly to a Borel probability measure $\mu_{\Xi_{\mathrm M}}$ on $\Xi_{\mathrm M}$.
Moreover, every level-$m$ cell $\Lambda$ satisfies $\mu_{\Xi_{\mathrm M}}(\Lambda)\asymp \rho(\mathbf M)^{-m}$, with constants depending only on the cell type, and every point of $\mathscr V$ has $\mu_{\Xi_{\mathrm M}}$-mass zero.
\end{lemma}

\begin{proof}
For each colour $i$, let $u_i(\ell)$ be the number of interior vertices in the level-$\ell$ expansion of a single colour-$i$ edge.
The construction gives an affine recursion
\[
u(\ell+1)=\mathbf M u(\ell)+b,
\]
where $b_i$ is the number of interior vertices of the rule graph $R_i$.
Since $\mathbf M$ is primitive, Perron--Frobenius theory implies that
\[
u_i(\ell)=\kappa_i\rho(\mathbf M)^\ell+o(\rho(\mathbf M)^\ell)
\qquad\text{with }\kappa_i>0.
\]

Let $\Lambda$ be a level-$m$ cell of type $i$.
Inside $V(\Xi^n)$, the cell contributes $u_i(n-m)+O(1)$ vertices, where the $O(1)$ term comes only from the two terminals.
Likewise,
\[
|V(\Xi^n)|=\kappa_*\rho(\mathbf M)^n+o(\rho(\mathbf M)^n)
\]
for some $\kappa_*>0$.
Hence $\mu_{\Xi^n}(\Lambda)\to c_i\rho(\mathbf M)^{-m}$ with $c_i:=\kappa_i/\kappa_*$.
In particular, the values of any subsequential weak limit on the algebra generated by level cells are uniquely determined.

To pass from cells to all Borel sets, it remains to control cell boundaries.
Every boundary point is finite-born.
Fix $x\in \mathscr V$ of birth level $k$.
Its level-$m$ star neighbourhood consists of at most $C\rho(\mathbf N)^{m-k}$ level-$m$ cells, each of which has $\mu_{\Xi_{\mathrm M}}$-mass $\asymp \rho(\mathbf M)^{-m}$.
Note first that $\rho(\mathbf N)<\rho(\mathbf M)$, since the degree of any fixed finite-born vertex grows like $\rho(\mathbf N)^{n}$ up to constants by Lemma~\ref{lem:basic_prop}, while it is at most $2|E(\Xi^n)|\asymp\rho(\mathbf M)^n$; the inequality is strict, for otherwise summing the degree estimates of Lemma~\ref{lem:basic_prop} over birth levels would give $2|E(\Xi^n)|\gtrsim\sum_{k=0}^{n}\rho(\mathbf M)^{k}\rho(\mathbf N)^{n-k}\gtrsim n\rho(\mathbf M)^n$, a contradiction.
Since $\rho(\mathbf N)<\rho(\mathbf M)$, the $\mu_{\Xi_{\mathrm M}}$-mass of this star tends to zero as $m\to\infty$.
Therefore $\mu_{\Xi_{\mathrm M}}(\{x\})=0$.
Thus all boundaries in the cell algebra are $\mu_{\Xi_{\mathrm M}}$-null, and a standard monotone class argument yields a unique weak limit measure on $\Xi_{\mathrm M}$.
\end{proof}

The Hausdorff dimension of $(\Xi_{\mathrm M},d_{\Xi_{\mathrm M}})$ is still $\dimbox(\Xi)$ by \cite{neroli2024fractal}, but the reference measure used below is $\mu_{\Xi_{\mathrm M}}$.

\subsection{Dirichlet form and diffusion limits}\label{subsec:Dirichlet_form}

Let $(\rho(\Psi),\mathbf v)$ be the positive eigenpair from Lemma~\ref{thm:sdprimitive_eigenpair}, normalised arbitrarily.
For each $n\ge 0$ and each edge $e\in E(\Xi^n)$ of colour $\mathscr C(e)=i$, set $r_n(e):=\rho(\Psi)^{-n}[\mathbf v]_i$ and $c_n(e):=r_n(e)^{-1}$.
For $f:V(\Xi^n)\to\mathbb R$, write
\[
\mathcal E_n^{W}(f,f):=\sum_{\{u,v\}\in E(\Xi^n)} c_n(\{u,v\})(f(u)-f(v))^2.
\]
Let $P_n$ be the transition operator of the discrete-time simple random walk on $\Xi^n$, let $\pi_n(x):=\deg_{\Xi^n}(x)/(2|E(\Xi^n)|)$, and write
\[
\mathcal E_n^{X}(f,f):=\langle f-P_nf,f\rangle_{L^2(\pi_n)}.
\]
Finally, set $a_n:=2c_\Psi^*\rho(\Psi)^n|E(\Xi^n)|$ and $\widetilde{\mathcal E}_n^{X}:=a_n\mathcal E_n^{X}$.

\begin{lemma}\label{lem:sec31_trace_weighted}
For every $n\ge 0$ and every $f:V(\Xi^n)\to\mathbb R$, the quantity $\mathcal E_n^{W}(f,f)$ is the minimum of $\mathcal E_{n+1}^{W}(g,g)$ over all extensions $g:V(\Xi^{n+1})\to\mathbb R$ of $f$.
The minimiser is unique.
\end{lemma}

\begin{proof}
Each edge $e=\{a,b\}\in E(\Xi^n)$ of colour $i$ is replaced by a copy of $R_i$ attached only at $a$ and $b$.
On that copy, the effective resistance between the two terminals is
\[
\rho(\Psi)^{-(n+1)}\psi_{R_i:(\beta_i^+,\beta_i^-)}(\mathbf v)
=
\rho(\Psi)^{-n}[\mathbf v]_i
=
r_n(e).
\]
The Dirichlet principle therefore shows that the least possible energy on this copy is exactly $c_n(e)(f(a)-f(b))^2$.
Summing over all edges proves the claim, and strict convexity gives uniqueness, since by canonicality every internal vertex is electrically connected to the planting vertices; see \cite{kigami2001analysis}.
\end{proof}

\begin{lemma}\label{lem:sec31_continuous_harmonics}
For every $m\ge 0$ and every $h:V(\Xi^m)\to\mathbb R$, there is a unique continuous function on $\Xi_{\mathrm M}$ whose restriction to each $V(\Xi^n)$, $n\ge m$, is the minimising extension from Lemma~\ref{lem:sec31_trace_weighted}.
The union of all such functions is uniformly dense in $C(\Xi_{\mathrm M})$.
\end{lemma}

\begin{proof}
Fix $m$ and $h$.
For each rule graph and each conductance vector from the compact family generated by $\alpha_j^{(q)}:=c_\Psi^*\rho(\Psi)^q[\Psi^q(\mathbf 1)]_j^{-1}$ and $\alpha_j^{(\infty)}:=[\mathbf v]_j^{-1}$, the $0$--$1$ harmonic profile has edge oscillation strictly smaller than $1$.
Indeed, no rule graph contains an edge joining its two planting vertices: for $K=1$ this is the assumption $d_{R_1}(\beta_1^+,\beta_1^-)\ge 2$, while for $K\ge 2$ such an edge would itself be an admissible path using a single colour, contradicting distance-positivity; the strong maximum principle then forces every edge oscillation of the profile to be strictly smaller than $1$, and compactness of the conductance family makes the bound uniform.
Since the family of rule graphs is finite, there is a uniform $\theta\in(0,1)$ such that every child-cell oscillation is at most $\theta$ times the parent-cell oscillation.
Hence the harmonic extensions on $\mathscr V$ are uniformly Cauchy and extend uniquely to a continuous function on $\Xi_{\mathrm M}$.

For density, let $f\in C(\Xi_{\mathrm M})$ and $\varepsilon>0$.
Choose $m$ so large that the oscillation of $f$ on every level-$m$ cell is at most $\varepsilon$.
If $h:=f|_{V(\Xi^m)}$, then the harmonic extension of $h$ stays between the terminal values on each level-$m$ cell, so it differs from $f$ by at most $\varepsilon$.
\end{proof}

\begin{lemma}\label{lem:sec31_projection_tail_continuous}
If $u$ is one of the continuous harmonic functions from Lemma~\ref{lem:sec31_continuous_harmonics}$,$ and $u_m$ is the level-$m$ harmonic function with $u_m|_{V(\Xi^m)}=u|_{V(\Xi^m)}$, then $\mathcal E(u,u)=\mathcal E(u_m,u_m)+\mathcal E(u-u_m,u-u_m)$ and
\[
\|u-u_m\|_{L^2(\mu_{\Xi_{\mathrm M}})}^2
\le
C(\rho(\mathbf M)\rho(\Psi))^{-m}\mathcal E(u,u).
\]
\end{lemma}

\begin{proof}
The orthogonality identity is the Euler--Lagrange relation for harmonic projection.
Recall that $\rho(\mathbf M)\rho(\Psi)\ge\rho_{\min}(\mathcal D)^2>1$ by Proposition~\ref{prop:MPsiD}, so the bound below decays geometrically in $m$.
For the $L^2$ bound, write $u-u_m$ as the sum of one-step increments across levels.
On each level-$\ell$ cell, such an increment vanishes at the two terminals and lies in a fixed finite-dimensional space, so its supremum is bounded by its cell energy.
Since each level-$\ell$ cell has $\mu_{\Xi_{\mathrm M}}$-mass $\asymp \rho(\mathbf M)^{-\ell}$ by Lemma~\ref{lem:GHP_convergence}, summing over cells and over $\ell\ge m$ proves the estimate.
\end{proof}

\begin{theorem}\label{thm:sec31_limit_dirichlet_form}
There is a conservative regular Dirichlet form $(\mathcal E_W,\mathcal F_W)$ on $L^2(\Xi_{\mathrm M},\mu_{\Xi_{\mathrm M}})$ whose trace on $V(\Xi^m)$ is $\mathcal E_m^W$ for every $m\ge 0$.
\end{theorem}

\begin{proof}
Let $\mathcal H_m$ be the space of continuous level-$m$ harmonic functions from Lemma~\ref{lem:sec31_continuous_harmonics}.
For $u\in\mathcal H_m$, define $\mathcal E(u,u):=\mathcal E_m^W(u|_{V(\Xi^m)},u|_{V(\Xi^m)})$.
Lemma~\ref{lem:sec31_trace_weighted} shows that this is well defined.
By Lemma~\ref{lem:sec31_continuous_harmonics}, $\bigcup_m\mathcal H_m$ is uniformly dense in $C(\Xi_{\mathrm M})$, hence dense in $L^2(\Xi_{\mathrm M},\mu_{\Xi_{\mathrm M}})$.
Lemma~\ref{lem:sec31_projection_tail_continuous} implies closability and compactness of the form norm embedding.
The Markov property is inherited from the finite traces.
The closure is therefore a conservative regular Dirichlet form; see \cite{fukushima1994dirichlet}.
\end{proof}

\begin{lemma}\label{lem:sec31_speed_measure}
We have $\pi_n\Rightarrow \mu_{\Xi_{\mathrm M}}$.
Moreover, for every $m\ge 0$ there is $C_m>0$ such that every level-$m$ cell $\Lambda\subset \Xi^n$ satisfies $\pi_n(V(\Lambda))\le C_m\rho(\mathbf M)^{-m}$ for all $n\ge m$.
\end{lemma}

\begin{proof}
Fix $m$ and a level-$m$ cell $\Lambda$ with planting vertices $u_1,u_2$.
Then
\[
\sum_{x\in V(\Lambda)}\deg_{\Xi^n}(x)
=
2|E(\Lambda)|+O(\deg_{\Xi^n}(u_1)+\deg_{\Xi^n}(u_2)).
\]
By Lemma~\ref{lem:basic_prop}, $|E(\Lambda)|\asymp \rho(\mathbf M)^{n-m}$, $|E(\Xi^n)|\asymp \rho(\mathbf M)^n$, and $\deg_{\Xi^n}(u_i)\lesssim \rho(\mathbf N)^{n-m}$.
Hence $\pi_n(V(\Lambda))\le C_m\rho(\mathbf M)^{-m}$.

The same decomposition shows that, on each fixed level-$m$ cell, $\pi_n$ and $\mu_{\Xi^n}$ differ by $O((\rho(\mathbf N)/\rho(\mathbf M))^{n-m})$.
Since $\mu_{\Xi^n}\Rightarrow \mu_{\Xi_{\mathrm M}}$ by Lemma~\ref{lem:GHP_convergence}, the weak convergence of $\pi_n$ follows.
\end{proof}

\begin{lemma}\label{lem:sec31_discrete_harmonics}
Fix $m\ge 0$ and $h:V(\Xi^m)\to\mathbb R$.
Let $P_m^{(n)}h$ be the unique minimiser of $\widetilde{\mathcal E}_n^{X}$ on $V(\Xi^n)$ with boundary values $h$ on $V(\Xi^m)$, and let $P_mh$ be the continuous harmonic function from Lemma~\ref{lem:sec31_continuous_harmonics}.
Then $\|P_m^{(n)}h-P_mh\|_{L^\infty(V(\Xi^n))}\to 0$.
In particular, if $(\eta_a)_{a=1}^{d_m}$ is a basis of $\mathbb R^{V(\Xi^m)}$, $\phi_{m,a}^{(n)}:=P_m^{(n)}\eta_a$, and $\phi_{m,a}:=P_m\eta_a$, then
\[
\widetilde{\mathcal E}_n^{X}(\phi_{m,a}^{(n)},\phi_{m,b}^{(n)})
\to
\mathcal E_W(\phi_{m,a},\phi_{m,b}),
\]
and
\[
\langle \phi_{m,a}^{(n)},\phi_{m,b}^{(n)}\rangle_{L^2(\pi_n)}
\to
\langle \phi_{m,a},\phi_{m,b}\rangle_{L^2(\mu_{\Xi_{\mathrm M}})}.
\]
\end{lemma}

\begin{proof}
Fix a coarse depth $r$.
Tracing the unit-resistance network on $\Xi^n$ to $V(\Xi^{m+r})$ gives macro-conductances $c_\Psi^*\rho(\Psi)^n[\Psi^{n-m-r}(\mathbf 1)]_j^{-1}$ on colour-$j$ edges.
By Perron--Frobenius asymptotics, these converge uniformly in $j$ to $\rho(\Psi)^{m+r}[\mathbf v]_j^{-1}$.
Hence the discrete harmonic extension on the fixed finite graph $V(\Xi^{m+r})$ converges to the weighted harmonic extension.
The oscillation contraction from Lemma~\ref{lem:sec31_continuous_harmonics} upgrades this to uniform convergence on all deeper levels.
The convergence of the energy and $L^2$ Gram matrices is then immediate from the traced energy convergence and Lemma~\ref{lem:sec31_speed_measure}.
\end{proof}

\begin{lemma}\label{lem:sec31_projection_tail_discrete}
For every $n\ge m$ and every $f:V(\Xi^n)\to\mathbb R$, if $f_m$ is the discrete harmonic extension of $f|_{V(\Xi^m)}$, then
\[
\|f-f_m\|_{L^2(\pi_n)}^2
\le
C(\rho(\mathbf M)\rho(\Psi))^{-m}\widetilde{\mathcal E}_n^{X}(f,f).
\]
\end{lemma}

\begin{proof}
Write $f-f_m$ as the sum of the one-step increments between consecutive coarse levels.
On each level-$\ell$ cell such an increment vanishes at the two terminals and lies in a fixed finite-dimensional prototype space.
The traced conductance vectors form a compact family, so a uniform norm equivalence bounds the cellwise supremum by the cellwise energy.
Lemma~\ref{lem:sec31_speed_measure} gives $\pi_n(V(\Lambda))\lesssim \rho(\mathbf M)^{-\ell}$ for every level-$\ell$ cell $\Lambda$.
Summing over cells and over $\ell\ge m$ proves the estimate.
\end{proof}

We now use the asymptotic relation from \cite{kuwae2003convergence} induced by the restriction maps $J_n\varphi:=\varphi|_{V(\Xi^n)}$ for $\varphi\in C(\Xi_{\mathrm M})$.
Strong and weak convergence below are understood with respect to this common-space realisation.

\begin{theorem}\label{thm:sec31_mosco}
The forms $\widetilde{\mathcal E}_n^{X}$ Mosco-converge to $(\mathcal E_W,\mathcal F_W)$.
\end{theorem}

\begin{proof}
For the limsup inequality, it is enough to approximate a harmonic spline $u=P_mh$ by $u_n:=P_m^{(n)}h$.
Lemma~\ref{lem:sec31_discrete_harmonics} gives $u_n\to u$ strongly, and the traced energy convergence gives $\widetilde{\mathcal E}_n^X(u_n,u_n)\to \mathcal E_W(u,u)$.
Since harmonic splines form a core by Theorem~\ref{thm:sec31_limit_dirichlet_form}, the full limsup follows.

For the liminf inequality, let $f_n\rightharpoonup f$ and assume that $\sup_n\widetilde{\mathcal E}_n^X(f_n,f_n)<\infty$.
Fix $m$ and let $u_n^{(m)}$ be the discrete harmonic extension of $f_n|_{V(\Xi^m)}$.
By Lemma~\ref{lem:sec31_projection_tail_discrete}, the difference $f_n-u_n^{(m)}$ tends to zero in $L^2$ uniformly in $n$ up to the factor $(\rho(\mathbf M)\rho(\Psi))^{-m/2}$.
Expand $u_n^{(m)}$ in a fixed basis of level-$m$ splines.
Lemma~\ref{lem:sec31_discrete_harmonics} implies convergence of the corresponding Gram matrices, hence after passing to a subsequence we obtain a limit $u_m\in \mathcal H_m$ with $\mathcal E_W(u_m,u_m)\le \liminf_n\widetilde{\mathcal E}_n^X(f_n,f_n)$.
The family $(u_m)$ is compatible under harmonic projection, and Lemma~\ref{lem:sec31_projection_tail_continuous} shows that it converges in $L^2(\mu_{\Xi_{\mathrm M}})$ to $f$.
Letting $m\to\infty$ yields the liminf inequality.
\end{proof}

\begin{lemma}\label{lem:sec31_one_step_stay}
There are constants $r_{\mathrm{stay}}\in(0,1)$, $c_{\mathrm{stay}}>0$, $p_{\mathrm{stay}}\in(0,1)$, and $n_{\mathrm{stay}}\in\mathbb N$ such that, for all $n\ge n_{\mathrm{stay}}$, all $x\in V(\Xi^n)$, and all admissible radii $\delta\in(0,r_{\mathrm{stay}}]$, the exit time $\tau_n(x,\delta)$ from the ball of radius $\delta$ in the metric $\hat d_{\Xi^n}$ satisfies
\[
\mathbb P_x\bigl(\tau_n(x,\delta)\ge c_{\mathrm{stay}}a_n\delta^{\dimwalk(\Xi)}\bigr)\ge p_{\mathrm{stay}}.
\]
\end{lemma}

\begin{proof}
Write $r:=\delta\operatorname{diam}(\Xi^n)$ for the graph-metric radius, so that admissibility provides an integer $m$ with $c_{\dist}\rho_{\min}(\mathcal D)^m\le r\le C_{\dist}\rho_{\min}(\mathcal D)^m$ and $r>1$, and $a_n\delta^{\dimwalk(\Xi)}\asymp(\rho(\mathbf M)\rho(\Psi))^m$ with constants depending only on the EIGS.
Fix the integer $c_1$ with $C_{\dist}\rho_{\min}(\mathcal D)^{-c_1}\le c_{\dist}/4$, and set $m_0:=c_1+2$.

\textbf{Small scales.}
Assume first $m<m_0$.
Two distinct vertices born by level $n-c_2$ are at distance at least $c_{\dist}\rho_{\min}(\mathcal D)^{c_2}>1$ in $\Xi^n$, where $c_2$ is the fixed integer with $c_{\dist}\rho_{\min}(\mathcal D)^{c_2}>1$; hence every neighbour of $x$ has birth level larger than $n-c_2$ and degree at most $D:=C_{\deg}\rho(\mathbf N)^{c_2}$ by Lemma~\ref{lem:basic_prop}.
The walk that alternates between $x$ and a neighbour for $\lceil c_{\mathrm{stay}}a_n\delta^{\dimwalk(\Xi)}\rceil$ steps stays in $B(x,r)$ because $r>1$, and this event has probability at least $D^{-\lceil c_{\mathrm{stay}}a_n\delta^{\dimwalk(\Xi)}\rceil}$, which is bounded below because $a_n\delta^{\dimwalk(\Xi)}\le C(\rho(\mathbf M)\rho(\Psi))^{m_0}$ is bounded.
Assume from now on $m\ge m_0$, and set $t:=n-m+c_1$ and $m':=m-c_1$.

\textbf{Uniform exit bound for cells.}
For $s\in\mathbb N$, define $T_s:=\sup\mathbb E_z\bigl[\sigma_{\partial H}\bigr]$, where the supremum runs over all $s$-step substituted copies $H$ arising in the construction and all $z\in V(H)$, and $\sigma_{\partial H}$ denotes the hitting time of the two terminals of $H$; until $\sigma_{\partial H}$ the walk on $\Xi^n$ agrees with the walk on $H$, since copies are attached only at their terminals.
We claim that $T_s\le C_T(\rho(\mathbf M)\rho(\Psi))^{s}$ with $C_T$ depending only on the EIGS.
From any $z$, the hitting time of the level-one skeleton $V_1(H)$ is bounded by $\sigma_{\partial\Lambda'}$ for the level-one cell $\Lambda'$ containing $z$, hence its expectation is at most $T_{s-1}$.
Observed at successive distinct skeleton vertices, the walk is the random walk on the rule graph in which every colour-$j$ edge carries the effective conductance of the corresponding $(s-1)$-step copy, by the super-edge reduction underlying Lemma~\ref{lem:renorm}.
By Corollary~\ref{cor:Psi_strong_convergence} these conductances lie, after normalisation, in a fixed compact subset of $(0,\infty)^K$, so the expected number of skeleton jumps before hitting the two terminals of $H$ is at most a constant $C_R$ depending only on the EIGS.
Each jump starts at a skeleton vertex $u$, ends at the first visit to $V_1(H)\setminus\{u\}$, and therefore lasts at most $C_\star(\rho(\mathbf M)\rho(\Psi))^{s-1}$ in expectation, uniformly in $u$, by Lemma~\ref{lem:star_domain_exit_time_scale} applied inside $H$ with parameter $s-1$.
Optional stopping then gives $T_s\le T_{s-1}+C_RC_\star(\rho(\mathbf M)\rho(\Psi))^{s-1}$, and since $\rho(\mathbf M)\rho(\Psi)\ge\rho_{\min}(\mathcal D)^2>1$ by Proposition~\ref{prop:MPsiD}, unrolling the recursion proves the claim.

\textbf{Staying at early-born vertices.}
Let $w\in V(\Xi^{t})$, so that the birth level of $w$ is at most $t$ and $m'\le n-t+c_1-c_1= m-c_1\le n-\mathrm{birth}(w)$.
Consider the star of $w$ at depth $m'$, namely the branches $H_e^{(m')}$ over $e\in E_{\Xi^{n-m'}}(w)$ with boundary $S_w^{(n,m')}$, and note that every point of the star is within distance $C_{\dist}\rho_{\min}(\mathcal D)^{m'}\le r/4$ of $w$.
Hence, on the event $\{\tau_w^{(n,m')}\ge T_1\}$ with $T_1:=\tfrac12 c_\star(\rho(\mathbf M)\rho(\Psi))^{m'}$, the walk started at $w$ stays in $B(w,r/4)$ up to time $T_1$.
By Lemma~\ref{lem:star_domain_exit_time_scale}, $\mathbb E_w\bigl[\tau_w^{(n,m')}\bigr]\ge c_\star(\rho(\mathbf M)\rho(\Psi))^{m'}$.
For any $y$ in the star, $y$ lies in a single branch $H$, and by the strong Markov property at $\sigma_{\partial H}$,
\[
\mathbb E_y\bigl[\tau_w^{(n,m')}\bigr]
\le
T_{m'}+C_\star(\rho(\mathbf M)\rho(\Psi))^{m'}
\le
C'(\rho(\mathbf M)\rho(\Psi))^{m'},
\]
because the walk either exits the star at the far terminal of $H$ or reaches $w$, from where Lemma~\ref{lem:star_domain_exit_time_scale} applies again.
Combining $\mathbb E_w[\tau_w^{(n,m')}]\le T_1+\mathbb P_w\bigl(\tau_w^{(n,m')}>T_1\bigr)\sup_y\mathbb E_y[\tau_w^{(n,m')}]$ with the two displays yields
\[
\mathbb P_w\bigl(\tau_w^{(n,m')}\ge T_1\bigr)\ge\frac{c_\star}{2C'}=:p_0.
\]

\textbf{Transfer to arbitrary starting points.}
If $x\in V(\Xi^{t})$, the previous step with $w:=x$ shows that the walk stays in $B(x,r/4)\subseteq B(x,r)$ up to time $T_1$ with probability at least $p_0$.
Otherwise, let $\Lambda$ be the level-$t$ cell containing $x$ and $\sigma:=\sigma_{\partial\Lambda}$; then $\operatorname{diam}(\Lambda)\le C_{\dist}\rho_{\min}(\mathcal D)^{m-c_1}\le r/4$, both terminals of $\Lambda$ belong to $V(\Xi^{t})$, and before $\sigma$ the walk stays in $\Lambda\subseteq B(x,r/4)$.
By the strong Markov property at $\sigma$,
\[
\mathbb P_x\bigl(\tau_n(x,\delta)\ge T_1\bigr)
\ge
\sum_{w\in\partial\Lambda}\mathbb P_x\bigl(X^{\Xi^n}_{\sigma}=w\bigr)\,\mathbb P_w\bigl(\tau_w^{(n,m')}\ge T_1\bigr)
\ge
p_0,
\]
because on the event $\{\tau_w^{(n,m')}\ge T_1\}$ the walk after $\sigma$ stays in $B(w,r/4)\subseteq B(x,r/2)$, and the time elapsed before $\sigma$ only increases $\tau_n(x,\delta)$.
Since $T_1\asymp a_n\delta^{\dimwalk(\Xi)}$ with constants depending only on the EIGS, choosing $c_{\mathrm{stay}}$ accordingly and setting $p_{\mathrm{stay}}$ to be the minimum of $p_0$ and the small-scale bound completes the proof.%

\end{proof}

\begin{theorem}\label{thm:sec31_discrete_eld_tightness}
There are constants $n_{\mathrm{eld}},r_{\mathrm{eld}},t_{\mathrm{eld}},c_{\mathrm{eld}},C_{\mathrm{eld}}>0$ such that, for all $n\ge n_{\mathrm{eld}}$, all $x\in V(\Xi^n)$, all admissible radii $r\in(0,r_{\mathrm{eld}}]$, and all $t\in(0,t_{\mathrm{eld}}r^{\dimwalk(\Xi)}]$, we have
\[
\mathbb P_x\Bigl(a_n^{-1}\tau_{B_{\hat d_{\Xi^n}}(x,r)^c}^{\Xi^n}\le t\Bigr)
\le
C_{\mathrm{eld}}
\exp\Bigl[-c_{\mathrm{eld}}
\Bigl(\frac{r^{\dimwalk(\Xi)}}{t}\Bigr)^{1/(\dimwalk(\Xi)-1)}\Bigr].
\]
In particular, the Poissonised walks $Y_t^{(n)}:=X_{N_{a_nt}}^{\Xi^n}$ are tight in $D([0,T],\Xi_{\mathrm M})$ for every $T>0$.
\end{theorem}

\begin{proof}
Choose $\delta\asymp (t/r)^{1/(\dimwalk(\Xi)-1)}$ (meaningful since $\dimwalk(\Xi)\ge 2$ by Proposition~\ref{prop:walk_dimension_ge_2}) with $16\delta\le r$, define radial stopping times as in the standard chaining proof, and let $I_k$ record the event that the $k$-th radial increment takes at least $c_{\mathrm{stay}}a_n(\delta/8)^{\dimwalk(\Xi)}$ time.
Lemma~\ref{lem:sec31_one_step_stay} and the strong Markov property give a uniform lower bound on the conditional expectations of $I_k$.
If the walk exits $B(x,r)$ by time $a_nt$, then the number of long increments must be abnormally small.
Azuma--Hoeffding therefore yields the stated lower-tail estimate, and Aldous' tightness criterion gives the path-space tightness; see \cite{azuma1967weighted,aldous1978stopping,billingsley1999convergence}.
\end{proof}

In the common realisation fixed above, we use
$\overset{\mathrm{GHPS}}{\longrightarrow}$
as shorthand for Gromov--Hausdorff--Prokhorov convergence of the metric-measure spaces together with weak convergence of the associated paths in $D([0,T],\Xi_{\mathrm M})$.

\begin{theorem}\label{thm:sec31_ghps}
Let $v_n\in V(\Xi^n)$ and suppose that $x_n\to x\in \Xi_{\mathrm M}$, where $x_n:=\pi(v_n)$.
Then for every $T>0$,
\[
\bigl(V(\Xi^n),\hat d_{\Xi^n},\mu_{\Xi^n},v_n,(X_{\lfloor a_nt\rfloor}^{\Xi^n})_{0\le t\le T}\bigr)
\overset{\mathrm{GHPS}}{\longrightarrow}
\bigl(\Xi_{\mathrm M},d_{\Xi_{\mathrm M}},\mu_{\Xi_{\mathrm M}},x,(W_t)_{0\le t\le T}\bigr).
\]
Equivalently, after the common realisation fixed above, the rescaled walks
$t\mapsto X_{\lfloor a_nt\rfloor}^{\Xi^n}$ started from $v_n$ converge in distribution in
$D([0,T],\Xi_{\mathrm M})$ to the diffusion $W$ started from $x$.
\end{theorem}

\begin{proof}
The corresponding statement for the Poissonised walks follows from the standard process-convergence theorem for Mosco-convergent regular Dirichlet forms in varying spaces, applied with Theorem~\ref{thm:sec31_mosco}, Theorem~\ref{thm:sec31_discrete_eld_tightness}, and the common-space realisation fixed above; see \cite{kuwae2003convergence,fukushima1994dirichlet,billingsley1999convergence}.
The de-Poissonised statement follows from $\sup_{0\le t\le T}|N_{a_nt}-a_nt|/a_n\to 0$ in probability.
\end{proof}

\begin{corollary}\label{cor:sec31_continuous_stay}
For every $x\in \Xi_{\mathrm M}$ there are constants $c_x,p_x,r_x>0$ such that
\[
\mathbb P_x\bigl(\tau_{B_{\Xi_{\mathrm M}}(x,r)}>c_xr^{\dimwalk(\Xi)}\bigr)\ge p_x
\]
for all $r\in(0,r_x)$.
\end{corollary}

\begin{proof}
Choose $x_n\in V(\Xi^n)$ with $x_n\to x$, and let $Y_t^{(n)}:=X_{\lfloor a_nt\rfloor}^{\Xi^n}$.
By Theorem~\ref{thm:sec31_ghps}, $Y^{(n)}$ converges in distribution to $W$ started from $x$.
Fix $r>0$ small and choose admissible $\delta_n$ with $r/4\le \delta_n\le r/2$.
Set $t_r:=c_{\mathrm{stay}}4^{-\dimwalk(\Xi)}r^{\dimwalk(\Xi)}$.
Lemma~\ref{lem:sec31_one_step_stay} gives
\[
\mathbb P_{x_n}\bigl(\tau_{B_{\hat d_{\Xi^n}}(x_n,\delta_n)}^{\Xi^n}>a_nt_r\bigr)\ge p_{\mathrm{stay}}
\]
for all large $n$.

By the Skorokhod representation theorem \cite{billingsley1999convergence}, we may realise $Y^{(n)}$ and $W$ on a common probability space so that $Y^{(n)}\to W$ almost surely in path space.
On this coupling, if $d(x_n,x)<r/4$ and $\sup_{0\le s\le t_r}d(Y_s^{(n)},W_s)<r/4$, then the event $\{\sup_{0\le s\le t_r}d(Y_s^{(n)},x_n)\le \delta_n\}$ implies $\{\sup_{0\le s\le t_r}d(W_s,x)<r\}$.
Passing to the limit yields
\[
\mathbb P_x\bigl(\tau_{B_{\Xi_{\mathrm M}}(x,r)}>t_r\bigr)\ge p_{\mathrm{stay}}.
\]
\end{proof}


\subsection{Resistance-form Brownian motion}\label{subsec:resistance_form}

Throughout this subsection we assume $\dimresis(\Xi)>0$.
Equivalently, by Theorem~\ref{thm:resistance_dimension_formula}, $\rho(\Psi)>1$.
We emphasise that the effective resistance used in this subsection is not the unit-resistance effective resistance $\Reff_G$ introduced in Section~\ref{sec:random_walks}; it is the renormalised resistance induced by the compatible energies $(\mathcal E_m^W)_{m\ge 0}$.

For $m\ge 0$ and $x,y\in V(\Xi^m)$, set $\Reff_m(x,x):=0$ and, for $x\ne y$, define
\[
\Reff_m(x,y)
:=
\left(
\inf\left\{
\mathcal E_m^W(f,f):
 f:V(\Xi^m)\to\mathbb R,
 f(x)=0,
 f(y)=1
\right\}
\right)^{-1}.
\]
By the trace compatibility of Lemma~\ref{lem:sec31_trace_weighted},
\[
\Reff_{m+1}(x,y)=\Reff_m(x,y),
\qquad x,y\in V(\Xi^m).
\]
Hence, for $x,y\in\mathscr V$, we define
\[
\Reff_*(x,y):=\Reff_m(x,y),
\]
where $m$ is any integer such that $x,y\in V(\Xi^m)$.
This is well defined and will be kept distinct from the unit-resistance quantity $\Reff_G$.

\begin{lemma}\label{lem:sec32_point_boundary}
There is $C_\partial>0$ such that, for every level-$m$ cell $\Lambda$ with terminals $a_{\Lambda},b_{\Lambda}$ and every $z\in\mathscr V\cap\Lambda$,
\[
\Reff_*(z,a_{\Lambda})\le C_\partial\rho(\Psi)^{-m},
\qquad
\Reff_*(z,b_{\Lambda})\le C_\partial\rho(\Psi)^{-m}.
\]
\end{lemma}

\begin{proof}
For $z\in\mathscr V\cap\Lambda$, let $h_{\Lambda}(z)$ be the least $\ell\ge 0$ such that $z$ is a terminal of some level-$(m+\ell)$ descendant cell of $\Lambda$.
For $L\ge 0$, define
\[
A_L:=
\sup
\left\{
\rho(\Psi)^m
\max\bigl(
\Reff_*(z,a_{\Lambda}),
\Reff_*(z,b_{\Lambda})
\bigr):
 h_{\Lambda}(z)\le L
\right\},
\]
where the supremum is over all levels $m$, all level-$m$ cells $\Lambda$, and all $z\in\mathscr V\cap\Lambda$ satisfying $h_{\Lambda}(z)\le L$.

If $L=0$, then $z$ is one of the two terminals of $\Lambda$.
The resistance between the two terminals of a level-$m$ cell of colour $i$ is $\rho(\Psi)^{-m}[\mathbf v]_i$.
Thus
\[
A_0\le \max_{i\in[K]}[\mathbf v]_i<\infty.
\]

Assume that $A_L<\infty$.
Let $z\in\mathscr V\cap\Lambda$ with $h_{\Lambda}(z)\le L+1$.
If $z$ is a terminal of $\Lambda$, the required bound is already covered by $A_0$.
Otherwise choose a first-generation child cell $\Lambda'\subset\Lambda$ containing $z$.
Then $h_{\Lambda'}(z)\le L$.
By the triangle inequality,
\[
\Reff_*(z,a_{\Lambda})
\le
\Reff_*(z,a_{\Lambda'})
+
\Reff_*(a_{\Lambda'},a_{\Lambda}).
\]
The first term is bounded by $\rho(\Psi)^{-(m+1)}A_L$.
For the second term, only finitely many first-generation prototype configurations occur.
Therefore there exists a finite constant $B$, depending only on the rule graphs and on $\mathbf v$, such that
\[
\Reff_*(a_{\Lambda'},a_{\Lambda})\le B\rho(\Psi)^{-m}.
\]
The same argument applies with $b_{\Lambda}$ in place of $a_{\Lambda}$.
Hence
\[
A_{L+1}\le \rho(\Psi)^{-1}A_L+B.
\]
Since $\rho(\Psi)>1$, the sequence $(A_L)_{L\ge 0}$ is uniformly bounded.
This proves the claim.
\end{proof}

\begin{theorem}\label{thm:sec32_resistance_metric}
There exists a resistance form $(\mathcal E_R,\mathcal F_R)$ on $\Xi_{\mathrm M}$ whose trace on $V(\Xi^m)$ is $\mathcal E_m^W$ for every $m\ge 0$.
If $\Reff_{\mathcal E_R}$ denotes the resistance metric induced by $(\mathcal E_R,\mathcal F_R)$, then
\[
\Reff_{\mathcal E_R}|_{\mathscr V\times\mathscr V}=\Reff_*,
\]
and $\Reff_{\mathcal E_R}$ induces the same topology as $d_{\Xi_{\mathrm M}}$.
\end{theorem}

\begin{proof}
Define
\[
\mathcal F_{\mathscr V}:=
\left\{
 f:\mathscr V\to\mathbb R:
 \sup_{m\ge 0}\mathcal E_m^W(f|_{V(\Xi^m)},f|_{V(\Xi^m)})<\infty
\right\},
\]
and, for $f\in\mathcal F_{\mathscr V}$,
\[
\mathcal E_{\mathscr V}(f,f):=
\lim_{m\to\infty}
\mathcal E_m^W(f|_{V(\Xi^m)},f|_{V(\Xi^m)}).
\]
The limit exists because the trace compatibility of Lemma~\ref{lem:sec31_trace_weighted} makes the displayed sequence non-decreasing.
By the standard resistance-form construction from compatible finite traces, $(\mathcal E_{\mathscr V},\mathcal F_{\mathscr V})$ is a resistance form on $\mathscr V$, and its resistance metric is precisely $\Reff_*$; see \cite{kigami2001analysis}.
Let $\Xi_{\mathrm R}$ be the $\Reff_*$-completion of $\mathscr V$.

We first show that the identity map on $\mathscr V$ extends to a homeomorphism
\[
\iota:\Xi_{\mathrm M}\longrightarrow \Xi_{\mathrm R}.
\]
Fix $x\in\Xi_{\mathrm M}$ and let $(x_j)_{j\ge 1}\subset\mathscr V$ be any sequence converging to $x$ in the $d_{\Xi_{\mathrm M}}$-topology.
If $x\in\mathscr V^\infty$, then for every fixed $m$, all sufficiently large $x_j$ belong to the unique level-$m$ cell $\Lambda_m(x)$ containing $x$.
By Lemma~\ref{lem:sec32_point_boundary},
\[
\Reff_*(x_j,x_{j'})\le 2C_\partial\rho(\Psi)^{-m}
\]
for all sufficiently large $j,j'$.
If $x\in\mathscr V$, then for every fixed $m$ greater than the birth level of $x$, all sufficiently large $x_j$ belong to the level-$m$ star of $x$.
Since $x$ is a terminal of each level-$m$ cell in this star, Lemma~\ref{lem:sec32_point_boundary} gives
\[
\Reff_*(x_j,x)\le C_\partial\rho(\Psi)^{-m}
\]
for all sufficiently large $j$, and hence again
\[
\Reff_*(x_j,x_{j'})\le 2C_\partial\rho(\Psi)^{-m}.
\]
Letting $m\to\infty$ shows that $(x_j)$ is Cauchy in $\Reff_*$.
The same estimate applied to two interlaced approximating sequences shows that the limit in $\Xi_{\mathrm R}$ is independent of the approximating sequence.
Thus $\iota$ is well defined.
The same argument also shows that $\iota$ is continuous.

To prove injectivity, let $x\ne y$ in $\Xi_{\mathrm M}$.
Choose $f\in C(\Xi_{\mathrm M})$ with $f(x)=0$ and $f(y)=1$.
By Lemma~\ref{lem:sec31_continuous_harmonics}, there exists a continuous harmonic spline $u$ such that $\|u-f\|_\infty<1/3$.
Then $u(x)\ne u(y)$, and $u$ has finite energy on $\mathscr V$.
Finite-energy functions are continuous in the resistance metric, so $u$ extends continuously to $\Xi_{\mathrm R}$, and the extension agrees with $u$ on all of $\Xi_{\mathrm M}$ because $\mathscr V$ is dense and both sides are continuous.
Therefore $\iota(x)\ne\iota(y)$.

The image $\iota(\Xi_{\mathrm M})$ is compact in the metric space $\Xi_{\mathrm R}$, hence closed.
It contains $\mathscr V$, and $\mathscr V$ is dense in $\Xi_{\mathrm R}$ by definition of completion.
Therefore $\iota(\Xi_{\mathrm M})=\Xi_{\mathrm R}$.
Thus $\iota$ is a homeomorphism.
Transporting the resistance form from $\Xi_{\mathrm R}$ to $\Xi_{\mathrm M}$ gives the claimed form $(\mathcal E_R,\mathcal F_R)$.
The identity $\Reff_{\mathcal E_R}|_{\mathscr V\times\mathscr V}=\Reff_*$ follows from the construction.
\end{proof}

\begin{lemma}\label{lem:sec32_common_core}
The union of the continuous harmonic splines from Lemma~\ref{lem:sec31_continuous_harmonics} is a form core for the $L^2(\Xi_{\mathrm M},\mu_{\Xi_{\mathrm M}})$ realisation of $(\mathcal E_R,\mathcal F_R)$.
\end{lemma}

\begin{proof}
Fix $q\in V(\Xi^0)$ and let $\mathcal F_R^0:=\{u\in\mathcal F_R:u(q)=0\}$.
For each $m\ge 0$, consider the subspace of $\mathcal F_R^0$ consisting of level-$m$ harmonic splines.
These subspaces are increasing.
Trace theory identifies the $\mathcal E_R$-orthogonal complement of the level-$m$ harmonic-spline subspace with the functions in $\mathcal F_R^0$ whose trace on $V(\Xi^m)$ is zero.
Hence, if a function in $\mathcal F_R^0$ is orthogonal to the harmonic splines at every level, then it vanishes on $\bigcup_m V(\Xi^m)=\mathscr V$.
Finite-energy functions are continuous in the resistance topology, and Theorem~\ref{thm:sec32_resistance_metric} identifies this topology with the natural topology.
Therefore such a function is identically zero.
Thus the union over all levels of the harmonic-spline subspaces is dense in $\mathcal F_R^0$ for the $\mathcal E_R$-norm.

For $u\in\mathcal F_R^0$, let $P_m^Ru$ be the level-$m$ harmonic spline whose trace on $V(\Xi^m)$ equals $u|_{V(\Xi^m)}$.
Then $u-P_m^Ru$ vanishes on $V(\Xi^m)$.
If $x\in\Xi_{\mathrm M}$ lies in a level-$m$ cell $\Lambda$, choose one terminal $a_{\Lambda}\in V(\Xi^m)$.
By Lemma~\ref{lem:sec32_point_boundary} and continuity in the $\Reff_{\mathcal E_R}$-metric,
\[
\Reff_{\mathcal E_R}(x,a_{\Lambda})\le C_\partial\rho(\Psi)^{-m}.
\]
Therefore
\[
|u(x)-P_m^Ru(x)|^2
\le
\mathcal E_R(u-P_m^Ru,u-P_m^Ru)\,\Reff_{\mathcal E_R}(x,a_{\Lambda})
\le
C_\partial\rho(\Psi)^{-m}\mathcal E_R(u,u).
\]
It follows that $\|u-P_m^Ru\|_{L^2(\mu_{\Xi_{\mathrm M}})}\to 0$.
Thus the harmonic splines are dense in $\mathcal F_R^0$ for the form norm and in $L^2$.
Adding constants gives the corresponding core statement for $\mathcal F_R$.
\end{proof}

\begin{theorem}\label{thm:sec32_brownian_motion}
The $L^2(\Xi_{\mathrm M},\mu_{\Xi_{\mathrm M}})$ realisation of $(\mathcal E_R,\mathcal F_R)$ coincides with $(\mathcal E_W,\mathcal F_W)$.
Consequently, the diffusion $W$ is the Brownian motion associated with the resistance form $(\mathcal E_R,\mathcal F_R)$.
\end{theorem}

\begin{proof}
By Lemma~\ref{lem:sec32_common_core}, the harmonic splines form a core for the $L^2$ realisation of $(\mathcal E_R,\mathcal F_R)$.
By Theorem~\ref{thm:sec31_limit_dirichlet_form}, the same family is a core for $(\mathcal E_W,\mathcal F_W)$.
If $u=P_mh$ is a level-$m$ harmonic spline, then both forms give
\[
\mathcal E_R(u,u)=\mathcal E_m^W(h,h)=\mathcal E_W(u,u).
\]
Thus the two closed forms agree on a common core and hence agree everywhere.
The associated symmetric Hunt process is therefore the Brownian motion of the resistance form; see \cite{kigami2001analysis,fukushima1994dirichlet}.
\end{proof}

We now clarify the role of the metric in the positive-resistance regime.
The assumption $\dimresis(\Xi)>0$ does not produce a second stochastic process;
rather, it gives a second geometric realisation of the same limiting Dirichlet form.
Indeed, Theorem~\ref{thm:sec32_resistance_metric} constructs a resistance form
$(\mathcal E_R,\mathcal F_R)$ whose resistance metric induces the same topology as the
natural metric.

Moreover, Theorem~\ref{thm:sec32_brownian_motion} shows that the
$L^2(\Xi_{\mathrm M},\mu_{\Xi_{\mathrm M}})$-realisation of this resistance form coincides with the limiting
Dirichlet form $(\mathcal E_W,\mathcal F_W)$ obtained from the renormalised graph energies.
Consequently the associated semigroups, Hunt processes, and heat kernels agree:
\[
    P_t^R = P_t^W,
    \qquad
    q_t^\Xi(x,y)=p_t^\Xi(x,y),
\]
after choosing their jointly continuous versions.
Thus the stochastic dynamics do not depend on whether we describe the compact state space using the natural metric or the resistance metric.
The metric changes the geometric language in which we state estimates, for example natural balls versus resistance balls, but it does not change the underlying diffusion.
This is why the next subsection can reinterpret the same heat kernel and the same spectral dimension in resistance-geometric terms without introducing any new process.

\begin{figure}[bt]
    \centering
    \begin{minipage}[t]{0.6\linewidth}
        \includegraphics[width=\linewidth]{figure/Xi_random_walk.png}
        \caption{A sample trace of a \(10^6\)-step simple random walk on the level-6 approximation \(\Xi^6\) of the canonical Xi graph, drawn in its natural embedding in \(\mathbb Z^2\). The underlying graph is shown in grey and the trace in red.}
    \end{minipage}
\hfill
    \begin{minipage}[t]{0.35\linewidth}
        \includegraphics[width=\linewidth]{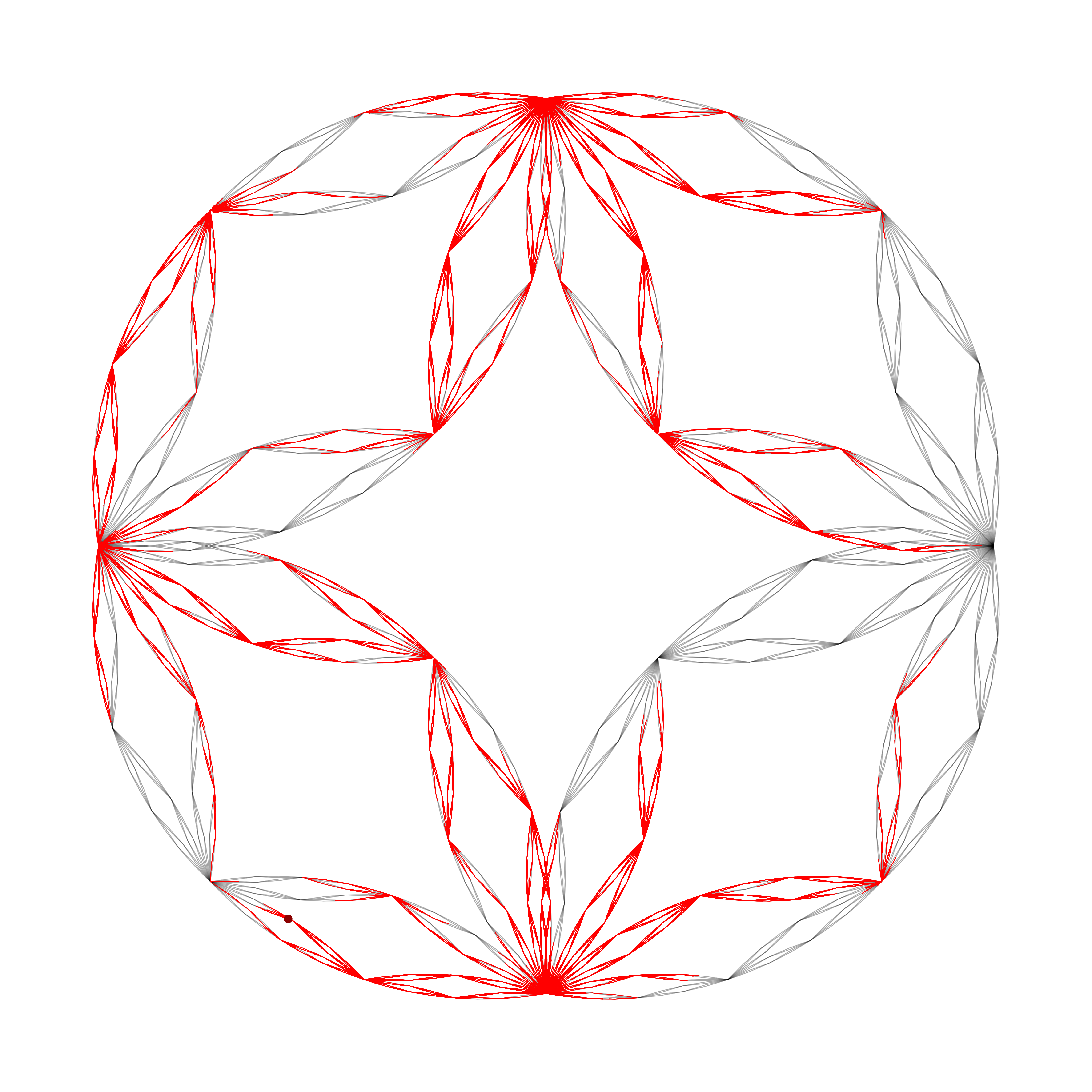}
        \caption{A sample trace of a $10^5$-step simple random walk on the level-8 approximation $\Xi^8$ of DHL, drawn in the symmetric planar embedding used for visualisation. The underlying graph is shown in grey and the trace in red.}
    \end{minipage}
\end{figure}

\subsection{Heat kernel for the diffusion limit}\label{subsec:heat_kernel_diffusion}

In this subsection we work with the diffusion $W=(W_t)_{t\ge 0}$ associated with $(\mathcal E_W,\mathcal F_W)$ and write $P_t$ for its semigroup.
By Corollary~\ref{cor:sec31_continuous_stay}, for every $x\in \Xi_{\mathrm M}$ there are constants $c_x,p_x,r_x>0$ such that
\[
\mathbb P_x\bigl(\tau_{B_{\Xi_{\mathrm M}}(x,r)}>c_xr^{\dimwalk(\Xi)}\bigr)\ge p_x
\]
for all $r\in(0,r_x)$.

\begin{lemma}\label{lem:heat-volume}
For every $x\in \mathscr V$ there are $c_x,C_x,r_0(x)>0$ such that
\[
c_x r^{\dimbox(\Xi)\left(1-\frac{1}{\dimdeg(\Xi)}\right)}
\le
\mu_{\Xi_{\mathrm M}}(B_{\Xi_{\mathrm M}}(x,r))
\le
C_x r^{\dimbox(\Xi)\left(1-\frac{1}{\dimdeg(\Xi)}\right)}
\]
for all $r\in(0,r_0(x))$.
For every $\varepsilon>0$ and for $\mu_{\Xi_{\mathrm M}}$-almost every $x\in \mathscr V^\infty$ there are $c_x,C_{x,\varepsilon},r_0(x)>0$ such that
\[
c_x r^{\dimbox(\Xi)}
\le
\mu_{\Xi_{\mathrm M}}(B_{\Xi_{\mathrm M}}(x,r))
\le
C_{x,\varepsilon} |\log r|^{\frac{1}{\dimdeg(\Xi)-1}+\varepsilon} r^{\dimbox(\Xi)}
\]
for all $r\in(0,r_0(x))$.
\end{lemma}

\begin{proof}
Fix $x=\pi(v)\in\mathscr V$, and let $k$ be the birth level of $v$.
Choose $m=m(r)$ so that $\rho_{\min}(\mathcal D)^{-(m+1)}<r\le \rho_{\min}(\mathcal D)^{-m}$.
On level $n$, the corresponding graph-metric radius is $R_n:=r\,\diam(\Xi^n)\asymp \rho_{\min}(\mathcal D)^{\,n-m}$.
Applying Lemma~\ref{lem:ball_volume_k_born} with parameter $n-m$ gives
\[
|V(B_{\Xi^n}(v,R_n))|
\asymp
\rho(\mathbf N)^{m-k}\rho(\mathbf M)^{n-m}.
\]
Dividing by $|V(\Xi^n)|\asymp \rho(\mathbf M)^n$ and letting $n\to\infty$ yields
\[
\mu_{\Xi_{\mathrm M}}(B(x,r))
\asymp
\rho(\mathbf N)^{m-k}\rho(\mathbf M)^{-m}.
\]
Since $\rho_{\min}(\mathcal D)^{-m}\asymp r$, this is exactly the finite-born exponent.
Since only finitely many birth types occur, the constants in this two-sided estimate may be chosen uniformly over all finite-born points, by Lemma~\ref{lem:ball_volume_k_born}.

Now fix $\varepsilon>0$ and consider $x\in\mathscr V^\infty$.
For the lower bound, let $\Lambda_m(x)$ be the unique level-$m$ cell containing $x$.
If $m=m(r)$ is chosen as above, then $\Lambda_m(x)\subset B(x,r)$ and Lemma~\ref{lem:GHP_convergence} gives
\[
\mu_{\Xi_{\mathrm M}}(\Lambda_m(x))
\asymp
\rho(\mathbf M)^{-m}
\asymp
r^{\dimbox(\Xi)}.
\]

For the upper bound, let $S_j$ denote the level-$j$ skeleton and fix a geometric constant $C_*>0$ large enough that every level-$m$ cell lies within $C_*\rho_{\min}(\mathcal D)^{-m}$ of at least one of its terminals.
For integers $m\ge 1$ and $L\le m$, set
\[
E_{m,L}
:=
\{y\in\Xi_{\mathrm M}: d_{\Xi_{\mathrm M}}(y,S_{m-L})\le C_*\rho_{\min}(\mathcal D)^{-m}\}.
\]
We first estimate $\mu_{\Xi_{\mathrm M}}(E_{m,L})$.
If $v$ is a $k$-born point with $k\le m-L$, then the finite-born estimate just proved gives
\[
\mu_{\Xi_{\mathrm M}}(B(v,C_*\rho_{\min}(\mathcal D)^{-m}))
\lesssim
\rho(\mathbf N)^{m-k}\rho(\mathbf M)^{-m}.
\]
The number of $k$-born points is $O(\rho(\mathbf M)^k)$ by Lemma~\ref{lem:basic_prop}.
Summing over $k\le m-L$, we obtain
\[
\mu_{\Xi_{\mathrm M}}(E_{m,L})
\lesssim
\sum_{k=0}^{m-L}
\rho(\mathbf M)^k\rho(\mathbf N)^{m-k}\rho(\mathbf M)^{-m}
\lesssim
\Bigl(\frac{\rho(\mathbf N)}{\rho(\mathbf M)}\Bigr)^L.
\]

Now set
\[
L_m
:=
\Bigl\lceil
\frac{(1+\varepsilon)\log m}{\log(\rho(\mathbf M)/\rho(\mathbf N))}
\Bigr\rceil.
\]
Since $L_m=O(\log m)$, we have $L_m\le m$ for all large $m$, and we tacitly restrict to such $m$.
Then $\sum_{m\ge 1}\mu_{\Xi_{\mathrm M}}(E_{m,L_m})<\infty$.
By Borel--Cantelli, for $\mu_{\Xi_{\mathrm M}}$-almost every $x\in\mathscr V^\infty$ there is $m_0(x)$ such that $x\notin E_{m,L_m}$ for all $m\ge m_0(x)$.
Fix such an $x$.
For each large $m$, choose a terminal $v_m(x)$ of the unique level-$m$ cell containing $x$ with $d_{\Xi_{\mathrm M}}(x,v_m(x))\le C_*\rho_{\min}(\mathcal D)^{-m}$.
Since $x\notin E_{m,L_m}$, this terminal cannot belong to the level-$(m-L_m)$ skeleton, so its birth level is at least $m-L_m+1$.
Hence
\[
B(x,c\rho_{\min}(\mathcal D)^{-m})
\subset
B(v_m(x),C\rho_{\min}(\mathcal D)^{-m}).
\]
Write $b_m$ for the birth level of $v_m(x)$, so that $m-L_m<b_m\le m$, and let $c_1$ be a fixed integer, depending only on $C$ and the constants of the uniform finite-born estimate, such that this estimate applies to every ball of radius $C\rho_{\min}(\mathcal D)^{-m}$ centred at a vertex of birth level at most $m-c_1$.
If $b_m\le m-c_1$, the uniform finite-born estimate gives
\[
\mu_{\Xi_{\mathrm M}}(B(x,c\rho_{\min}(\mathcal D)^{-m}))
\lesssim
\rho(\mathbf N)^{m-b_m}\rho(\mathbf M)^{-m}
\le
\rho(\mathbf N)^{L_m}\rho(\mathbf M)^{-m}.
\]
If instead $b_m>m-c_1$, then $B(v_m(x),C\rho_{\min}(\mathcal D)^{-m})\subset B(v_m(x),C\rho_{\min}(\mathcal D)^{-b_m})$, and the uniform finite-born estimate at the birth scale gives
\[
\mu_{\Xi_{\mathrm M}}(B(x,c\rho_{\min}(\mathcal D)^{-m}))
\lesssim
\rho(\mathbf M)^{-b_m}
\le
\rho(\mathbf M)^{c_1}\rho(\mathbf M)^{-m}.
\]
In either case,
\[
\mu_{\Xi_{\mathrm M}}(B(x,c\rho_{\min}(\mathcal D)^{-m}))
\lesssim
\rho(\mathbf N)^{L_m}\rho(\mathbf M)^{-m}.
\]
Since $\rho(\mathbf N)^{L_m}\le \rho(\mathbf N)\,m^{\frac{1+\varepsilon}{\dimdeg(\Xi)-1}}$ and $m\asymp |\log r|$, this yields the claimed upper bound.
\end{proof}

\begin{lemma}\label{lem:heat-exit}
For every $x\in \Xi_{\mathrm M}$ there are $C_x,r_0(x)>0$ such that
\[
\sup_{y\in B_{\Xi_{\mathrm M}}(x,r)}
\mathbb E_y[\tau_{B_{\Xi_{\mathrm M}}(x,r)}]
\le
C_x r^{\dimwalk(\Xi)}
\]
for all $r\in(0,r_0(x))$.
\end{lemma}

\begin{proof}
Fix $x\in \Xi_{\mathrm M}$ and $r>0$ small.
If $y\in B_{\Xi_{\mathrm M}}(x,r)$, then $\tau_{B(x,r)}\le \tau_{B(y,2r)}$, so it is enough to bound $\mathbb E_y[\tau_{B(y,2r)}]$ uniformly in $y$.
Choose $y_n\in V(\Xi^n)$ with $y_n\to y$ and let $Y^{(n)}$ be the Poissonised walk started from $y_n$.
By Theorem~\ref{thm:sec31_ghps} and the Skorokhod representation theorem \cite{billingsley1999convergence}, we may couple $Y^{(n)}$ with $W$ started from $y$ so that $Y^{(n)}\to W$ almost surely in $D([0,T],\Xi_{\mathrm M})$ for every fixed $T$.

Fix $\eta>0$.
On the coupling event that $d_{\Xi_{\mathrm M}}(y_n,y)<\eta$ and $\sup_{0\le s\le T}d_{\Xi_{\mathrm M}}(Y_s^{(n)},W_s)<\eta$, every path segment of $W$ staying inside $B(y,2r)$ up to time $t\le T$ is approximated by a path segment of $Y^{(n)}$ staying inside $B(y_n,2r+2\eta)$ up to time $t$.
Hence
\[
\tau_{B(y,2r)}(W)\wedge T
\le
\liminf_{n\to\infty}\tau_{B(y_n,2r+2\eta)}(Y^{(n)})\wedge T
\]
almost surely.
Fatou's lemma therefore gives
\[
\mathbb E_y[\tau_{B(y,2r)}\wedge T]
\le
\liminf_{n\to\infty}
\mathbb E_{y_n}[\tau_{B(y_n,2r+2\eta)}(Y^{(n)})].
\]

The continuous-time exit time of $Y^{(n)}$ is the discrete exit time divided by $a_n$.
Let $w_n$ be a terminal of the cell containing $y_n$ at the level for which the star of $w_n$ at the corresponding depth contains $B(y_n,2r+2\eta)$ with its boundary outside it.
The uniform cell and star bounds established in the proof of Lemma~\ref{lem:sec31_one_step_stay}, together with Theorem~\ref{thm:exit_time_scale}, then yield
\[
\mathbb E_{y_n}[\tau_{B(y_n,2r+2\eta)}(Y^{(n)})]
\lesssim
(2r+2\eta)^{\dimwalk(\Xi)},
\]
uniformly in $n$ and $y_n$.
Letting $n\to\infty$, then $T\to\infty$, and finally $\eta\downarrow 0$, we obtain $\mathbb E_y[\tau_{B(y,2r)}]\lesssim r^{\dimwalk(\Xi)}$.
\end{proof}

\begin{proposition}\label{prop:heat-kernel}
There is a jointly continuous heat kernel $p_t^\Xi(x,y)$ of $W$ on $(0,\infty)\times \Xi_{\mathrm M}\times \Xi_{\mathrm M}$.
Moreover, there are constants $C,t_0>0$ such that
\[
0<p_t^\Xi(x,y)\le C t^{-\frac{\dimbox(\Xi)}{\dimwalk(\Xi)}}
\]
for all $x,y\in \Xi_{\mathrm M}$ and all $t\in(0,t_0]$.
\end{proposition}

\begin{proof}
By Lemma~\ref{lem:GHP_convergence}, every level-$n$ cell has $\mu_{\Xi_{\mathrm M}}$-mass $\asymp \rho(\mathbf M)^{-n}$.
By Theorem~\ref{thm:sec31_limit_dirichlet_form}, the trace form on such a cell is renormalised by the factor $\rho(\Psi)^n$.
Thus the time scale of a level-$n$ cell is $(\rho(\mathbf M)\rho(\Psi))^{-n}$.
The multiscale Nash argument of Hambly and Kumagai therefore gives the ultracontractive estimate
\[
\|P_t\|_{L^1(\mu_{\Xi_{\mathrm M}})\to L^\infty(\mu_{\Xi_{\mathrm M}})}
\le
C t^{-\frac{\dimbox(\Xi)}{\dimwalk(\Xi)}}
\]
for small $t$; see \cite{hambly1999transition,hambly2010diffusion,kumagai2004heat}.

It remains to prove continuity.
Fix $\lambda>0$ and $f\in L^\infty(\mu_{\Xi_{\mathrm M}})$, and set
\[
U_\lambda f(x):=\mathbb E_x\Bigl[\int_0^\infty e^{-\lambda t}f(W_t)\,\mathrm dt\Bigr].
\]
Fix $x_0\in \Xi_{\mathrm M}$ and $r>0$ small.
Applying the strong Markov property at the exit time from $B(x_0,r)$ and using Lemma~\ref{lem:heat-exit}, we see that the oscillation of $U_\lambda f$ on $B(x_0,r/2)$ is bounded by the oscillation of a bounded harmonic function on $B(x_0,r)$ plus an error of order $r^{\dimwalk(\Xi)}\|f\|_\infty$.
Since harmonic functions are uniform limits of harmonic splines, they are continuous.
Hence $U_\lambda$ maps $L^\infty(\mu_{\Xi_{\mathrm M}})$ to $C(\Xi_{\mathrm M})$.

Standard Dirichlet-form theory on compact spaces now yields a jointly continuous heat kernel, and the same ultracontractive estimate gives the stated bound; see \cite{fukushima1994dirichlet,kumagai2004heat}.
\end{proof}

\begin{theorem}\label{thm:general-lower}
For every $x\in \Xi_{\mathrm M}$ there are constants $c_x,C_x,t_0(x)>0$ such that
\[
p_t^\Xi(x,x)
\ge
\frac{c_x}{\mu_{\Xi_{\mathrm M}}(B_{\Xi_{\mathrm M}}(x,C_x t^{1/\dimwalk(\Xi)}))}
\]
for all $t\in(0,t_0(x))$.
\end{theorem}

\begin{proof}
Fix $x\in \Xi_{\mathrm M}$.
By Corollary~\ref{cor:sec31_continuous_stay}, there are $a_x,p_x,r_x>0$ such that
\[
\mathbb P_x\bigl(\tau_{B_{\Xi_{\mathrm M}}(x,r)}>a_xr^{\dimwalk(\Xi)}\bigr)\ge p_x
\]
for all $r\in(0,r_x)$.
If $t\le a_xr^{\dimwalk(\Xi)}$, then $\mathbb P_x(W_t\in B(x,r))\ge p_x$.
Using the heat kernel and Cauchy--Schwarz, we obtain
\[
p_x^2
\le
\mu_{\Xi_{\mathrm M}}(B(x,r))\,p_{2t}^\Xi(x,x).
\]
Applying this with $t$ replaced by $t/2$ and choosing $r$ comparable to $(t/2)^{1/\dimwalk(\Xi)}$ proves the result, after enlarging $C_x$.
\end{proof}

\begin{theorem}\label{thm:diffusion-heat}
For every $x\in \mathscr V$ there are constants $c_x,C_x,t_0(x)>0$ such that
\[
c_x t^{-\frac{\dimbox(\Xi)\left(1-\frac{1}{\dimdeg(\Xi)}\right)}{\dimwalk(\Xi)}}
\le
p_t^\Xi(x,x)
\le
C_x t^{-\frac{\dimbox(\Xi)\left(1-\frac{1}{\dimdeg(\Xi)}\right)}{\dimwalk(\Xi)}}
\]
for all $t\in(0,t_0(x))$.
For every $\varepsilon>0$ and for $\mu_{\Xi_{\mathrm M}}$-almost every $x\in \mathscr V^\infty$ there are constants $c_{x,\varepsilon},C_x,t_0(x)>0$ such that
\[
c_{x,\varepsilon} |\log t|^{-\frac{1}{\dimdeg(\Xi)-1}-\varepsilon} t^{-\frac{\dimbox(\Xi)}{\dimwalk(\Xi)}}
\le
p_t^\Xi(x,x)
\le
C_x t^{-\frac{\dimbox(\Xi)}{\dimwalk(\Xi)}}
\]
for all $t\in(0,t_0(x))$.
\end{theorem}

\begin{proof}
The lower bounds follow immediately from Theorem~\ref{thm:general-lower} and Lemma~\ref{lem:heat-volume}.
The upper bound at $\mu_{\Xi_{\mathrm M}}$-almost every infinite-born point is the global estimate from Proposition~\ref{prop:heat-kernel}.

It remains to prove the upper bound at a fixed finite-born point $x$.
Let $k$ be the birth level of $x$.
For each $n\ge k$, let $\Lambda_{n,1},\dots,\Lambda_{n,d_n}$ be the level-$n$ cells incident to $x$, and let $z_{n,j}$ be the other terminal of $\Lambda_{n,j}$.
By the degree-growth estimate from Section~\ref{sec:random_walks},
\[
d_n\asymp \rho(\mathbf N)^{n-k}.
\]
Each $\Lambda_{n,j}$ has $\mu_{\Xi_{\mathrm M}}$-mass $\asymp \rho(\mathbf M)^{-n}$, and the natural time scale is
\[
T_n:=(\rho(\mathbf M)\rho(\Psi))^{-n}.
\]

Because only finitely many prototype cells occur, the Dirichlet Laplacian on a prototype cell with zero boundary values at its two terminals has strictly positive first eigenvalue.
After rescaling, there is a constant $C_1$ such that every $f\in \mathcal F_W\cap C(\Lambda_{n,j})$ with $f(x)=f(z_{n,j})=0$ satisfies
\[
\int_{\Lambda_{n,j}}f^2\,\mathrm d\mu_{\Xi_{\mathrm M}}
\le
C_1T_n\mathcal E_{W,\Lambda_{n,j}}(f,f).
\]
Likewise, if $\phi_{n,j}$ is the harmonic function on $\Lambda_{n,j}$ with $\phi_{n,j}(x)=1$ and $\phi_{n,j}(z_{n,j})=0$, then
\[
\int_{\Lambda_{n,j}}\phi_{n,j}^2\,\mathrm d\mu_{\Xi_{\mathrm M}}
\asymp
\rho(\mathbf M)^{-n}.
\]

Now let $u_s(y):=p_s^\Xi(x,y)$ and $q(t):=p_t^\Xi(x,x)$.
Replacing $u_s$ on each $\Lambda_{n,j}$ by its harmonic interpolant with the same boundary values at $x$ and $z_{n,j}$ gives a function $h_{n,s}$ with
\[
\|u_s-h_{n,s}\|_2^2
\lesssim
T_n\mathcal E_W(u_s,u_s)
\]
and
\[
\|h_{n,s}\|_2^2
\gtrsim
d_n\rho(\mathbf M)^{-n}u_s(x)^2.
\]
Since $\|h_{n,s}\|_2^2\le 2\|u_s\|_2^2+2\|u_s-h_{n,s}\|_2^2$, we obtain
\[
q(s)^2
\lesssim
\frac{\rho(\mathbf M)^n}{d_n}\bigl(q(2s)+T_n\mathcal E_W(u_s,u_s)\bigr).
\]
Choose $s\in[t/2,t]$ so that $\mathcal E_W(u_s,u_s)\le q(t)/t$, and then choose $n=n(t)$ with $T_n\le t<T_{n-1}$.
This gives
\[
q(t)\lesssim \frac{\rho(\mathbf M)^n}{d_n}\lesssim \Bigl(\frac{\rho(\mathbf M)}{\rho(\mathbf N)}\Bigr)^n.
\]
Since $t\asymp (\rho(\mathbf M)\rho(\Psi))^{-n}$, this yields exactly the finite-born exponent.
\end{proof}

It is worth comparing the logarithmic correction in Theorem~\ref{thm:diffusion-heat} with the argument of Hambly and Kumagai.
In \cite{hambly2010diffusion}, the radius is taken strictly below the cell scale, so that the ball avoids the whole skeleton and lies in a single cell; the price is a radius deficit raised to the power $\dimbox(\Xi)$, and the optimised outcome of that route is the logarithmic exponent $\frac{\dimdeg(\Xi)}{\dimdeg(\Xi)-1}$, which equals $2$ for the DHL.
The proof of Lemma~\ref{lem:heat-volume} instead keeps the radius at the cell scale and covers the ball by the star of a terminal, so the same fluctuation enters only linearly, through the degree $\rho(\mathbf N)^{L_m}$.
The resulting exponent $\frac{1}{\dimdeg(\Xi)-1}$ is intrinsic: it is the large-deviation rate of the age of the nearest terminal under $\mu_{\Xi_{\mathrm M}}$.
A renewal argument, which we do not pursue here, shows that for the DHL infinitely many scales satisfy a matching lower bound $\mu_{\Xi_{\mathrm M}}(B(x,r))\ge c|\log r|^{1-\varepsilon}r^{2}$ at $\mu_{\Xi_{\mathrm M}}$-almost every point of $\mathscr V^\infty$, because the age of the retained terminal performs a geometric renewal; hence the exponent cannot be improved below $\frac{1}{\dimdeg(\Xi)-1}$.

We define the local diffusion spectral dimension, whenever the limit exists, by
\[
\dimspec^{(\mathrm N)}(\Xi_{\mathrm M}:x)
:=
-2 \lim_{t \downarrow 0} \frac{\log p_t^\Xi(x,x)}{\log t}.
\]

\begin{theorem}\label{thm:diffusion-spectral}
With the convention that $\dimdeg(\Xi)=\infty$ in the non-scale-free case, we have
\[
\dimspec^{(\mathrm N)}(\Xi_{\mathrm M}:x)
=
\begin{cases}
\displaystyle
\frac{2\dimbox(\Xi)\left(1-\frac{1}{\dimdeg(\Xi)}\right)}{\dimwalk(\Xi)},
&
x \in \mathscr V,
\\[2mm]
\displaystyle
\frac{2\dimbox(\Xi)}{\dimwalk(\Xi)},
&
x \in \mathscr V^\infty
\quad
(\mu_{\Xi_{\mathrm M}}\text{-almost everywhere}).
\end{cases}
\]
\end{theorem}

\begin{proof}
This is immediate from Theorem~\ref{thm:diffusion-heat}.
\end{proof}

\begin{proposition}\label{prop:final_recurrent}
Simple random walk on $\Xi$ is recurrent if and only if $\dimspec^{(\mathrm N)}(\Xi_{\mathrm M}:x)<2$ for $\mu_{\Xi_{\mathrm M}}$-almost every $x\in \mathscr V^\infty$.
\end{proposition}

\begin{proof}
By Theorem~\ref{thm:Einstein}, for $\mu_{\Xi_{\mathrm M}}$-almost every $x\in \mathscr V^\infty$ we have
\[
\dimspec^{(\mathrm N)}(\Xi_{\mathrm M}:x)
=
\frac{2\dimbox(\Xi)}{\dimbox(\Xi)+\dimresis(\Xi)}.
\]
This quantity is less than $2$ exactly when $\dimresis(\Xi)>0$.
By Proposition~\ref{prop:asymp_res_recurrence_rhoPsi}, this is exactly the recurrence regime of the simple random walk on $\Xi$.
\end{proof}

\begin{example}
Let $\Xi$ be the one-colour EIGS generating the diamond hierarchical lattice.
Then $\dimbox(\Xi)=2$, $\dimdeg(\Xi)=2$, $\dimresis(\Xi)=0$, and $\dimwalk(\Xi)=2$.
Hence Theorem~\ref{thm:diffusion-heat} gives the following on-diagonal estimates.

For every $x\in\mathscr V$, there are $c_x,C_x>0$ and $t_0(x)>0$ such that
\[
c_x t^{-1/2}\le p_t^\Xi(x,x)\le C_x t^{-1/2}
\]
for all $t\in(0,t_0(x))$.

For $\mu_{\Xi_{\mathrm M}}$-almost every $x\in\mathscr V^\infty$, there are $c'_{x,\varepsilon},C'_x>0$ and $t_0(x)>0$ such that
\[
c'_{x,\varepsilon}|\log t|^{-1-\varepsilon}t^{-1}
\le
p_t^\Xi(x,x)
\le
C'_x t^{-1}
\]
for all $t\in(0,t_0(x))$.

This agrees with the heat-kernel estimates of Hambly and Kumagai in \cite{hambly2010diffusion} at polynomial order, and sharpens the logarithmic correction in their typical-point lower bound from $|\log t|^{-2-\varepsilon}$ to $|\log t|^{-1-\varepsilon}$.
\end{example}


\bigskip 

Assume $\dimresis(\Xi)>0$.
We write $q_t^\Xi(x,y)$ for the heat kernel of the Brownian motion associated with $(\mathcal E_R,\mathcal F_R)$, and define
\[
\dimspec^{(\mathrm R)}(\Xi_{\mathrm M}:x)
:=
-2 \lim_{t \downarrow 0} \frac{\log q_t^\Xi(x,x)}{\log t},
\]
whenever the limit exists.

\begin{corollary}\label{cor:brownian-heat}
Assume $\dimresis(\Xi)>0$.
For every $x\in \mathscr V$ there are constants $c_x,C_x,t_0(x)>0$ such that
\[
c_x t^{-\frac{\dimbox(\Xi)\left(1-\frac{1}{\dimdeg(\Xi)}\right)}{\dimwalk(\Xi)}}
\le
q_t^\Xi(x,x)
\le
C_x t^{-\frac{\dimbox(\Xi)\left(1-\frac{1}{\dimdeg(\Xi)}\right)}{\dimwalk(\Xi)}}
\]
for all $t\in(0,t_0(x))$.
For every $\varepsilon>0$ and for $\mu_{\Xi_{\mathrm M}}$-almost every $x\in \mathscr V^\infty$ there are constants $c_{x,\varepsilon},C_x,t_0(x)>0$ such that
\[
c_{x,\varepsilon} |\log t|^{-\frac{1}{\dimdeg(\Xi)-1}-\varepsilon} t^{-\frac{\dimbox(\Xi)}{\dimwalk(\Xi)}}
\le
q_t^\Xi(x,x)
\le
C_x t^{-\frac{\dimbox(\Xi)}{\dimwalk(\Xi)}}
\]
for all $t\in(0,t_0(x))$.
\end{corollary}

\begin{proof}
By Theorem~\ref{thm:sec32_brownian_motion}, the Brownian motion of $(\mathcal E_R,\mathcal F_R)$ is the same $L^2(\Xi_{\mathrm M},\mu_{\Xi_{\mathrm M}})$ process as the diffusion $W$.
Hence $q_t^\Xi(x,y)=p_t^\Xi(x,y)$ for all $t>0$ and all $x,y\in \Xi_{\mathrm M}$.
The result is therefore exactly Theorem~\ref{thm:diffusion-heat}.
\end{proof}

\begin{theorem}\label{thm:brownian-spectral}
Assume $\dimresis(\Xi)>0$.
For every $x\in \mathscr V$,
\[
\dimspec^{(\mathrm R)}(\Xi_{\mathrm M}:x)
=
\dimspec^{(\mathrm N)}(\Xi_{\mathrm M}:x)
=
\frac{2\dimbox(\Xi)\left(1-\frac{1}{\dimdeg(\Xi)}\right)}{\dimwalk(\Xi)}.
\]
For $\mu_{\Xi_{\mathrm M}}$-almost every $x\in \mathscr V^\infty$,
\[
\dimspec^{(\mathrm R)}(\Xi_{\mathrm M}:x)
=
\dimspec^{(\mathrm N)}(\Xi_{\mathrm M}:x)
=
\frac{2\dimbox(\Xi)}{\dimwalk(\Xi)}
=
\frac{2\frac{\dimbox(\Xi)}{\dimresis(\Xi)}}{1+\frac{\dimbox(\Xi)}{\dimresis(\Xi)}}.
\]
\end{theorem}

\begin{proof}
By Theorem~\ref{thm:sec32_brownian_motion}, we have $q_t^\Xi(x,y)=p_t^\Xi(x,y)$.
Hence the two spectral dimensions agree wherever the limit exists.
The formulas then follow from Theorem~\ref{thm:diffusion-spectral}.
The final identity is just the Einstein relation $\dimwalk(\Xi)=\dimbox(\Xi)+\dimresis(\Xi)$ from Section~\ref{sec:random_walks}.
\end{proof}

\begin{remark}\label{rem:three_resistance_regimes}
The sign of $\dimresis(\Xi)$ separates EIGS into three qualitatively different regimes, matching the trichotomy of Subsection~\ref{subsec:main_results}.

If $\dimresis(\Xi)=0$ (the Hambly--Kumagai regime), the effective resistance across a macroscopic cell stays of order one when every microscopic edge has unit resistance.
This is the regime represented by the diamond hierarchical lattice, and Hambly and Kumagai construct the limiting diffusion directly from the renormalised Dirichlet energy \cite{hambly2010diffusion}.

If $\dimresis(\Xi)<0$ (the superconducting regime), the effective resistance across a macroscopic cell tends to zero in the unit-resistance network.
After renormalisation, one keeps the macroscopic resistance of order one, but the resistance of a single microscopic edge diverges.
This is incompatible with the resistance-form construction used above.
The name superconducting refers to the unit-resistance network: every edge has resistance one, and the vanishing of the two-terminal resistance is emergent in the limit; it is unrelated to conductor--superconductor mixtures in the percolation literature, where some bonds carry literally zero resistance.

If $\dimresis(\Xi)>0$ (the Brownian regime), the effective resistance across a macroscopic cell diverges in the unit-resistance network.
After renormalisation, each microscopic edge has vanishing resistance, while each macroscopic cell has resistance of order one.
This is the standard resistance-form regime.
Theorem~\ref{thm:sec32_resistance_metric} shows that the resistance and natural metrics induce the same topology on $\Xi_{\mathrm M}$, so it is natural that the diffusion and Brownian spectral dimensions agree in this regime.
\end{remark}

\begin{figure}[bt]
    \centering
    \begin{minipage}[t]{0.32\linewidth}
        \includegraphics[width=\linewidth]{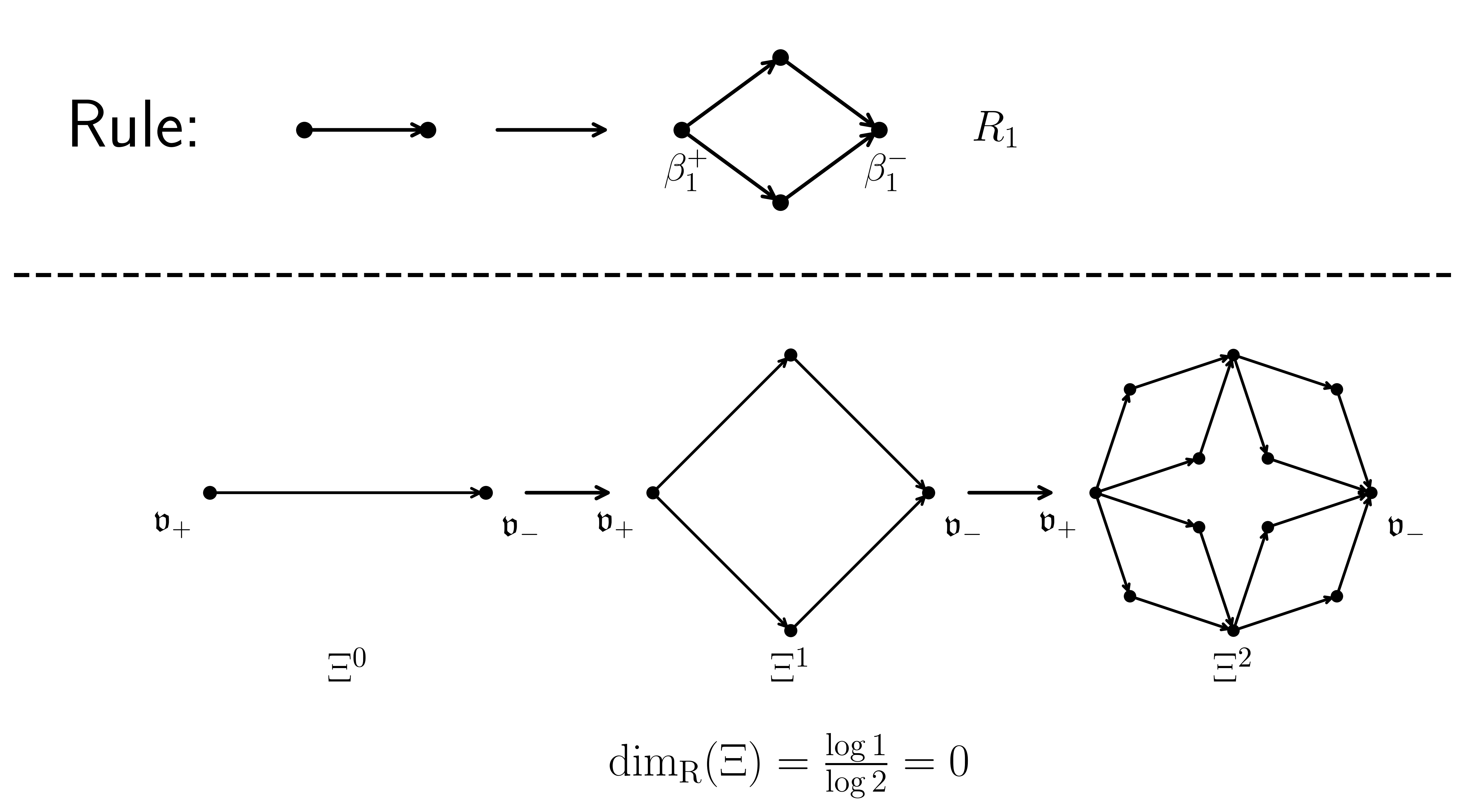}
        \caption{$(2,2)$-flower}
        \label{fig:2-2flower}
    \end{minipage}
\hfill
    \begin{minipage}[t]{0.32\linewidth}
        \includegraphics[width=\linewidth]{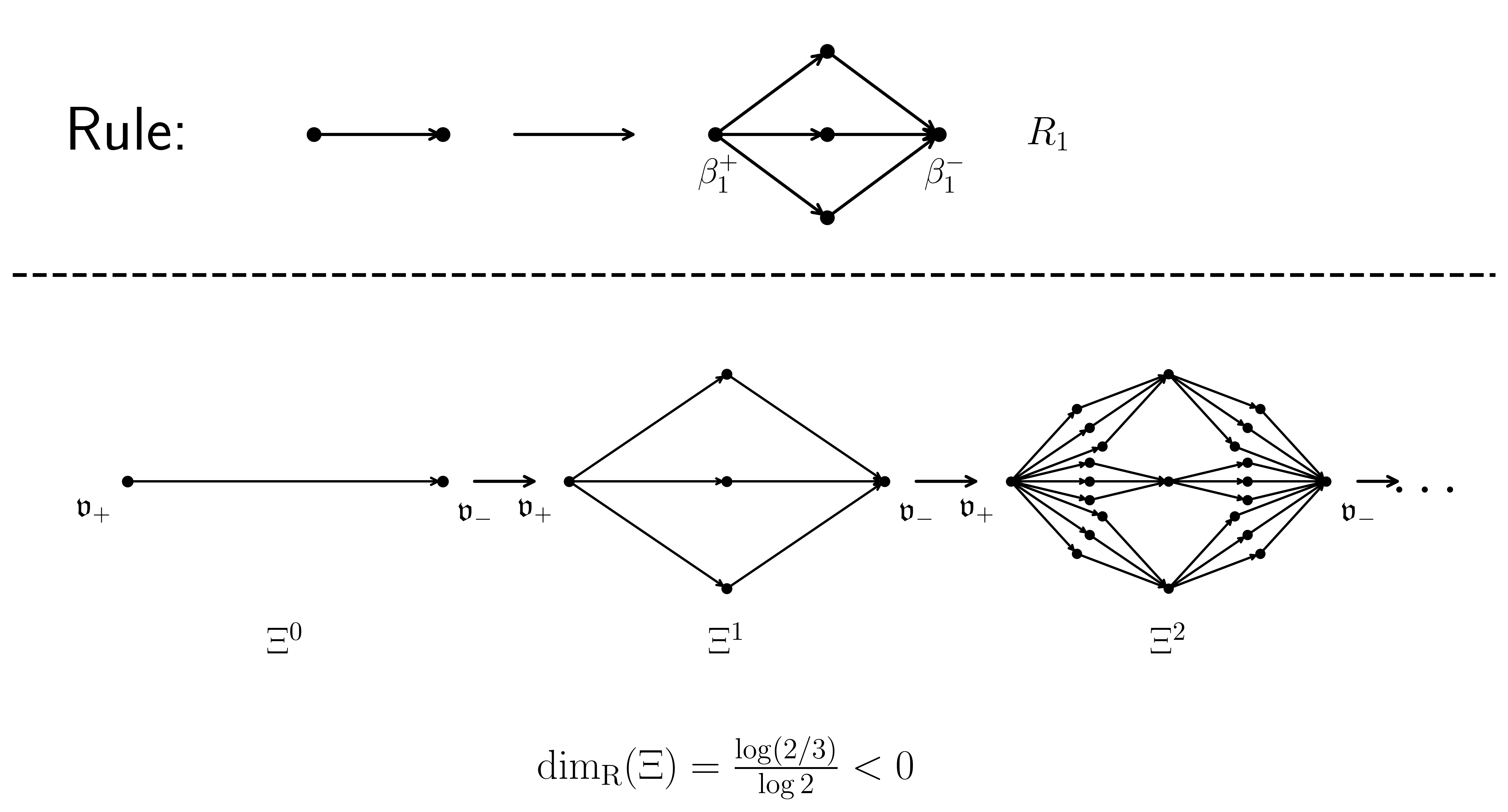}
        \caption{$(3,2)$-flower}
        \label{fig:3-2flower}
    \end{minipage}
\hfill
    \begin{minipage}[t]{0.32\linewidth}
        \includegraphics[width=\linewidth]{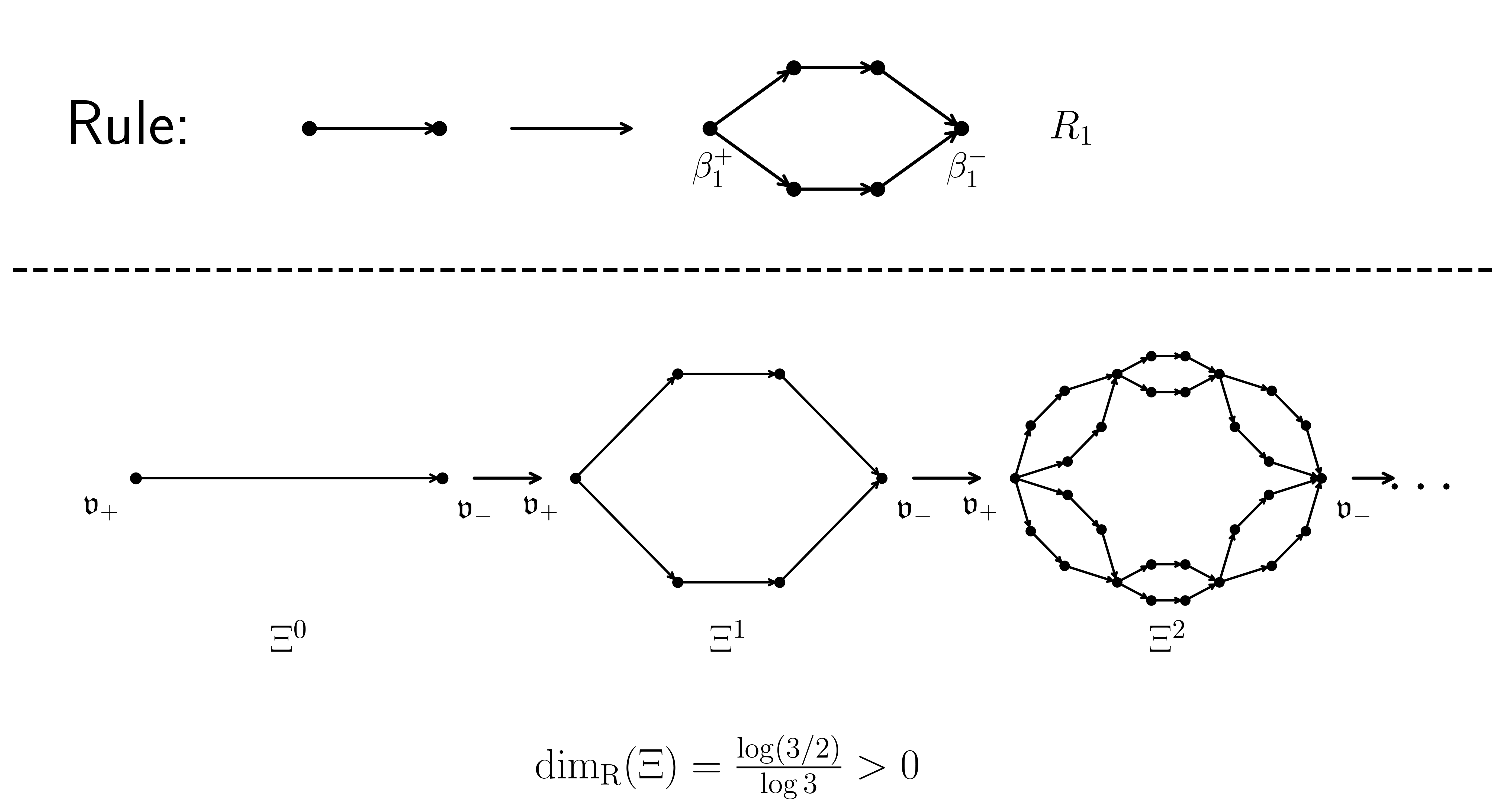}
        \caption{$(2,3)$-flower}
        \label{fig:2-3flower}
    \end{minipage}
\end{figure}

\begin{example}\label{ex:two-colour-spectral}
Consider again the two-coloured EIGS of Figure~\ref{fig:two_colour}.
From Example~\ref{example:two_colour_1} we have known $\dimbox$, $\dimdeg$, $\dimresis$ and $\dimwalk$.
Applying Theorem~\ref{thm:diffusion-spectral}, we obtain that for every \(x\in \mathscr V\),
\begin{align*}
\dimspec^{(\mathrm R)}(\Xi_{\mathrm M}:x)
&=
\dimspec^{(\mathrm N)}(\Xi_{\mathrm M}:x)
= \\
&\frac{2\dimbox(\Xi)\left(1-\frac{1}{\dimdeg(\Xi)}\right)}{\dimwalk(\Xi)}
=
\frac{2\log((3+\sqrt6)/2)}{\log((3+\sqrt6)\rho(\Psi))}
\approx 1.1160.
\end{align*}
Also, for \(\mu_{\Xi_{\mathrm M}}\)-almost every \(x\in \mathscr V^\infty\),
\[
\dimspec^{(\mathrm R)}(\Xi_{\mathrm M}:x)
=
\dimspec^{(\mathrm N)}(\Xi_{\mathrm M}:x)
=
\frac{2\dimbox(\Xi)}{\dimwalk(\Xi)}
=
\frac{2\log(3+\sqrt6)}{\log((3+\sqrt6)\rho(\Psi))}
\approx 1.8877.
\]

Therefore, the two-coloured EIGS has diffusion spectral dimension approximately \(1.1160\) at finite-born points and approximately \(1.8877\) at \(\mu_{\Xi_{\mathrm M}}\)-almost every point of \(\mathscr V^\infty\).
\end{example}

\section{Random EIGS and percolation cluster}\label{sec:percolation}

\begin{figure}[tb]
    \centering
    \includegraphics[width=0.9\linewidth]{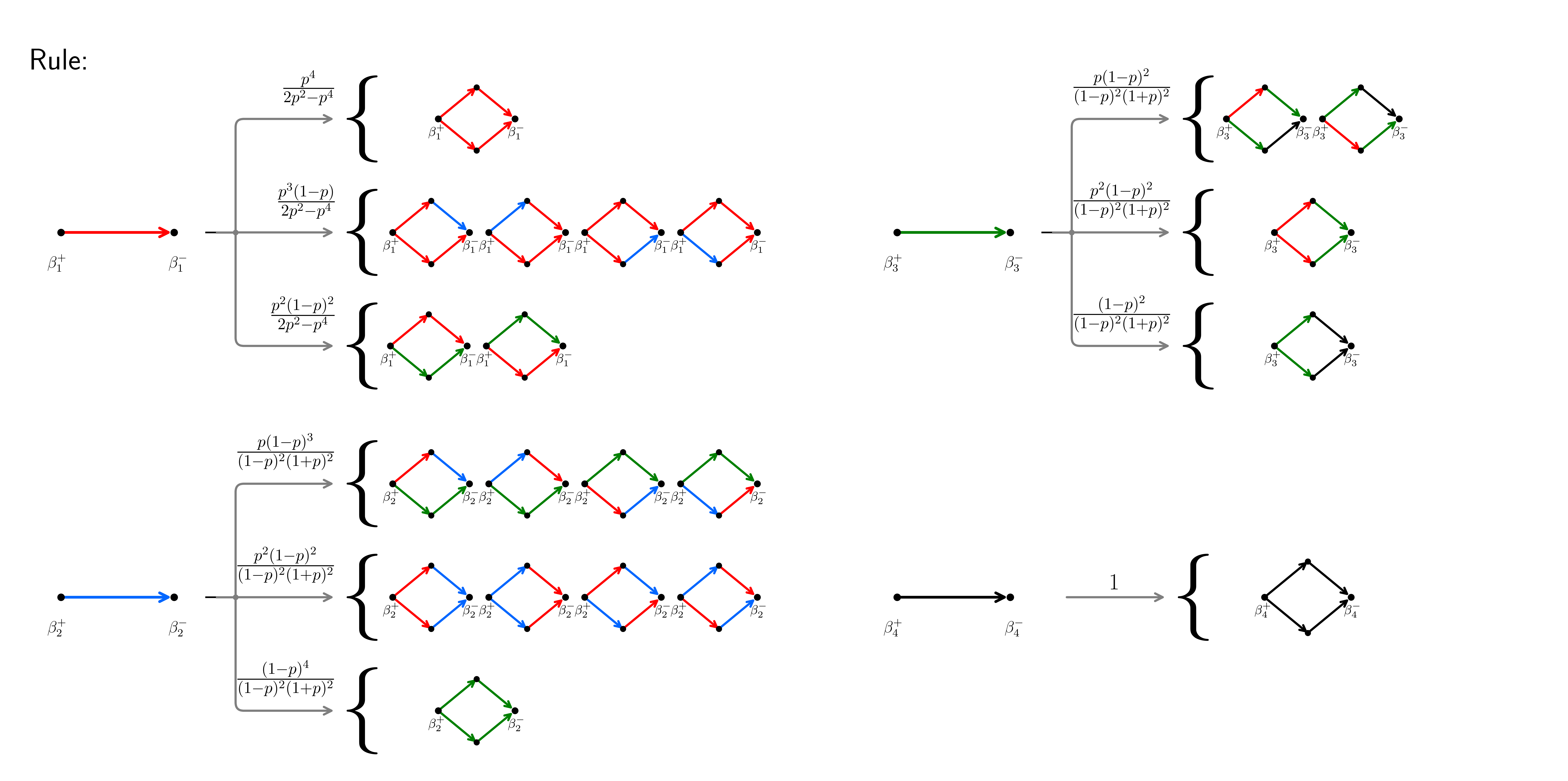}
    \caption{The random EIGS associated with percolation}
    \label{fig:random EIGS DHL}
\end{figure}

In \cite{hambly2010diffusion}, Hambly and Kumagai studied critical bond percolation on the diamond hierarchical lattice under the two-terminal conditioning that the two boundary vertices are connected.
They showed that the scaling limit of the cluster is a graph-directed random recursive fractal.
In the language of the present paper, the discrete model is naturally encoded by a four-colour random EIGS.

For the DHL application, the relevant parameter is
\[
p_c=\frac{\sqrt5-1}{2},
\qquad
1-p_c=p_c^2.
\]
We start from a single red edge $\Xi_1^0$ joining the boundary vertices $\mathfrak v_+$ and $\mathfrak v_-$.
At each substitution step, every coloured edge is replaced independently according to the rule in Figure~\ref{fig:random EIGS DHL}.
The colour records the two-terminal state of the corresponding cell.
Red means that the two terminals are connected inside the cell.
Blue means that they are disconnected, but both terminals remain active.
Green means that they are disconnected, but exactly one terminal remains active.
Black is a cemetery state: the terminals are disconnected and neither remains active.
Here a terminal is called active if it remains attached to the cluster outside the current cell.

For each $n\ge 0$, let $\Xi_1^n$ be the coloured graph after $n$ substitution steps, and let $\langle \Xi_1^n\rangle_1$ be its colour-$1$ subgraph.
Then $\langle \Xi_1^n\rangle_1$ has the same law as the level-$n$ critical two-terminal cluster on the DHL.
Thus the red subgraph is the cluster itself, while the blue, green and black edges are auxiliary states introduced only to keep the recursion closed.

\begin{figure}[b]
    \centering
    \includegraphics[width=\linewidth]{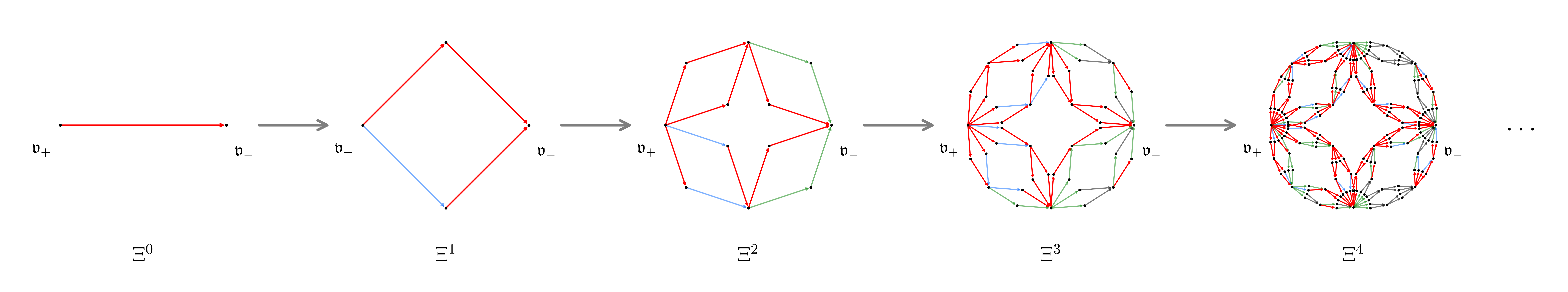}
    \caption{A critical percolation cluster (red subgraph) on DHL generated by random EIGS.}
    \label{fig:DHL percolation}
\end{figure}

The diffusion constructed in \cite{hambly2010diffusion} lives in a resistance-form setting in which the edge weights are renormalised so that the total resistance across the cluster stays equal to $1$.
This is different from the unit-resistance setting considered in the present paper, where every edge has resistance $1$.
For this reason, the almost-sure exponential growth rate of the two-terminal effective resistance was left open.
More precisely, the problem is:
\begin{conjecture}
There exists a deterministic constant $\alpha$ such that, almost surely,
\[
\lim_{n\to\infty}\frac{1}{n}\log \Reff_{\langle \Xi_1^n\rangle_1}(\mathfrak v_+,\mathfrak v_-)=\alpha.
\]
\end{conjecture}

The key point is that, for this observable, the full four-colour bookkeeping is more than we need.
After deleting red components attached to only one terminal, the two-terminal resistance recursion closes on a monochromatic random EIGS built from independent series and diamond replacements.
Subsection~\ref{subsec:Quenched unit-resistance exponent} makes this reduction precise and proves the existence of the quenched unit-resistance exponent.
Subsection~\ref{subsec:dhl_cluster} then returns to the full four-colour system and discusses the geometric and Brownian exponents encoded by the critical cluster.

\subsection{Solution to the open problem in \texorpdfstring{\cite{hambly2010diffusion}}{}}\label{subsec:Quenched unit-resistance exponent}

\begin{figure}[bt]
    \centering
    \includegraphics[width=0.9\linewidth]{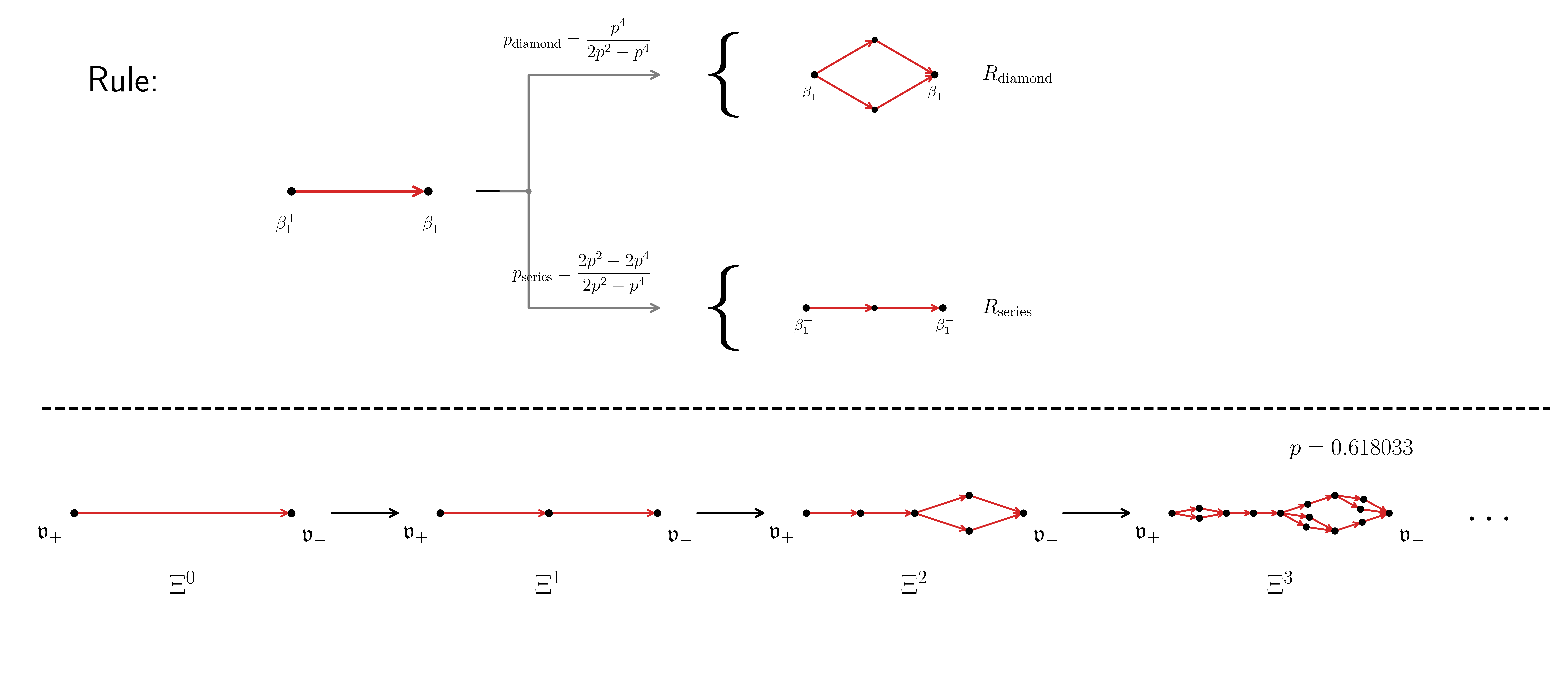}
    \caption{Reduction of the two-terminal resistance problem to a random EIGS.}
    \label{fig:dhl_reduction}
\end{figure}

We now isolate the observable that matters for the open problem, namely the effective resistance between the two boundary vertices.
Although the critical cluster is encoded by the four-colour system $\Xi_1^n$, the full colour information is not needed for this quantity.

Indeed, fix a red cell and look one scale below.
Only the red pieces that lie on a connection from the left terminal to the right terminal can influence the effective resistance across the cell.
Any red component attached to only one terminal is a dangling piece, and deleting it does not change the two-terminal effective resistance.
After this reduction, there are only two possible reduced geometries, shown in Figure~\ref{fig:dhl_reduction}.
If both branches connect the two terminals, the reduced cell is a diamond.
If exactly one branch connects the two terminals, the reduced cell is a two-edge series.

For the critical DHL cluster this reduction is applied at $p=p_c$.
However, the resulting two-terminal recursion is naturally defined for every $p\in(0,1)$, and the argument below works in this slightly greater generality.
Accordingly, for $p\in(0,1)$, define
\[
p_{\mathrm{diamond}}:=\frac{p^4}{2p^2-p^4},
\qquad
p_{\mathrm{series}}:=\frac{2p^2-2p^4}{2p^2-p^4}.
\]
These are the conditional probabilities that, given a red parent cell, both branches survive or exactly one branch survives.

Let $G^0$ be the single edge between two vertices $a$ and $b$.
For each $n\ge 0$, construct $G^{n+1}$ from $G^n$ by replacing every edge independently by a diamond with probability $p_{\mathrm{diamond}}$ and by a two-edge series with probability $p_{\mathrm{series}}$, as in Figure~\ref{fig:dhl_reduction}.
All new edges have unit resistance.
Write
\[
R_n:=\Reff_{G^n}(a,b)
\]
for the effective resistance between the boundary vertices of $G^n$.

At the critical parameter, this monochromatic process captures exactly the law of the two-terminal effective resistance of the red cluster.

\begin{proposition}\label{prop:dhl_reduction}
At $p=p_c$,
\[
\bigl(\Reff_{\langle \Xi_1^n\rangle_1}(\mathfrak v_+,\mathfrak v_-)\bigr)_{n\ge 0}
\stackrel{d}{=}
(R_n)_{n\ge 0}.
\]
\end{proposition}

\begin{proof}
 
After deleting the red pieces attached to only one terminal, the reduced red network satisfies exactly the same recursive replacement rule as $G^n$.

The claim is clear at level $0$.
Now fix a red parent cell and look one scale below.
If exactly one of its two first-generation branches connects the boundary vertices, then every red component in the other branch is attached to only one terminal.
Deleting these dangling pieces does not change the two-terminal effective resistance, and the reduced cell becomes a two-edge series.
If both branches connect the boundary vertices, then the reduced cell becomes a diamond.

These two cases occur with probabilities $p_{\mathrm{series}}$ and $p_{\mathrm{diamond}}$, respectively.
Moreover, conditioned on either case, the red descendant cells evolve independently and each carries the same reduced law as the parent cell.
Thus the sequence of reduced two-terminal networks issued from the root cell has the same initial condition and the same independent replacement rule as $(G^n)_{n\ge 0}$.
Taking effective resistances gives the claimed equality in law.
\end{proof}

We therefore first prove the existence of the quenched exponent for the reduced process $(R_n)_{n\ge 0}$.
The critical DHL statement then follows immediately from Proposition~\ref{prop:dhl_reduction}.

By the standard two-terminal reduction, replacing each edge of a fixed finite graph by a two-terminal network changes the effective resistance between two marked vertices by replacing that edge with its two-terminal effective resistance.
In our one-step replacement, every edge is replaced by a two-terminal network whose effective resistance is either $1$ in the diamond case or $2$ in the series case.
Therefore, for every realisation and every $n\ge 0$,
\[
R_{n+1}\ge R_n,
\qquad
R_{n+1}\le 2R_n,
\]
and hence
\[
1\le R_n\le 2^n.
\]

\begin{lemma}\label{lem:annealed_exponent}
For every $p\in(0,1)$, the limit
\[
\alpha(p)=\lim_{n\to\infty}\frac{1}{n}\log \mathbb E[R_n]
\]
exists.
\end{lemma}

\begin{proof}
 
Condition on the graph at time $n$, replace each edge by an independent copy of $G^m$, and use concavity of the effective resistance as a function of the edge resistances.

Fix $n,m\ge 0$ and condition on the realisation of $G^n$.
Build $G^{n+m}$ by replacing each edge of $G^n$ by an independent copy of $G^m$.
For each edge $e$ of $G^n$, let $r_e$ be the effective resistance between the endpoints of the copy of $G^m$ attached to $e$.
Conditional on $G^n$, the random variables $(r_e)$ are independent and identically distributed with law $R_m$, and satisfy $\mathbb E[r_e\mid G^n]=\mathbb E[R_m]$.

For a fixed finite graph $H$ with marked vertices $a,b$, the Thomson principle gives
\[
\Reff_H(a,b;r)=\inf_{\theta\in\mathcal F_{a\to b}}\sum_{e\in E(H)} r_e\theta_e^2,
\]
so the map $r\mapsto \Reff_H(a,b;r)$ is concave and $1$-homogeneous in $r$.
By the two-terminal reduction,
\[
R_{n+m}=\Reff_{G^n}(a,b;r).
\]
Therefore, by Jensen's inequality for concave functions and by $1$-homogeneity,
\[
\mathbb E[R_{n+m}\mid G^n]
=
\mathbb E\bigl[\Reff_{G^n}(a,b;r)\mid G^n\bigr]
\le
\Reff_{G^n}\bigl(a,b;\mathbb E[r\mid G^n]\bigr)
=
\mathbb E[R_m]R_n.
\]
Taking expectations gives
\[
\mathbb E[R_{n+m}]\le \mathbb E[R_n]\mathbb E[R_m].
\]
Apply Fekete's lemma to $\log \mathbb E[R_n]$.
\end{proof}

\begin{lemma}\label{lem:logistic_product}
For every $n\ge 1$,
\begin{align*}
&\mathbb E[R_{n+1}]\,\mathbb E\Big[\frac{1}{R_{n+1}}\Big]
\le \\
&\frac{(2-p_{\mathrm{diamond}})(1+p_{\mathrm{diamond}})}{2}\,
\mathbb E[R_n]\mathbb E\Big[\frac{1}{R_n}\Big]
-
\frac{(2-p_{\mathrm{diamond}})(1+p_{\mathrm{diamond}})}{2}\,
\frac{\big(\mathbb E[R_n]\mathbb E[\frac{1}{R_n}]-1\big)^2}
{\frac{\mathbb E[R_n^2]\mathbb E[\frac{1}{R_n}]}{\mathbb E[R_n]}-1}.
\end{align*}
\end{lemma}

\begin{proof}
 
Write exact one-step identities for $\mathbb E[R_{n+1}]$ and $\mathbb E[1/R_{n+1}]$, and then use Cauchy--Schwarz to estimate the correction term from below.

Condition on the first step of the construction.
If the first step is series, then $G^{n+1}$ is a series of two independent copies of $G^n$, so $R_{n+1}=X+Y$, where $X,Y$ are independent copies of $R_n$.
If the first step is diamond, then $G^{n+1}$ is a diamond whose four edges are independent copies of $G^n$, so
\[
R_{n+1}=\frac{AB}{A+B},
\qquad
A=X_1+X_2,
\qquad
B=X_3+X_4,
\]
where $X_1,X_2,X_3,X_4$ are independent copies of $R_n$.

Define
$
C_n:=\mathbb E\Big[\frac{(A-B)^2}{A+B}\Big].
$
Using
$
\frac{AB}{A+B}=\frac{A+B}{4}-\frac{(A-B)^2}{4(A+B)},
$
we obtain
\[
\mathbb E[R_{n+1}]
=
2p_{\mathrm{series}}\mathbb E[R_n]
+
p_{\mathrm{diamond}}\mathbb E\Big[\frac{AB}{A+B}\Big]
=
(2-p_{\mathrm{diamond}})\mathbb E[R_n]-\frac{p_{\mathrm{diamond}}}{4}C_n.
\]

Define
\[
\Delta_n:=\mathbb E\Big[\frac{(X-Y)^2}{XY(X+Y)}\Big].
\]
For a diamond cell,
$
\frac{1}{AB/(A+B)}=\frac{1}{A}+\frac{1}{B},
$
so
\[
\mathbb E\Big[\frac{1}{R_{n+1}}\Big]
=
p_{\mathrm{series}}\mathbb E\Big[\frac{1}{X+Y}\Big]
+
p_{\mathrm{diamond}}\mathbb E\Big[\frac{1}{A}+\frac{1}{B}\Big]
=
(1+p_{\mathrm{diamond}})\mathbb E\Big[\frac{1}{X+Y}\Big].
\]
Using
$
\frac{1}{X+Y}
=
\frac{1}{4}\Big(\frac{1}{X}+\frac{1}{Y}\Big)
-
\frac{1}{4}\frac{(X-Y)^2}{XY(X+Y)},
$
we get
$
\mathbb E\Big[\frac{1}{X+Y}\Big]
=
\frac{1}{2}\mathbb E\Big[\frac{1}{R_n}\Big]-\frac{1}{4}\Delta_n,
$
and hence
\[
\mathbb E\Big[\frac{1}{R_{n+1}}\Big]
=
\frac{1+p_{\mathrm{diamond}}}{2}\mathbb E\Big[\frac{1}{R_n}\Big]
-\frac{1+p_{\mathrm{diamond}}}{4}\Delta_n.
\]
Moreover,
\[
\Delta_n
=
\mathbb E\Big[\frac{1}{X}+\frac{1}{Y}-\frac{4}{X+Y}\Big]
\le
2\mathbb E\Big[\frac{1}{R_n}\Big].
\]

Multiply the identities for $\mathbb E[R_{n+1}]$ and $\mathbb E[1/R_{n+1}]$ and expand.
Using $\Delta_n\le 2\mathbb E[1/R_n]$ and $C_n\ge 0$ shows that the mixed correction term is non-positive, and yields
\begin{align*}
&\mathbb E[R_{n+1}]\,\mathbb E\Big[\frac{1}{R_{n+1}}\Big]
\le \\
&\frac{(2-p_{\mathrm{diamond}})(1+p_{\mathrm{diamond}})}{2}\,
\mathbb E[R_n]\mathbb E\Big[\frac{1}{R_n}\Big]
-
\frac{(2-p_{\mathrm{diamond}})(1+p_{\mathrm{diamond}})}{4}\,
\mathbb E[R_n]\Delta_n.
\end{align*}

Set
$
W:=\frac{(X-Y)^2}{XY}.
$
Then $\Delta_n=\mathbb E\big[\frac{W}{X+Y}\big]$.
We compute
\[
\mathbb E[W]
=
\mathbb E\Big[\frac{X}{Y}+\frac{Y}{X}-2\Big]
=
2\mathbb E[R_n]\mathbb E\Big[\frac{1}{R_n}\Big]-2
=
2\Big(\mathbb E[R_n]\mathbb E\Big[\frac{1}{R_n}\Big]-1\Big),
\]
and
$
W(X+Y)=\frac{X^2}{Y}+\frac{Y^2}{X}-(X+Y),
$
so by independence
\[
\mathbb E[W(X+Y)]
=
2\mathbb E[R_n^2]\mathbb E\Big[\frac{1}{R_n}\Big]-2\mathbb E[R_n]
=
2\mathbb E[R_n]\Big(\frac{\mathbb E[R_n^2]\mathbb E[\frac{1}{R_n}]}{\mathbb E[R_n]}-1\Big).
\]
By Cauchy--Schwarz,
\[
\mathbb E[W]^2
\le
\mathbb E\Big[\frac{W}{X+Y}\Big]\mathbb E[W(X+Y)]
=
\Delta_n\,\mathbb E[W(X+Y)].
\]
Therefore, for $n\ge 1$,
\[
\mathbb E[R_n]\Delta_n
\ge
\frac{2\big(\mathbb E[R_n]\mathbb E[\frac{1}{R_n}]-1\big)^2}
{\frac{\mathbb E[R_n^2]\mathbb E[\frac{1}{R_n}]}{\mathbb E[R_n]}-1}.
\]
Substituting this bound into the previous inequality gives the claim.
\end{proof}

\begin{lemma}\label{lem:uniform_product}
For every $p\in(0,1)$, there exists a constant $M(p)<\infty$ such that for all $n\ge 1$,
\[
\mathbb E[R_n]\mathbb E\Big[\frac{1}{R_n}\Big]\le M(p).
\]
\end{lemma}

\begin{proof}
 
We first control the relative second moment $\mathbb E[R_n^2]/\mathbb E[R_n]^2$, and then feed this bound into Lemma~\ref{lem:logistic_product}.

We start with a coarse second-moment estimate.
On the series event, $R_{n+1}=X+Y$, and hence
\[
\mathbb E[(X+Y)^2]=2\mathbb E[R_n^2]+2\mathbb E[R_n]^2.
\]
On the diamond event, $R_{n+1}=AB/(A+B)$ with $A=X_1+X_2$ and $B=X_3+X_4$, and
$
\frac{AB}{A+B}\le \frac{\sqrt{AB}}{2}.
$
Taking expectations and using independence gives
\[
\mathbb E\Big[\Big(\frac{AB}{A+B}\Big)^2\Big]\le \mathbb E[R_n]^2.
\]
Therefore
\[
\mathbb E[R_{n+1}^2]
\le
2p_{\mathrm{series}}\mathbb E[R_n^2]+(2-p_{\mathrm{diamond}})\mathbb E[R_n]^2.
\]

\textbf{Case 1: $p_{\mathrm{diamond}}\le \frac{1}{2}$.}
We first show that
$
\frac{\mathbb E[R_n^2]}{\mathbb E[R_n]^2}\le 2
$
for all $n\ge 0$.
Indeed, let $A=X_1+X_2$, $B=X_3+X_4$, and $D=AB/(A+B)$, where $X_1,\dots,X_4$ are independent copies of $R_n$.
By Cauchy--Schwarz, $\mathbb E[AB]^2\le \mathbb E[D]\mathbb E[AB(A+B)]$.
Since $\mathbb E[AB]=4\mathbb E[R_n]^2$ and $\mathbb E[AB(A+B)]=8\mathbb E[R_n](\mathbb E[R_n^2]+\mathbb E[R_n]^2)$, we get
\[
\mathbb E[D]\ge \frac{2\mathbb E[R_n]}{\frac{\mathbb E[R_n^2]}{\mathbb E[R_n]^2}+1}.
\]
Hence
\[
\frac{\mathbb E[R_{n+1}]}{\mathbb E[R_n]}
\ge
2p_{\mathrm{series}}+\frac{2p_{\mathrm{diamond}}}{\frac{\mathbb E[R_n^2]}{\mathbb E[R_n]^2}+1}.
\]
Combining this with the coarse second-moment inequality yields
\[
\frac{\mathbb E[R_{n+1}^2]}{\mathbb E[R_{n+1}]^2}
\le
\frac{
2p_{\mathrm{series}}\frac{\mathbb E[R_n^2]}{\mathbb E[R_n]^2}+(2-p_{\mathrm{diamond}})
}{
\left(
2p_{\mathrm{series}}+\frac{2p_{\mathrm{diamond}}}{\frac{\mathbb E[R_n^2]}{\mathbb E[R_n]^2}+1}
\right)^2
}.
\]
A direct check shows that the right-hand side is at most $2$ whenever $p_{\mathrm{diamond}}\in[0,\frac{1}{2}]$ and $\frac{\mathbb E[R_n^2]}{\mathbb E[R_n]^2}\le 2$.
Since $\frac{\mathbb E[R_0^2]}{\mathbb E[R_0]^2}=1$, induction gives $\frac{\mathbb E[R_n^2]}{\mathbb E[R_n]^2}\le 2$ for all $n\ge 0$.

Lemma~\ref{lem:logistic_product} now implies
\[
\mathbb E[R_{n+1}]\mathbb E\Big[\frac{1}{R_{n+1}}\Big]
\le
\frac{(2-p_{\mathrm{diamond}})(1+p_{\mathrm{diamond}})}{2}
\left(
\mathbb E[R_n]\mathbb E\Big[\frac{1}{R_n}\Big]
-
\frac{\big(\mathbb E[R_n]\mathbb E[\frac{1}{R_n}]-1\big)^2}{2\mathbb E[R_n]\mathbb E[\frac{1}{R_n}]-1}
\right).
\]
The map $x\mapsto x-\frac{(x-1)^2}{2x-1}$ is increasing on $[1,\infty)$, and $p_{\mathrm{diamond}}\le \frac{1}{2}$ implies $\frac{(2-p_{\mathrm{diamond}})(1+p_{\mathrm{diamond}})}{2}\le \frac{9}{8}$.
Therefore
\[
\mathbb E[R_{n+1}]\mathbb E\Big[\frac{1}{R_{n+1}}\Big]
\le
\frac{9}{8}\left(\frac{11}{5}-\frac{(6/5)^2}{2(11/5)-1}\right)
=
\frac{9}{8}\cdot\frac{151}{85}
<
\frac{11}{5}
\]
whenever $\mathbb E[R_n]\mathbb E[1/R_n]\le \frac{11}{5}$.
Since
\[
\mathbb E[R_1]\mathbb E\Big[\frac{1}{R_1}\Big]
=
\frac{(2-p_{\mathrm{diamond}})(1+p_{\mathrm{diamond}})}{2}
\le
\frac{9}{8}
<
\frac{11}{5},
\]
we conclude that
$
\mathbb E[R_n]\mathbb E\Big[\frac{1}{R_n}\Big]\le \frac{11}{5}
$
for all $n\ge 1$ in this case.

\textbf{Case 2: $\frac{1}{2}\le p_{\mathrm{diamond}}\le \frac{1}{\sqrt 3}$.}
We prove by induction on $n\ge 1$ that
\[
\frac{\mathbb E[R_n^2]}{\mathbb E[R_n]^2}\le 3
\quad\text{and}\quad
\mathbb E[R_n]\mathbb E\Big[\frac{1}{R_n}\Big]\le \frac{11}{5}.
\]

Assume that these two bounds hold at level $n$.
Lemma~\ref{lem:logistic_product} then gives
\[
\mathbb E[R_{n+1}]\mathbb E\Big[\frac{1}{R_{n+1}}\Big]
\le
\frac{(2-p_{\mathrm{diamond}})(1+p_{\mathrm{diamond}})}{2}
\left(
\frac{11}{5}-\frac{(6/5)^2}{3(11/5)-1}
\right)
<
\frac{11}{5},
\]
because $x\mapsto x-\frac{(x-1)^2}{3x-1}$ is increasing on $[1,\infty)$ and $\frac{(2-p_{\mathrm{diamond}})(1+p_{\mathrm{diamond}})}{2}\le \frac{9}{8}$.

Let $X,Y$ be independent copies of $R_n$.
From the identity in the proof of Lemma~\ref{lem:logistic_product}, we have $\mathbb E[1/(X+Y)]\le \frac{1}{2}\mathbb E[1/R_n]$, and hence
\[
\mathbb E[X+Y]\mathbb E\Big[\frac{1}{X+Y}\Big]\le \frac{11}{5}.
\]

Now let $A=X_1+X_2$, $B=X_3+X_4$, and $D=AB/(A+B)$, where $X_1,\dots,X_4$ are independent copies of $R_n$.
Since $D=(1/A+1/B)^{-1}$ and $x\mapsto 1/x$ is convex, Jensen gives $\mathbb E[D]\ge 1/(2\mathbb E[1/A])$.
Using the previous bound,
\[
\mathbb E[D]
\ge
\frac{\mathbb E[R_n]}{\mathbb E[X+Y]\mathbb E\big[\frac{1}{X+Y}\big]}
\ge
\frac{5}{11}\mathbb E[R_n].
\]
Therefore
\[
\frac{\mathbb E[R_{n+1}]}{\mathbb E[R_n]}
\ge
2(1-p_{\mathrm{diamond}})+\frac{5}{11}p_{\mathrm{diamond}}
=
2-\frac{17}{11}p_{\mathrm{diamond}}.
\]

Define $C_n=\mathbb E\big[\frac{(A-B)^2}{A+B}\big]$ and $K_n=\mathbb E\big[\frac{AB(A-B)^2}{(A+B)^2}\big]$.
As in the previous argument, Cauchy--Schwarz gives $K_n\ge \frac{5}{12}C_n^2$.
Combining this with the identities for $\mathbb E[R_{n+1}]$ and $\mathbb E[R_{n+1}^2]$, and using $\frac{\mathbb E[R_n^2]}{\mathbb E[R_n]^2}\le 3$, we obtain
\[
\frac{\mathbb E[R_{n+1}^2]}{\mathbb E[R_{n+1}]^2}
\le
\frac{
8-7p_{\mathrm{diamond}}
-
\frac{5}{3p_{\mathrm{diamond}}}
\left(
2-p_{\mathrm{diamond}}-\frac{\mathbb E[R_{n+1}]}{\mathbb E[R_n]}
\right)^2
}{
\left(\frac{\mathbb E[R_{n+1}]}{\mathbb E[R_n]}\right)^2
}.
\]
The right-hand side is decreasing in $\frac{\mathbb E[R_{n+1}]}{\mathbb E[R_n]}$ on the relevant range, so the previous lower bound yields
\[
\frac{\mathbb E[R_{n+1}^2]}{\mathbb E[R_{n+1}]^2}
\le
\frac{968-907p_{\mathrm{diamond}}}{(22-17p_{\mathrm{diamond}})^2}
\le 3
\]
for $p_{\mathrm{diamond}}\in[1/2,1/\sqrt 3]$.
This closes the induction step.

For $n=1$, since $R_1\in\{1,2\}$,
\[
\frac{\mathbb E[R_1^2]}{\mathbb E[R_1]^2}
=
\frac{4-3p_{\mathrm{diamond}}}{(2-p_{\mathrm{diamond}})^2}
<3
\]
and
\[
\mathbb E[R_1]\mathbb E\Big[\frac{1}{R_1}\Big]
=
\frac{(2-p_{\mathrm{diamond}})(1+p_{\mathrm{diamond}})}{2}
\le \frac{9}{8}
<
\frac{11}{5}.
\]
Therefore the claim holds for all $n\ge 1$.

\textbf{Case 3: $p_{\mathrm{diamond}}>\frac{1}{\sqrt 3}$.}
Since $R_{n+1}\ge R_n$ almost surely, we have $\mathbb E[R_{n+1}]^2\ge \mathbb E[R_n]^2$.
Combining this with the coarse second-moment bound gives
\[
\frac{\mathbb E[R_{n+1}^2]}{\mathbb E[R_{n+1}]^2}
\le
2p_{\mathrm{series}}\frac{\mathbb E[R_n^2]}{\mathbb E[R_n]^2}
+
(2-p_{\mathrm{diamond}}).
\]
As $p_{\mathrm{diamond}}>\frac{1}{2}$, this recursion and the bound
\[
\frac{\mathbb E[R_1^2]}{\mathbb E[R_1]^2}
=
\frac{4-3p_{\mathrm{diamond}}}{(2-p_{\mathrm{diamond}})^2}
\le
\frac{2-p_{\mathrm{diamond}}}{2p_{\mathrm{diamond}}-1}
\]
imply that
\[
\frac{\mathbb E[R_n^2]}{\mathbb E[R_n]^2}
\le
\frac{2-p_{\mathrm{diamond}}}{2p_{\mathrm{diamond}}-1}
\]
for all $n\ge 1$.

Applying Lemma~\ref{lem:logistic_product} and using the elementary inequality
$
\frac{(x-1)^2}{\frac{2-p_{\mathrm{diamond}}}{2p_{\mathrm{diamond}}-1}x-1}
\ge
\frac{x-2}{\frac{2-p_{\mathrm{diamond}}}{2p_{\mathrm{diamond}}-1}}
$
for $x\ge 1$, we get
\[
\mathbb E[R_{n+1}]\mathbb E\Big[\frac{1}{R_{n+1}}\Big]
\le
\frac{3}{2}(1-p_{\mathrm{diamond}}^2)\,
\mathbb E[R_n]\mathbb E\Big[\frac{1}{R_n}\Big]
+
(1+p_{\mathrm{diamond}})(2p_{\mathrm{diamond}}-1).
\]
Since $p_{\mathrm{diamond}}>\frac{1}{\sqrt 3}$ implies $\frac{3}{2}(1-p_{\mathrm{diamond}}^2)<1$, iteration yields a finite uniform bound on $\mathbb E[R_n]\mathbb E[1/R_n]$.

Combining the three cases completes the proof.
\end{proof}

\begin{theorem}\label{thm:reduced_quenched_exponent}
For every $p\in(0,1)$, there exists a constant $\alpha(p)\in[0,\log 2]$ such that
\[
\frac{1}{n}\log R_n\to \alpha(p)
\]
almost surely as $n\to\infty$.
Moreover,
\[
\lim_{n\to\infty}\frac{1}{n}\mathbb E[\log R_n]
=
\lim_{n\to\infty}\frac{1}{n}\log \mathbb E[R_n]
=
\alpha(p).
\]
\end{theorem}

\begin{proof}
 
Define $\alpha(p)$ by Lemma~\ref{lem:annealed_exponent}, prove exponential concentration of $\log R_n$ around its mean using Lemma~\ref{lem:uniform_product}, and then compare $\mathbb E[\log R_n]$ with $\log \mathbb E[R_n]$.
By Lemma~\ref{lem:annealed_exponent}, the limit
$
\alpha(p)=\lim_{n\to\infty}\frac{1}{n}\log \mathbb E[R_n]
$
exists.
Since
$
1\le R_n\le 2^n,
$
we have
$
1\le \mathbb E[R_n]\le 2^n,
$
and hence
$
\alpha(p)\in[0,\log 2].
$

By Lemma~\ref{lem:uniform_product}, there exists $M(p)<\infty$ such that for all $n\ge 1$,
$
\mathbb E[R_n]\mathbb E\Big[\frac{1}{R_n}\Big]\le M(p).
$

By Jensen's inequality,
$
\exp\big(-\mathbb E[\log R_n]\big)\le \mathbb E\Big[\frac{1}{R_n}\Big].
$
Therefore
\[
\mathbb E\Big[\exp\big(\log R_n-\mathbb E[\log R_n]\big)\Big]
=
\mathbb E[R_n]\exp\big(-\mathbb E[\log R_n]\big)
\le
M(p).
\]
Applying the same argument to $1/R_n$ gives
$
\mathbb E\Big[\exp\big(-(\log R_n-\mathbb E[\log R_n])\big)\Big]\le M(p).
$
Hence, for every $\varepsilon>0$ and every $n\ge 1$,
\[
\mathbb P\Big(\big|\log R_n-\mathbb E[\log R_n]\big|\ge \varepsilon n\Big)
\le
2M(p)e^{-\varepsilon n}.
\]
The sum over $n$ is finite, so Borel--Cantelli gives
$
\frac{\log R_n-\mathbb E[\log R_n]}{n}\to 0
$
almost surely.

By Jensen's inequality,
$
\mathbb E[\log R_n]\le \log \mathbb E[R_n].
$
Applying Jensen's inequality to $-\log R_n$ gives
$
-\mathbb E[\log R_n]\le \log \mathbb E\Big[\frac{1}{R_n}\Big].
$
Therefore
\[
0\le \log \mathbb E[R_n]-\mathbb E[\log R_n]
\le
\log\Big(\mathbb E[R_n]\mathbb E\Big[\frac{1}{R_n}\Big]\Big)
\le
\log M(p).
\]
Dividing by $n$ and letting $n\to\infty$ yields
$
\frac{1}{n}\mathbb E[\log R_n]-\frac{1}{n}\log \mathbb E[R_n]\to 0.
$
Since
$
\frac{1}{n}\log \mathbb E[R_n]\to \alpha(p),
$
we also have
$
\frac{1}{n}\mathbb E[\log R_n]\to \alpha(p).
$
Combining this with the almost sure convergence of
$
\frac{\log R_n-\mathbb E[\log R_n]}{n}
$
proves that
\[
\frac{1}{n}\log R_n\to \alpha(p)
\]
almost surely.
\end{proof}

\begin{proposition}\label{prop:alpha_pc_positive}
At the critical parameter \(p=p_c\),
\[
\alpha(p_c)\ge \log\frac{14-4\sqrt5}{3}>0.
\]
\end{proposition}

\begin{proof}
At \(p=p_c\), we have \(p_{\mathrm{diamond}}=\sqrt5-2<\frac{1}{2}\).
By Case 1 in the proof of Lemma~\ref{lem:uniform_product},
$
\frac{\mathbb E[R_n^2]}{\mathbb E[R_n]^2}\le 2
$
for all \(n\ge 0\).

Let \(D_n\) denote the effective resistance in the diamond case at level \(n+1\).
The lower bound obtained there gives
$
\mathbb E[D_n]
\ge
\frac{2\mathbb E[R_n]}{\frac{\mathbb E[R_n^2]}{\mathbb E[R_n]^2}+1}
\ge
\frac{2}{3}\mathbb E[R_n].
$
Therefore
$
\mathbb E[R_{n+1}]
=
2(1-p_{\mathrm{diamond}})\mathbb E[R_n]
+
p_{\mathrm{diamond}}\mathbb E[D_n]
\ge
\left(2-\frac{4}{3}p_{\mathrm{diamond}}\right)\mathbb E[R_n].
$
Substituting \(p_{\mathrm{diamond}}=\sqrt5-2\) and iterating, we obtain
$
\mathbb E[R_n]\ge \left(\frac{14-4\sqrt5}{3}\right)^n.
$
Hence, by the definition of \(\alpha(p_c)\),
$
\displaystyle
\alpha(p_c)
=
\lim_{n\to\infty}\frac{1}{n}\log \mathbb E[R_n]
\ge
\log\frac{14-4\sqrt5}{3}
>
0.
$
\end{proof}

Moreover, the Hambly--Kumagai first open problem also asks ``If there is exponential growth, is there a limit distribution such that $R_n \lambda^n \to c$ as $n \to \infty$".
We now rule out convergence of $R_n \lambda^n$ to a non-zero deterministic constant, giving a partial negative answer to this part.

\begin{proposition}\label{prop:no_deterministic_limit}
Assume that \(p_{\mathrm{series}},p_{\mathrm{diamond}}\in(0,1)\).
For \(\lambda>0\), set \(Z_n=\lambda^n R_n\).
If \(Z_n\) converges in distribution to a deterministic constant \(c\ge 0\), then \(c=0\).
\end{proposition}

\begin{proof}
Assume for contradiction that \(c>0\).
Since the limit is deterministic, \(Z_n\to c\) in probability.
Let \(Z_n^{(1)},\dots,Z_n^{(4)}\) be independent copies of \(Z_n\).
Then
\[
\mathbb P\Big(\max_{1\le i\le 4}|Z_n^{(i)}-c|>\varepsilon\Big)\le 4\mathbb P(|Z_n-c|>\varepsilon)\to 0,
\]
so \((Z_n^{(1)},\dots,Z_n^{(4)})\to (c,c,c,c)\) in probability.
Hence
$
\lambda(Z_n^{(1)}+Z_n^{(2)})\to 2\lambda c
$
and
\[
\lambda\frac{(Z_n^{(1)}+Z_n^{(2)})(Z_n^{(3)}+Z_n^{(4)})}{Z_n^{(1)}+Z_n^{(2)}+Z_n^{(3)}+Z_n^{(4)}}\to \lambda c
\]
in probability.
By the one-step recursion,
\[
Z_{n+1}\stackrel d=
\begin{cases}
\lambda(Z_n^{(1)}+Z_n^{(2)}), & \text{with probability } p_{\mathrm{series}},\\
\lambda\frac{(Z_n^{(1)}+Z_n^{(2)})(Z_n^{(3)}+Z_n^{(4)})}{Z_n^{(1)}+Z_n^{(2)}+Z_n^{(3)}+Z_n^{(4)}}, & \text{with probability } p_{\mathrm{diamond}}.
\end{cases}
\]
Therefore \(Z_{n+1}\) converges in distribution to
$
p_{\mathrm{series}}\delta_{2\lambda c}+p_{\mathrm{diamond}}\delta_{\lambda c}.
$
On the other hand, \(Z_{n+1}\to c\) in distribution, so the limit law is \(\delta_c\).
Thus
$
\delta_c=p_{\mathrm{series}}\delta_{2\lambda c}+p_{\mathrm{diamond}}\delta_{\lambda c}.
$
Since \(p_{\mathrm{series}},p_{\mathrm{diamond}}>0\), this is impossible unless \(2\lambda c=\lambda c=c\), hence impossible for \(c>0\).
Therefore \(c=0\).
\end{proof}

\textbf{Conclusion.} 
Let $\alpha=\alpha(p_c)$ and $p=p_c$ at the critical parameter.
Proposition~\ref{prop:dhl_reduction} and Theorem~\ref{thm:reduced_quenched_exponent} indicate that, almost surely, \[ \lim_{n\to\infty}\frac{1}{n}\log \Reff_{\langle \Xi_1^n \rangle_1}(\mathfrak v_+,\mathfrak v_-) = \lim_{n\to\infty}\frac{1}{n}\log \mathbb E\big[\Reff_{\langle \Xi_1^n \rangle_1}(\mathfrak v_+,\mathfrak v_-)\big] = \alpha. \] 
In particular, the quenched exponent and the annealed exponent for $\Reff_{\langle \Xi_1^n \rangle_1}(\mathfrak v_+,\mathfrak v_-)$ coincide. We estimated $\alpha$ numerically using a log-domain population-dynamics simulation in Python with $p=(\sqrt 5-1)/2$. 
At $n=10000$, both $\frac{1}{n}\mathbb E[\log \Reff_{\langle \Xi_1^n \rangle_1}(\mathfrak v_+,\mathfrak v_-)]$ and $\frac{1}{n}\log \mathbb E[\Reff_{\langle \Xi_1^n \rangle_1}(\mathfrak v_+,\mathfrak v_-)]$ stabilise near $0.5631$, and the observed gap is about $7.3\times 10^{-6}$.
Overall, we obtain \[ \alpha\approx 0.5631 \] so $\exp(\alpha)\approx 1.7561$, meaning that the typical growth scale of the two-terminal resistance is $(1.7561\cdots)^n$. 
This is consistent with the deterministic bounds $1\le \Reff_{\langle \Xi_1^n \rangle_1}(\mathfrak v_+,\mathfrak v_-)\le 2^n$. 


\subsection{Outlook: Brownian spectral dimensions for the critical percolation cluster of the DHL}\label{subsec:dhl_cluster}

We now return to the full four-colour coding $\Xi_1^n$.
In Subsection~\ref{subsec:Quenched unit-resistance exponent}, only the reduced monochromatic process $G^n$ was needed in order to study the two-terminal effective resistance.
To discuss the geometric, diffusion and Brownian exponents of the critical cluster, however, we must go back to the full coloured branching structure.

The four-colour random EIGS in Figure~\ref{fig:random EIGS DHL} should be viewed as a finite-state coding of the critical percolation cluster on the diamond hierarchical lattice, rather than as the cluster itself.
At the critical parameter
\[
p=p_c=\frac{\sqrt5-1}{2},
\qquad
1-p=p^2,
\]
we start from the single red edge $\Xi_1^0$ joining the two boundary vertices.
At each substitution step, the colour records the two-terminal state of the corresponding cell.
Red means that the two terminals are connected inside the cell.
Blue means that they are disconnected, but both terminals remain active.
Green means that they are disconnected, but exactly one terminal remains active.
Black is a cemetery bookkeeping state.
It is retained only in order to keep the coloured substitution rule closed.
Once a cell turns black, it no longer contributes to the growth of the critical cluster, and we therefore assign zero offspring and zero boundary growth to this type in the mean matrices below.

For each $n\ge 0$, let $\Xi_1^n$ be the full coloured graph after $n$ substitution steps, and let $\langle \Xi_1^n\rangle_1$ be its colour-$1$ subgraph.
Then $\langle \Xi_1^n\rangle_1$ has the same law as the level-$n$ critical two-terminal cluster on the DHL in the sense of \cite{hambly2010diffusion}.
We write
\[
\langle \Xi_1\rangle_1:=\bigcup_{n\ge 0}\langle \Xi_1^n\rangle_1
\]
for the combinatorial limit of the red subgraphs, and $\mathfrak C$ for the corresponding Gromov--Hausdorff--Prokhorov scaling limit of the critical percolation cluster on the DHL.
We also write $\mu_{\mathfrak C}$ for the natural probability measure on $\mathfrak C$, $\mathscr V_{\mathfrak C}$ for the embedded finite-born points, and
\[
\mathscr V_{\mathfrak C}^{\infty}:=\mathfrak C\setminus \mathscr V_{\mathfrak C}
\]
for the points added by the metric completion.
Thus the red subgraph is the actual cluster, while the blue, green and black edges are auxiliary states introduced only to keep the recursion closed.

This coding allows the geometric observables of the critical cluster to be read from the coloured branching process.
The size of the cluster is determined by the number of red edges, and hence also by the number of vertices up to comparable order.
The degree of a finite-born vertex is obtained by counting only its red incidences in the final graph, although the recursive evolution of the boundary data must still keep track of the blue and green states, since they may later produce red descendants.
In the present DHL model these coloured counts form a finite-type branching process, so the relevant exponential growth constants are governed by the Perron roots of the corresponding mean matrices.
Thus the mass growth is encoded by the mean offspring matrix $\mathbb E(\mathbf M)$ of the live colours, while degree growth is encoded by the corresponding mean boundary matrix $\mathbb E(\mathbf N)$.
For notational convenience we keep the black type in the indexing, but only as a zero placeholder.
Equivalently, the relevant information is contained in the upper-left live red-blue-green blocks.

The four-colour system is reducible, since colour $4$ is a cemetery state.
Accordingly, the quantities relevant to the critical cluster are not the full spectral radii of arbitrary four-type matrices, but the eigenvalues governing the growth of the colour-$1$ component.
With the present convention, these are simply the Perron roots of the live red-blue-green blocks, which we still denote by
\[
\lambda\bigl(\mathbb E(\mathbf M)\bigr)
\qquad
\text{and}
\qquad
\lambda\bigl(\mathbb E(\mathbf N)\bigr).
\]

With the colour order
\[
(1,2,3,4)=(\mathrm{red},\mathrm{blue},\mathrm{green},\mathrm{black}),
\]
the relevant offspring matrix is
\[
\mathbb E(\mathbf M)
=
\begin{pmatrix}
8p^2 & 4p^4 & 4p^5 & 0\\
4p^2 & 4p^2 & 4p^3 & 0\\
2p^2 & 0 & 2 & 0\\
0 & 0 & 0 & 0
\end{pmatrix}.
\]

For $\mathbb E(\mathbf N)$, we keep the same convention and write
\[
\mathbb E(\mathbf N)
=
\begin{pmatrix}
4p^2&0&2p^4&0&2p^5&0&0&0\\
0&4p^2&0&2p^4&0&2p^5&0&0\\
2p^2&0&2p^2&0&2p^3&0&0&0\\
0&2p^2&0&2p^2&0&2p^3&0&0\\
2p&0&0&0&2p^2&0&0&0\\
0&0&0&0&0&2p&0&0\\
0&0&0&0&0&0&0&0\\
0&0&0&0&0&0&0&0
\end{pmatrix}.
\]

Evaluated at $p=p_c$, the relevant live-block eigenvalues are
\[
\lambda\bigl(\mathbb E(\mathbf M)\bigr)\approx 3.7303,
\qquad
\lambda\bigl(\mathbb E(\mathbf N)\bigr)=2.
\]
Since the underlying DHL distance scale is $2$, by \cite{li2024on,neroli2024fractal,li2025reducible} this suggests the  values
\[
\dimbox\big(\langle \Xi_1\rangle_1\big)
=
\frac{\log \lambda\bigl(\mathbb E(\mathbf M)\bigr)}{\log 2}
\approx 1.8993,
\qquad
\dimdeg\big(\langle \Xi_1\rangle_1\big)
=
\frac{\log \lambda\bigl(\mathbb E(\mathbf M)\bigr)}{\log \lambda\bigl(\mathbb E(\mathbf N)\bigr)}
\approx 1.8993.
\]

By the conclusion of Subsection~\ref{subsec:Quenched unit-resistance exponent}, the quenched resistance exponent exists almost surely.
Accordingly, the same formal scaling suggests
\[
\dimresis\big(\langle \Xi_1\rangle_1\big)
=
\frac{\alpha}{\log 2}
\approx 0.8124,
\]
and the conjectural walk dimension
\[
\dimwalk\big(\langle \Xi_1\rangle_1\big)
=
\dimbox\big(\langle \Xi_1\rangle_1\big)+\dimresis\big(\langle \Xi_1\rangle_1\big)
\approx 2.7117.
\]

Taken together, these values suggest that the critical cluster should lie in the Brownian regime of Remark~\ref{rem:three_resistance_regimes}.
In particular, unlike the deterministic DHL itself, which lies in the Hambly--Kumagai regime $\dimresis=0$, we conjecture that the critical cluster satisfies
\[
\dimresis\big(\langle \Xi_1\rangle_1\big)>0.
\]
Accordingly, we further conjecture that the renormalised resistance defines a genuine resistance form on the limiting random cluster, and that the natural metric and the resistance metric induce the same topology.
From the point of view of Section~\ref{sec:diffusion}, this is exactly the regime in which the natural-metric diffusion and the resistance-form Brownian motion should describe the same limiting process.

At this point it is important to emphasise the gap between the rigorous theory developed in Section~\ref{sec:diffusion} and the present random setting.
The issue is not only that the four-colour system is reducible.
More importantly, Section~\ref{sec:diffusion} concerns deterministic graph and metric-space environments, together with the associated simple random walks and Brownian motions on those fixed spaces.
Here, by contrast, both the discrete critical clusters and the limiting space $\mathfrak C$ are random, so the relevant simple random walk and Brownian motion live in a random environment.
A full extension of Section~\ref{sec:diffusion} to the present model would therefore require a quenched reducible random EIGS theory, including almost-sure control of the geometric exponents, a passage from the random discrete clusters to the random Gromov--Hausdorff--Prokhorov limit, and a corresponding transfer from random-walk exponents to diffusion and Brownian spectral dimensions on the limiting random space.
We do not develop that theory here.

Nevertheless, the branching-process structure strongly suggests that the exponent relations from Subsection~\ref{subsec:heat_kernel_diffusion} should continue to hold almost surely for the colour-$1$ cluster.
This leads us to formulate the following open conjecture.
We write $\dimspec^{(\mathrm N)}(\mathfrak C:x)$ and $\dimspec^{(\mathrm R)}(\mathfrak C:x)$ for the conjectural diffusion and Brownian spectral dimensions at $x$.
Since $\dimresis\big(\langle \Xi_1\rangle_1\big)>0$, we conjecture that these two quantities agree almost surely.
More precisely, we conjecture that finite-born points in the limit carry the same degree correction as in the scale-free EIGS setting, whereas $\mu_{\mathfrak C}$-almost every point of $\mathscr V_{\mathfrak C}^{\infty}$ does not.

Accordingly, we conjecture that, almost surely,
\[
\dimspec^{(\mathrm N)}(\mathfrak C:x)
=
\dimspec^{(\mathrm R)}(\mathfrak C:x)
=
\begin{cases}
\displaystyle
2\frac{\dimbox\big(\langle \Xi_1\rangle_1\big)\bigl(1-\frac{1}{\dimdeg\big(\langle \Xi_1\rangle_1\big)}\bigr)}{\dimwalk\big(\langle \Xi_1\rangle_1\big)}
\approx 0.6633,
\quad
x\in \mathscr V_{\mathfrak C},
\\[3mm]
\displaystyle
2\frac{\dimbox\big(\langle \Xi_1\rangle_1\big)}{\dimwalk\big(\langle \Xi_1\rangle_1\big)}
\approx 1.4008,
\quad\qquad
x\in \mathscr V_{\mathfrak C}^{\infty}
\quad
(\mu_{\mathfrak C}\text{-a.e.}).
\end{cases}
\]

We look forward to a systematic theory of effective resistance and diffusion limits for general random EIGS.
Such a theory would upgrade the conjectures of this subsection into theorems, and would thereby establish the failure of the Alexander--Orbach-type conjecture for the critical percolation cluster on the DHL and even a broad family of EIGS.
Though building this theory will be far from straightforward, for the moment the heuristics already tell the story: the cluster should carry two spectral dimensions, $0.6633$ at finite-born points and $1.4008$ at typical points, neither of which equals the Alexander--Orbach value $\frac43$.
In this sense, we have given a partial answer to the question with which the paper began.

\section*{Acknowledgement}
The author is grateful for helpful comments on an earlier version of this manuscript.
The Lean formalisation was produced with Aristotle by Harmonic; the author thanks the Harmonic team for making Aristotle available.
This work was supported by the Additional Funding Programme for Mathematical Sciences, delivered by EPSRC (EP/V521917/1) and the Heilbronn Institute for Mathematical Research,
and also by the EPSRC Centre for Doctoral Training in Mathematics of Random Systems: Analysis, Modelling and Simulation (EP/S023925/1).

\bibliographystyle{amsplain}
\bibliography{references}

\end{document}